\documentclass[11pt]{article}

\usepackage{setspace}
\usepackage[english]{babel}
\usepackage[utf8]{inputenc}
\usepackage{amsmath}
\usepackage{scalerel}
\usepackage{amssymb}
\usepackage{graphicx}
\usepackage{textcomp}
\usepackage{fancyhdr}
\usepackage{amsthm}
\usepackage{amssymb}
\usepackage{graphicx}
\usepackage{appendix}
\usepackage{indentfirst}
\usepackage{mathtools}
\mathtoolsset{showonlyrefs,showmanualtags}
\usepackage{bm}
\usepackage[maxbibnames=99,doi=false,isbn=false,url=false]{biblatex}
\renewbibmacro{in:}{}
\usepackage{csquotes}
\usepackage{enumerate}
\usepackage{xcolor}
\usepackage{subcaption} 
\usepackage{hyperref}
\usepackage[margin=1in]{geometry}
\allowdisplaybreaks

\numberwithin{equation}{section}
\newcommand{\tz}{{\Tilde{Z}}}
\newcommand{\tbh}{{\Tilde{\bm{H}}}}
\newcommand{\tbq}{{\Tilde{\bm{Q}}}}
\newcommand{\tbx}{{\Tilde{\bm{X}}}}
\newcommand{\tbxi}{{\Tilde{\bm{\xi}}}}
\newcommand{\bx}{{\bm{X}}}
\newcommand{\bxi}{{\bm{\xi}}}

\newcommand{\tx}{{\Tilde{X}}}

\newcommand{\txi}{{\Tilde{\xi}}}

\newcommand{\Prob}{\bm{\mathrm{P}}}
\newcommand{\Expe}{\bm{\mathrm{E}}}
\newcommand{\st}[1]{#1^{\e*}}
\newcommand{\T}{\Tilde}

\newcommand{\e}{\varepsilon}

\usepackage[active]{srcltx}

\newtheorem{theorem}{Theorem}[section]
\newtheorem{corollary}[theorem]{Corollary}

\newtheorem{lemma}[theorem]{Lemma}
\newtheorem{proposition}[theorem]{Proposition}
\newtheorem{definition}[theorem]{Definition}

\newtheorem{remark}[theorem]{Remark}

\newcommand{\thmref}[1]{Theorem~\ref{#1}}

\newcommand{\lemref}[1]{Lemma~\ref{#1}}

\newcommand{\E}{\bm{\mathrm E}}

\hypersetup{
    colorlinks=true,
    linkcolor=black,
    citecolor=black,
    pdftitle={Fast-oscillating random perturbations of Hamiltonian systems},
    }

\title{Fast-oscillating random perturbations of Hamiltonian systems}
\author{Shuo Yan\\
Department of Mathematics, University of Maryland\\ 
College Park, Maryland, United States\\
shuoyan@umd.edu}
\date{}

\begin{document}
\maketitle
\begin{abstract}
    We consider coupled slow-fast stochastic processes, where the averaged slow motion is given by a two-dimensional Hamiltonian system with multiple critical points. On a proper time scale, the evolution of the first integral converges to a diffusion process on the corresponding Reeb graph, with certain gluing conditions specified at the interior vertices, as in the case of additive white noise perturbations of Hamiltonian systems considered by M. Freidlin and A. Wentzell.
    The current paper provides the first result where the motion on a graph and the corresponding gluing conditions appear due to the averaging of a slow-fast system{, with a Hamiltonian structure, on a large time scale}.
    The result allows one to consider, for instance, long-time diffusion approximation for an oscillator with a potential with more than one well.
\end{abstract}
\section{Introduction}
\label{sec:introduction}
Consider a family of diffusion processes $(\bx_t^\e,\bxi_t^\e)$ in $\mathbb R^2\times\mathbb T^m$ satisfying 
\begin{equation}
\label{eq:theprocess1}
\begin{aligned}
    d\bm X_t^\e=&~b(\bx_t^\e,\bxi_t^\e)dt,~~~~~~~ \bx_0^\e=x_0\in\mathbb R^2,\\
    d\bxi_t^\e=&~\frac{1}{\e}v(\bxi_t^\e)dt+\frac{1}{\sqrt{\e}}\sigma(\bxi_t^\e)dW_t,~~~\bxi_0^\e=y_0\in\mathbb T^m,
\end{aligned}
\end{equation}
where $\e$ is a small positive parameter, $\mathbb T^m$ is the $m$-dimensional torus, and $W_t$ is an $m$-dimensional Brownian motion. 
In the coupled slow-fast system, $\bx_t^\e$ is the slow component and $\bxi_t^\e$ is the fast component, since the generator of the diffusion in the second equation is multiplied by $\e^{-1}$.
On the space $\mathbb R^2\times\mathbb T^m$, the diffusion \eqref{eq:theprocess1} is everywhere degenerate. All the randomness comes from the second equation and is transmitted to $\bx_t^\e$ through the vector $b(\bx_t^\e,\bxi_t^\e)$, which is fast-oscillating in time.
Under natural conditions, the averaging principle holds for the process in \eqref{eq:theprocess1} 
(cf. \cite{randomperturbation}). For example, if $\sigma(y)$ is non-degenerate (and thus 
$\bxi_t^\e$ has a unique invariant measure $\mu$ independent of $\e$), then $\bx_t^\e$ converges as $\e\to0$ in probability on each finite interval $[0,T]$ to an averaged process defined by the differential equation
\begin{equation}
\label{eq:averagedprocess}
    d{\bm{x}_t}=\bar b({\bm{x}_t})dt,
\end{equation}
where $\bar b(x)=\int_{\mathbb T^m}b(x,y)d\mu(y)$. 
Therefore, $\bx_t^\e$ can be viewed as a result of fast-oscillating random perturbations{, i.e. $b(x,y)-\bar b(x)$,} of the deterministic process ${\bm{x}_t}$. 
Moreover, the deviation can be described more precisely: the process $\e^{-1/2}(\bx_t^\e-{\bm{x}_t})$ converges weakly to a Gaussian Markov process on a finite interval $[0,T]$ {(cf. \cite{MR0203789}, \cite{MR0517995})}, and, if we assume a special type of vector $b(x,y)$, then the local limit theorem holds for $\e^{-1}(\bx_t^\e-{\bm{x}_t})$ at time $t$ (\cite{LLT}).

If the system \eqref{eq:averagedprocess} has a first integral $H$, then, by the averaging principle, $H(\bx_t^\e)$ is nearly constant on finite time intervals when $\e$ is small. Nontrivial behavior can, however, be observed on larger time intervals (of order $\e^{-1}$). Assume, momentarily, that $H$ has a single critical point. Then it was demonstrated in \cite{AIHPB_1995__31_3_485_0} that $H(\bx_{t/\e}^\e)$ converges weakly in $C([0,T])$, as $\e\to0$, to a diffusion process for any finite $T$, under additional assumptions. 
A similar result in the case of multiple degrees of freedom was obtained recently in \cite{Freidlin2021} and the main goal there was to overcome difficulties related to resonances, which is typical in the case of multiple degrees of freedom.
The result holds in the region where no critical points of the first integrals are present and action-angle-type coordinates can be introduced. 
\begin{figure}[!ht] 
    \centering
    \includegraphics[width=0.75\textwidth]{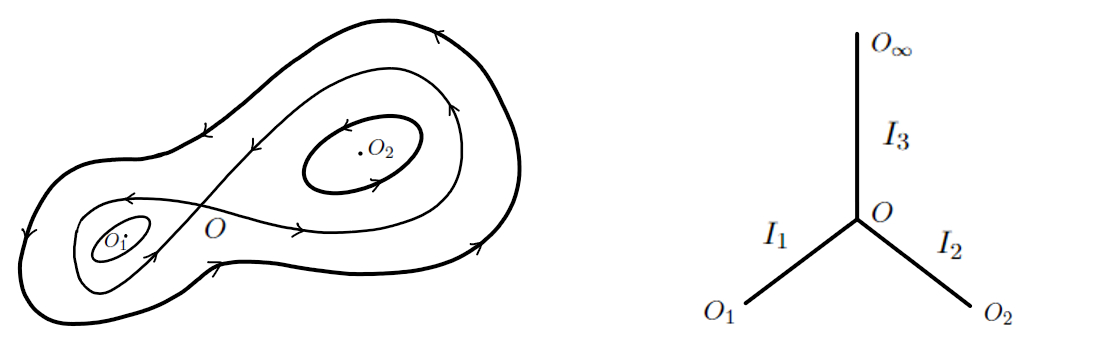}
    \caption{Level sets and the Reeb graph.}
    \label{fig:reeb_graph}
    \end{figure}
    
Let us return to the two-dimensional situation. In the presence of multiple critical points, including saddle points, the problem gets more complicated as we need to consider the Reeb graph in order to describe the evolution of the first integrals denoted by $h=(k,H)${, where the additional discrete-valued first integral $k$ is the label of the edge on the Reeb graph. (For instance, in Figure~\ref{fig:reeb_graph}, we have one saddle point and two local minima of $H$ in the space $\mathbb R^2$. Accordingly, on the graph, we have one interior vertex, three exterior vertices, including one formally representing the infinity, and three edges connecting them.)} In particular, the interior vertices on the graph correspond to the level curves that contain the saddle points, and those level curves are called the separatrices. 
In this situation, the limiting behavior has already been described for the white-noise-type additive perturbations of dynamical systems: Hamiltonian systems in $\mathbb R^2$ (\cite{FreidlinWentzell1994}), general dynamical systems with conservation laws in $\mathbb R^n$ (\cite{Freidlin2004}), and Hamiltonian systems with an ergodic component on two-dimensional surfaces (\cite{Dmitry2008},\cite{Dmitry2013},\cite{dolgopyat_freidlin_koralov_2012}). 

In this article, we consider fast-oscillating random perturbations, as discussed above, of Hamiltonian system in $\mathbb R^2$ with multiple critical points and prove that the evolution of the first integrals $h$ converges to a diffusion process defined by an operator ($\mathcal L, D(\mathcal L))$ on the corresponding Reeb graph. In particular, the exterior vertices turn out to be inaccessible and the behavior of the process near the interior vertices is described in terms of the domain $D(\mathcal L)$ in the following way: for interior vertex $O_i$, there are constants $p_k$ such that each function $f\in D(\mathcal L)$ satisfies
\begin{equation}
\label{eq:gluing_condtion}
\sum_{I_k\sim O_i}p_k\lim_{h_k\to O_i}f'(h_k)=0,
\end{equation}{where $I_k\sim O_i$ means that $O_i$ is an endpoint of $I_k$.}
Intuitively, the absolute value of $p_k$ is proportional to the probability of entering edge $I_k$ after the process arrives at the vertex $O_i$. The relation \eqref{eq:gluing_condtion} is usually referred to as the gluing condition. 
In the next section, we will formulate the results along with the assumptions more precisely. The coefficients $p_k$ will be calculated explicitly. As we mentioned, similar results hold in case of additive perturbations of Hamiltionian systems. Now the techniques in the proof are more involved and require new ideas with analysis on multiple time scales : $O(\e^{-1})$, $O(1)$, $O(\e)$, etc.
It is worth noting that our result provides the first example where the motion on a graph and the corresponding gluing conditions appear as a result of averaging of a slow-fast system{, with a Hamiltonian structure, on a large time scale}.

{In the remainder of this section, we briefly introduce the main idea and the critical steps of the proof, and outline the plan of the paper.}
To start with, since our interest is in the long-time behavior of $\bx_t^\e$ on $O(\e^{-1})$ time scales, it is often convenient to consider a temporally re-scaled process $(X_t^\e,\xi_t^\e)$:
\begin{equation}
\label{eq:rescaled_process1}
\begin{aligned}
    d X_t^\e=&~\frac{1}{\e}b(X_t^\e,\xi_t^\e)dt,~~~~~~~ X_0^\e=x_0\in\mathbb R^2,\\
    d\xi_t^\e=&~\frac{1}{\e^2}v(\xi_t^\e)dt+\frac{1}{{\e}}\sigma(\xi_t^\e)dW_t,~~~\xi_0^\e=y_0\in \mathbb T^m.
\end{aligned}
\end{equation}
It is clear that $(X_t^\e,\xi_t^\e)=(\bx_{t/\e}^\e,\bxi_{t/\e}^\e)$ in distribution.
Thus, it suffices to prove the weak convergence of $h(X_t^\e)$ in the space $\bm{\mathrm{C}}([0,T],\mathbb G)$, where $\mathbb G$ is the Reeb graph.
The proof of the weak convergence relies on demonstrating that the pre-limiting process asymptotically solves the martingale problem. Namely, we will show that, for each $f$ in a sufficiently large subset of $D(\mathcal L)$ and $T>0$,
\begin{equation}
\label{eq:martingale_problem}
    \E_{(x,y)}[f(h(X_{T}^\e))-f(h(x))-\int_0^T\mathcal Lf(h(X_{t}^\e))dt]\to 0, 
\end{equation}
as $\e\to 0$, uniformly in $x$ in any compact set in $\mathbb R^2$ and in $y\in\mathbb T^m$.
{Note that, contrary to the standard formulation of the martingale problem, there is no conditioning in \eqref{eq:martingale_problem}. However, \eqref{eq:martingale_problem} is still enough for our purpose (see Lemma~\ref{lem:martingale_problem}), since $(X_t^\e,\xi_t^\e)_{\e>0}$ is a family of strong Markov processes.}
The main idea in our proof of \eqref{eq:martingale_problem} is to divide the time interval $[0,T]$ into excursions between different visits to the separatrices and show that the contribution from each excursion is small and they do not accumulate. 
For example, suppose for now that there is only one saddle point, as shown in Figure~\ref{fig:markov_chain}. 
Let $O$ be the saddle point with $H(O)=0$, $\gamma$ be the separatrix, $\gamma'=\{x:|H(x)|=\e^\alpha\}$ be a set near the separatrix, where $0<\alpha<1/2$, and $\sigma\geq0$ be the first time when the process $X_{t}^\e$ reaches $\gamma$. Define inductively the two sequences of stopping times:
\begin{equation}
\label{eq:stopping_times}
   \begin{aligned}
    \sigma_0=&~\sigma,\\
    \tau_n=&\inf\{t>\sigma_{n-1}:X_t^\e\in\gamma'\},\\
    \sigma_n=&\inf\{t>\tau_{n}:X_t^\e\in\gamma\},
    \end{aligned}
\end{equation}
and consequently two Markov chains $(X_{\tau_n}^\e,\xi_{\tau_n}^\e)$ and $(X_{\sigma_n}^\e,\xi_{\sigma_n}^\e)$.
\begin{figure}[!ht] 
    \centering
    \includegraphics[width=0.5\textwidth]{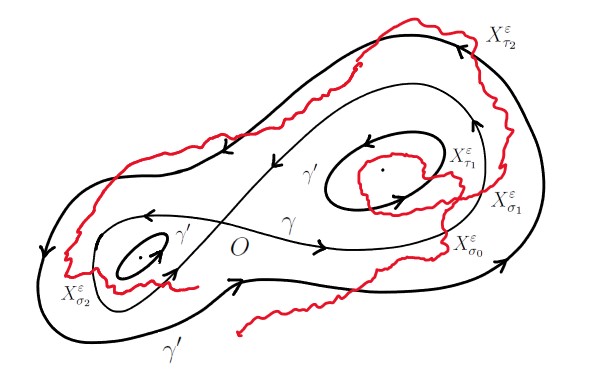}
    \caption{Construction of discrete Markov chains.}
    \label{fig:markov_chain}
    \end{figure}
As pointed out earlier, we wish to prove that the contributions to \eqref{eq:martingale_problem} from all individual excursions are small and the sum converges to zero as $\e\downarrow0$. 
{Thus, except for the first and the last excursions, it is sufficient to show that 
(a) the expectation corresponding to one excursion $[\sigma_n,\sigma_{n+1}]$ converges to zero as $\e\downarrow0$ uniformly in initial distribution,
(b) the expectation corresponding to one excursion $[\sigma_n,\sigma_{n+1}]$ is exactly zero, for all $\e$, if the process starts with the invariant measure of the Markov chain $(X_{\sigma_n}^\e,\xi_{\sigma_n}^\e)$ on $\gamma\times\mathbb T^m$, and (c) the measures on $\gamma\times\mathbb T^m$ induced by $(X_{\sigma_n}^\e,\xi_{\sigma_n}^\e)$ converge exponentially, as $n\to\infty$, uniformly in $\e$ and in initial distribution, to the invariant one. }

{The claim in (a) is an extension of the results outside of the singularities in \cite{Freidlin2021}, and new difficulties arise due to the degenerations occurring on the boundaries.
The claim in (b) is true if there is a common invariant measure for the processes $(X_t^\e,\xi_t^\e)$ for all $\varepsilon$ and the gluing conditions are chosen appropriately. In general, there is no common invariant measure for all $\e$, and we need to consider a family of auxiliary processes that do have a common invariant measure, and then use the proximity of the auxiliary and the original processes near the separatrix, and the Girsanov theorem to show that the gluing conditions are actually the same. 
The assertion in (c) is hard to verify, and its proof requires new techniques, including a local limit theorem for time-inhomogeneous functions of Markov processes and density estimates for hypoelliptic diffusions that will be discussed in later sections.}

{\textbf{\textit{Plan of the paper.}}} 

{In Section~\ref{sec:mainresult}, we introduce the notation, state the assumptions, formulate the main result and the lemma we use to establish the weak convergence. 
In Section~\ref{sec:preliminaries}, the problem is reduced to the case we discussed where there is only one saddle point. Besides, we construct an auxiliary process with $\e$-independent invariant measure and derive diffusion approximations of the processes. 
In Section~\ref{sec:averaging}, we prove the averaging principle for the process on the Reeb graph up to the time when the process reaches an interior vertex. 
In Section~\ref{sec:exponenitalconvergence}, we construct the Markov chain on the product space of the separatrix and the $m$-torus (see \eqref{eq:stopping_times}) and prove its convergence to the invariant measure with an exponential rate uniformly in $\e$. 
In Section~\ref{sec:proofofthemainresult}, we prove the main result. 
A few technical results including time estimates near the vertices and tightness of the processes are included in the Appendix.}
\section{Main result}
\label{sec:mainresult}
Throughout this article, $\Prob$ and $\E$ represent the probability and expectation, respectively, and the  subscripts pertain to initial conditions. 
{$\mathcal F_t^{X^\e_\cdot}$ denotes the natural filtration generated by the process $X_t^\e$.}
For brevity, the stopping times' dependence on parameters and initial conditions is not always indicated in the notation when introduced (e.g. \eqref{eq:stopping_times}). $\nabla$ denotes a first order differential operator, i.e., derivative, gradient, Jacobian, etc., depending on the context. 
$\chi_A$ denotes the indicator function of the event $A$.
If $A$ and $B$ are two non-negative functions that depend on an asymptotic parameter, we write $A\lesssim B$ if $A=O(B)$.
$\bm{\mathrm C}_0(\mathbb G)$ is the space of continuous functions on the Reeb graph $\mathbb G$ that tend to zero at infinity with uniform norm.
$h$ is the projection onto $\mathbb G$.
In order to formulate the assumptions and results, we {further introduce the following notation.}
\\

{\textbf{\textit{Notation.}}}
\begin{enumerate}[(i)]
    \item[\textit{(i)}] $O_i$'s are the vertices on the graph and are occasionally used to denote the corresponding critical points on the plane when there is no ambiguity. $I_k$'s are the edges on the graph and $U_k$'s are the corresponding two-dimensional domains. Formally, $O_\infty$ is the vertex that corresponds to infinity. 
    A symbol $\sim$ between a vertex and an edge means that the vertex is an endpoint of the edge. 
    \item[\textit{(ii)}] Consider the following metric on $\mathbb G$: $r(h_1,h_2)$ is the length of the shortest path connecting $h_1$ and $h_2$. For example, if $h_1=(1,H_1)$, $I_1\sim O_1$, $O_1\sim I_2$, $I_2\sim O_2$, $O_2\sim I_3$ and $h_2=(3,H_2)$, then $r(h_1,h_2)=|H_1-H(O_1)|+|H(O_1)-H(O_2)|+|H(O_2)-H_2|$. 
    \item[\textit{(iii)}] $\gamma(h)=\{x:H(x)=h\}$ and $\gamma_k(h)$ is the connected component of $\gamma(h)$ in the domain $U_k$.
    \item[\textit{(iv)}] $b_h(x,y)=\nabla H(x)\cdot b(x,y)$. 
    \item[\textit{(v)}] $\xi_t$ is the diffusion process on $\mathbb T^m$ with the generator $L$, where
    \begin{equation}
        \label{eqb:def_operator_L}
    L f(y)=v(y)\cdot\nabla_y f(y)+\frac{1}{2}\sum_{i,j}(\sigma\sigma^*)_{ij}(y)\frac{\partial^2}{\partial y_i\partial y_j}f(y).
    \end{equation}
    \item[\textit{(vi)}] For $h$ in the interior of $I_k$, define
    \begin{equation}
    \label{def:llimit_operator}
    \begin{aligned}
            Q_k(h)&=\int_{\gamma_k(h)}\frac{dl}{|\nabla H(x)|},\\
            A_k(h)&=\frac{2}{Q_k(h)}\int_{\gamma_k(h)}\frac{1}{|\nabla H(x)|}\int_0^\infty\E_\mu b_h(x,\xi_s)b_h(x,\xi_0)dsdl,\\
            B_k(h)&=\frac{1}{Q_k(h)}\int_{\gamma_k(h)}\frac{1}{|\nabla H(x)|}\int_0^\infty\E_\mu\nabla_x b_h(x,\xi_s)\cdot(b(x,\xi_0)-\nabla^{\perp}H(x))dsdl,\\
            L_kf(h)&=\frac{1}{2}A_k(h)f''(h)+B_k(h)f'(h).
        \end{aligned}
        \end{equation}
\end{enumerate}

The following conditions are assumed to hold throughout the article.
\\

{\textbf{\textit{Assumptions.}}}
\begin{enumerate}
    \item[\hypertarget{H1}{\textit{(H1)}}] $v(y)$ and $\sigma(y)$ are $C^\infty$ functions on $\mathbb T^m$. $\sigma(y)$ is $m\times m$ matrix-valued and $\sigma(y)\sigma(y)^{\mathsf T}$ is positive-definite for all $y\in\mathbb T^m$.
    \item[\hypertarget{H2}{\textit{(H2)}}] $H(x)$ is a $C^\infty$ function from $\mathbb R^2$ to $\mathbb R$ with bounded second derivatives. $H(x)$ has a finite number of non-degenerate critical points. Each level curve corresponding to a vertex on the Reeb graph contains at most one critical point. As $|x|\to+\infty$, $H(x)/|x|\to+\infty$. 
    \item[\hypertarget{H3}{\textit{(H3)}}] $b(x,y)$ is a $C^\infty$ function from $\mathbb R^2\times\mathbb T^m$ to $\mathbb R^2$ such that the averaged process is a Hamiltonian system with $H$, i.e. $\bar b(x)=\nabla^\perp H(x)$. 
    \item[\hypertarget{H4}{\textit{(H4)}}] The fast-oscillating perturbation is non-degenerate, i.e. $\{b(x,y)-\bar b(x):y\in\mathbb T^m\}$ spans $\mathbb R^2$ for each $x\in\mathbb R^2$, and is uniformly bounded together with its first derivatives.
    \item[\hypertarget{H5}{\textit{(H5)}}] For each $x$ that belongs to one of the separatrices, there exists $y\in\mathbb T^m$ such that the process in \eqref{eq:theprocess1} satisfies the parabolic H\"ormander condition at $(x,y)$. Namely, with $\e^{-1}\tilde v(y)$ being the drift term in the equation for $\bxi^\e_t$ in the Stratonovich form, we have that
    \begin{equation}
        \mathrm{Lie}\left(
        \left\{\begin{pmatrix}
            0\\ \sigma_k(y)
        \end{pmatrix},1\leq k\leq m\right\}\bigcup
        \left\{\left[\begin{pmatrix}
            b(x,y)\\ \tilde v(y)
        \end{pmatrix},
        \begin{pmatrix}
            0\\ \sigma_k(y)
        \end{pmatrix}\right],1\leq k\leq m\right\}\right)
    \end{equation} at $(x,y)$ spans $\mathbb R^{2+m}$, where $\sigma_k(y)$ is the $k$-th column of $\sigma(y)$, $[\cdot,\cdot]$ is the Lie bracket, and Lie$(\cdot)$ is the Lie algebra generated by a set (cf. \cite{Hormander_Hairer} or Section 2.3.2 of \cite{Nualart}). 
\end{enumerate}

\begin{definition}
\label{def:domain_original}
    The domain $D(\mathcal L)$ consists of functions $f\in \bm{\mathrm{C}}_0(\mathbb G)$ satisfying:
    \begin{enumerate}[(i)]
        \item $f$ is twice continuously differentiable in the interior of each edge $I_k$ of $\mathbb G$;
        \item The limits $\lim_{h_k\to O_i}L_kf(h_k)$ exist and do not depend on the edge $I_k$;
        \item For interior vertex $O_i$, there are constants $p_k:=\pm\lim_{h\to O_i}A_k(h)Q_k(h)$ such that
        \begin{equation}
        \label{eq2:gluing_condition}
            \sum_{I_k\sim O_i}p_k\lim_{h_k\to O_i}f'(h_k)=0,
        \end{equation}
    \end{enumerate}
    where the sign $+$ is taken if $O_i$ is minimum on $I_k$, and the sign $-$ is taken otherwise.
    The operator $\mathcal L$ on the Reeb graph is defined by 
        \begin{equation}
            \label{def:operatorL}
            \mathcal Lf(h)=L_kf(h)
        \end{equation}
    for $f\in D(\mathcal L)$ and $h$ in the interior of $I_k$, and defined as $\lim_{h\to O_i}\mathcal Lf(h)$ at the vertex $O_i$.
\end{definition}
By the Hille-Yosida theorem (see, for example, Theorem 4.2.2 in \cite{markov_process}), one can check that there exists a unique strong Markov process on $\mathbb G$ with continuous sample paths that has $\mathcal L$ as its generator.
Now we are ready to formulate the main result of this article.
\begin{theorem}
\label{thm:mainresult}
    Let the process $(\bx_t^\e,\bxi_t^\e)$ be defined as in \eqref{eq:theprocess1} and the conditions \hyperlink{H1}{\textit{(H1)-(H5)}} hold.
    Then $h(\bx_{t/\e}^\e)$ converges weakly to the strong Markov process on the Reeb graph $\mathbb G$ that has the generator $(\mathcal L,D(\mathcal L))$ and the initial distribution $h(x_0)$.  
\end{theorem}
\begin{remark}
    The last condition in \hyperlink{H2}{\textit{(H2)}} can be relaxed without much extra effort since the limiting process defined by $\mathcal L$ cannot reach infinity in finite time.
    In addition, as seen from the proofs in Section~\ref{sec:exponenitalconvergence} and Remark~\ref{rmk:positive_any_subset_separatrix}, assumption \hyperlink{H5}{\textit{(H5)}} can be relaxed so that it holds for at least one point  on each separatrix. Moreover, if the number of Lie brackets needed to generate $\mathbb R^{2+m}$ in the parabolic H\"ormander condition is assumed to be given, then we can relax the assumptions on smoothness of the coefficients.
\end{remark}
To prove the theorem, we need a result on weak convergence of processes, that is Lemma 4.1 in \cite{FreidlinKoralov2021} adapted to our case (see also the original statement in \cite{FreidlinWentzell1994}):
\begin{lemma}
\label{lem:martingale_problem}
Let $\Psi$ be a dense linear subspace of $\bm{\mathrm{C}}_0(\mathbb G)$ and $\mathcal D_{\mathcal L}$ be a linear subspace of $D(\mathcal L)$, and suppose that $\Psi$ and $\mathcal D_{\mathcal L}$ have the following properties:
\begin{enumerate}[(1)]
    \item There is a $\lambda>0$ such that for {each $F\in\Psi$} the equation $\lambda f-\mathcal Lf=F$ has a solution $f\in\mathcal D_{\mathcal L}$;
    \item For each $T>0$, each $f\in\mathcal D_{\mathcal L}$, and each compact $K\subset\mathbb G$, 
    \begin{equation}
        \label{eq:mgprob}
        \E_{(x,y)}[f(h(X_{T}^\e))-f(h(x))-\int_0^T\mathcal Lf(h(X_{t}^\e))dt]\to 0,
    \end{equation}
\end{enumerate}
uniformly in $x\in h^{-1}(K)$ and $y\in\mathbb T^m$. 

Suppose that{, for each starting point $(x,y)$ of $(X_t^\e,\xi_t^\e)$,} the family of measures on $\bm{\mathrm{C}}([0,\infty), \mathbb G)$ induced by the processes $h(X_t^\e)$, $\e>0$, is tight. Then{, for each starting point $(x,y)$ of $(X_t^\e,\xi_t^\e)$,} $h(X_t^\e)$ converges weakly to the strong Markov process on the Reeb graph $\mathbb G$ that has the generator $(\mathcal L,D(\mathcal L))$ and the initial distribution $h(x)$.  
\end{lemma}
Here we choose $\Psi$ to be all the functions in $\bm{\mathrm{C}}_0(\mathbb G)$ that are twice continuously differentiable in the interior of each edge; $\mathcal D_{\mathcal L}$ to be all the functions in $D(\mathcal L)$ that are four times continuously differentiable in the interior of each edge. It is easy to check condition $(1)$ holds in Lemma~\ref{lem:martingale_problem}, and the tightness of distributions of $h(X_t^\e)$, $\e>0$, is verified in Appendix~\ref{sec:tightness}. Then the main ingredient of the proof is to verify \eqref{eq:mgprob} in condition (2) of Lemma~\ref{lem:martingale_problem}.

\section{Preliminaries}
\label{sec:preliminaries}
In this section, we explain some technical difficulties and our approach to the proof.

\subsection{Localization}
Considering \thmref{thm:mainresult} for $\bx_t^\e$ in the state space $\mathbb R^2$ causes technical difficulties due to the presence of multiple separatrices of the Hamiltonian and to the fact that the process $\bx_t^\e$ is not positive recurrent. 
However, such difficulties can be circumvented by considering the process $\bx_t^\e$ locally. 
Namely, let us cover the plane $\mathbb R^2$ by finitely many bounded domains, each containing one of the separatrices and bounded by up to three connected components of level sets of $H$, and one unbounded domain not containing any critical points. 
For example, as shown in Figure~\ref{fig:local_graph}, we have different parts of the Reeb graph $\mathbb G$ that correspond to the domains in $\mathbb R^2$.
Every point of $\mathbb R^2$ can be assumed to be contained in the interior either one or two domains. 
Since it takes positive time to travel from the boundary of one domain to the boundary of another, it suffices to prove the result up to time of exit from one domain. 
To be more precise, let $\{V_k:1\leq k \leq K\}$ be the open cover. 
Define $\eta_0=\inf\{t\geq0:X_t^\e\in\bigcup_{1\leq k \leq K}\partial V_k\}$ and, for $k$ such that $X_{\eta_{n-1}}^\e\in V_k$, define $\eta_n=\inf\{t>\eta_{n-1}: X_t^\e\not\in V_k\}$, $n\geq 1$. 
In order to prove \eqref{eq:mgprob}, it suffices to prove instead, uniformly in $x$ in any compact set in $\mathbb R^2$ and in $y\in\mathbb T^m$, that
\begin{equation}
\label{eq:martingale_problem_stopped}
    \E_{(x,y)}[f(h(X_{T\wedge\eta_1}^\e))-f(h(x))-\int_0^{T\wedge{\eta_1}}\mathcal Lf(h(X_{t}^\e))dt]\to 0,~~{\rm as}~\varepsilon \downarrow 0,
\end{equation}
since it also implies that $\Prob(\eta_n<T)\to0$ as $n\to\infty$, uniformly in all $\e$ sufficiently small.
In the unbounded domain without critical points, \eqref{eq:martingale_problem_stopped} can be obtained using the result in bounded domain together with the tightness of $h(X_t^\e)$.
It remains to consider the bounded domains.
Let $V$ be one of the bounded domains. As explained below, the process $X_t^\e$ in $V$ can be extended beyond the time when it reaches the boundary by embedding $V$ into a compact manifold $M$ with an area form and a Hamiltonian such that there are no other separatrices.
\begin{figure}[!ht] 
    \centering
    \includegraphics[width=0.4\textwidth]{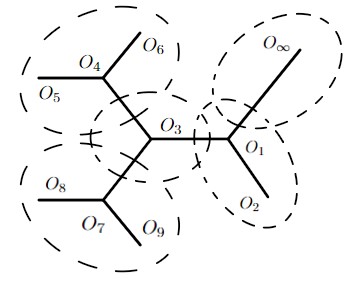}
    \caption{Domains projected on the graph.}
    \label{fig:local_graph}
    \end{figure}

Consider the case where $V$ is not simply connected - for example, $V$ is the domain that contains $O_3$ in Figure~\ref{fig:local_graph}. There are three connected components of $\mathbb R^2\setminus V$ as shown in Figure~\ref{fig:three_components} (other situations can be treated similarly). Then we can modify the Hamiltonian and the vector field in $C_1$ and $C_2$ in such  a way that assumptions \hyperlink{H2}{\textit{(H2)}-\textit{(H4)}} hold locally, and there is only one extremum point of $H$ in $C_1$ and one in $C_2$ (this modification is not needed if $V$ is simply connected). 
The unbounded domain $C_3$ outside $V$ can be replaced by a compact surface $S$ so that the resulting state space of $X_t^\e$ is, topologically, a sphere $M = V \bigcup C_1 \bigcup C_2 \bigcup S$. 
Then, the vector field on the surface can be chosen as a smooth extension from $V$ so that the averaged process is a Hamiltonian system on $M$ with respect to an area form $\omega$, which is simply $dx_1\wedge dx_2$ on $V$, $C_1$, and $C_2$. 
Moreover, there exists a chart $(S,\Phi)$ such that the corresponding vector field $b(x,y)$ on $D:=\Phi(S)$ satisfies that $\{b(x,y)-\bar b(x):y\in\mathbb T^m\}$ spans $\mathbb R^2$ for each $x\in D$ and the averaged process is a Hamiltonian system with respect to $dx_1\wedge dx_2$ on $D$.
For example, as shown in Figure~\ref{fig:sphere}, we can modify the vector field on the plane outside $V$ so that there are two disks with the same center $D_1\subset D_2$, and the averaged process is a Hamiltonian system in $D_2$, in particular, rotation between $\partial D_1$ and $\partial D_2$. Then $D_2$ is smoothly glued to a hemisphere and the resulting manifold is $M$, and the vector field can be extended to the surface in such a way that the averaged process is a rotation with certain constant angular velocity on the level sets. 
\begin{figure}[!htbp]
\centering
\begin{subfigure}{.5\textwidth}
  \centering
  \includegraphics[width=.9\linewidth]{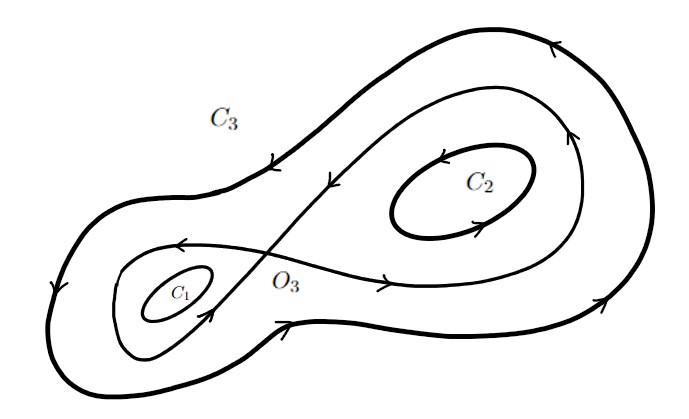}
  \caption{Three connected components of $\mathbb R^2\setminus V$.}
  \label{fig:three_components}
\end{subfigure}%
\begin{subfigure}{.5\textwidth}
  \centering
  \includegraphics[width=.9\linewidth]{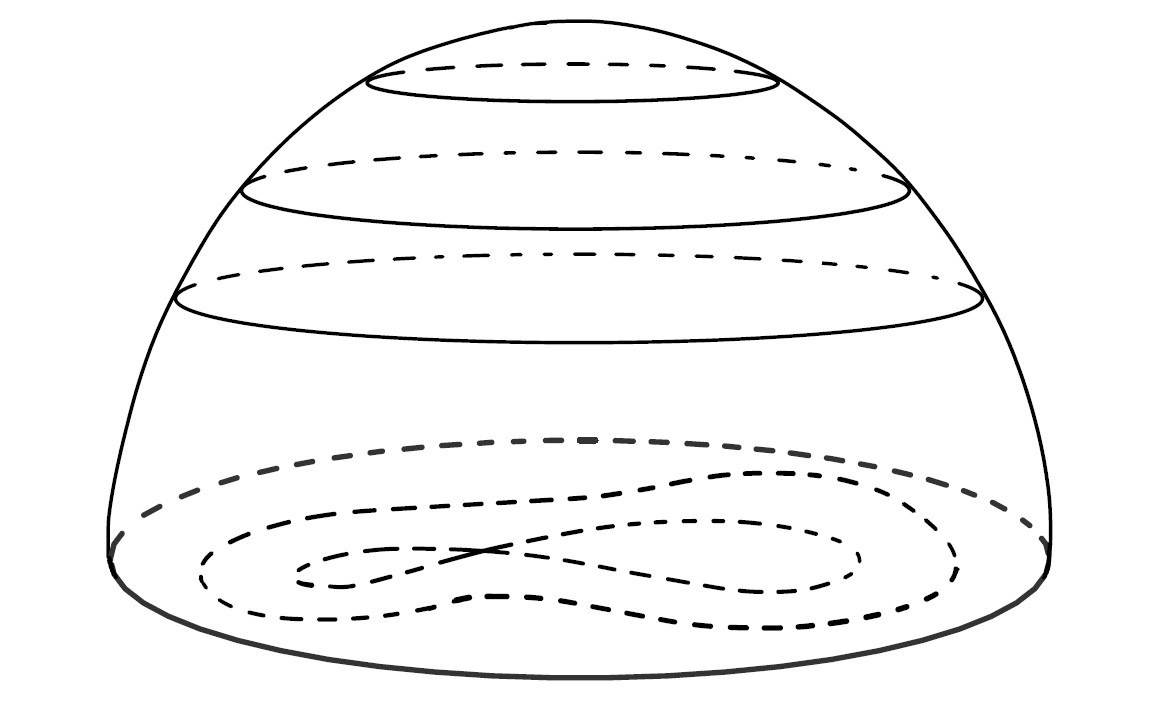}
  \caption{Hamiltonian system on manifold $M$.}
  \label{fig:sphere}
\end{subfigure}
\caption{Localization.}
\label{fig:localization}
\end{figure}

It is clear that, when restricted to $V\times\mathbb T^m$, the resulting system defined on $M\times\mathbb T^m$ has exactly the same behavior as the original process on $\mathbb R^2\times\mathbb T^m$. Therefore, it suffices to prove \eqref{eq:martingale_problem_stopped} for the new process on $M\times\mathbb T^m$.
Let us formally restate the corresponding assumptions and formulate the result on $M\times\mathbb T^m$. 
In the remainder of the paper, all the definitions (e.g. the quantities defined in Section~\ref{sec:mainresult}) and statements on $M$ are understood by locally choosing coordinates so that $\omega=dx_1\wedge dx_2$. In particular, on $S$, they are understood in the coordinate $\Phi$, while on the "flat" part that contains $V$, they are understood in the usual way.
Then the assumptions on the coefficients on $M\times\mathbb T^m$ are analogous to those introduced earlier, so we only mention the differences:
\begin{enumerate}
    \item[\hypertarget{H2'}{\textit{(H2$'$)}}]$H(x)$ is a $C^\infty$ function from $M$ to $\mathbb R$  that has three extremum points and one saddle point.
    \item[\hypertarget{H3'}{\textit{(H3$'$)}}] $b(x,y)$ is a $C^\infty$ function from $M\times\mathbb T^m$ to $TM$ such that $\bar b(x)=\nabla^\perp H(x)$. 
    \item[\hypertarget{H4'}{\textit{(H4$'$)}}] $\{b(x,y)-\bar b(x):y\in\mathbb T^m\}$ spans $TM$ for all $x\in M$.
\end{enumerate}

From this point on, we denote the process on $M\times\mathbb T^m$ as $(\bx_t^\e,\bxi_t^\e)$, and $(X_t^\e,\xi_t^\e)$ on the time scale $O(\e^{-1})$ (defined by \eqref{eq:theprocess1} and \eqref{eq:rescaled_process1} with $\mathbb R^2$ replaced by $M$), and assume that the conditions \hyperlink{H2'}{\textit{(H2$'$)-(H4$'$)}} replacing \hyperlink{H2}{\textit{(H2)-(H4)}} hold. Then \eqref{eq:mgprob} follows from the next result (see \eqref{eq:martingale_problem_stopped}).
\begin{proposition}
\label{prop:main_result}
For each $f\in \mathcal D_{\mathcal L}$ and each $T>0$,
    \begin{equation}
    \label{eq:mg_problem_M}
        \E_{(x,y)}[f(h(X_{\eta}^\e))-f(h(x))-\int_0^{\eta}\mathcal Lf(h(X_t^\e))dt]\to0,
    \end{equation}
    as $\e\to0$, uniformly in $x\in M$, $y\in\mathbb T^m$, and $\eta\leq T$ that is a stopping time w.r.t. $\mathcal F_t^{X^\e_\cdot}$.
\end{proposition}
\subsection{Auxiliary process}
It turns out that similar results hold for a more general process with a slightly perturbed fast component:
\begin{equation}
\label{eq:auxiliary}
\begin{aligned}
    d\tbx_t^\e=&\ b(\tbx_t^\e,\tbxi^\e_t)dt,\\
    d\tbxi^\e_t=&\ \frac{1}{\e}v(\tbxi^\e_t)dt+\frac{1}{\sqrt\e}\sigma(\tbxi^\e_t)dW_t+c(\tbx_t^\e,\tbxi^\e_t)dt,
\end{aligned}
\end{equation}
where $c(x,y)$ is infinitely differentiable. 
Namely, $h(\tbx_{t/\e}^\e)$ converges weakly to the Markov process defined by the operator $(\mathcal L_c, D(\mathcal L_c))$ on the Reeb graph. 
Here, the subscript $c$ indicates that $\mathcal L_c$ depends on the choice of $c(x,y)$. If $c(x,y)=0$, then it is clear that $(\tbx_t^\e,\tbxi_t^\e)=(\bx_t^\e,\bxi_t^\e)$, and thus $\mathcal L=\mathcal L_c$. 
However, if $c(x,y)\not=0$, then we have an additional drift term in \eqref{eq:auxiliary}, and thus we need an additional drift term in the generator of the limiting process. 
While a precise definition of $(\mathcal L_c, D(\mathcal L_c))$ is deferred to later sections, we observe that the operators replacing $L_k$ depend on $c(x,y)$, and the domain $D(\mathcal L_c)$ as well as the linear subspace, denoted by $\mathcal D_{\mathcal L_c}$, chosen in Lemma~\ref{lem:martingale_problem} also vary for different $c(x,y)$. 
Therefore, in order to formulate general results, we consider $\mathcal D$, the set of  continuous functions on $\mathbb G$ that are four-times continuously differentiable inside each edge and satisfy conditions (i) and (iii) in Definition~\ref{def:domain_original}, as well as a weaker form of condition (ii), namely, the limits $\lim_{h_k\to O_i} L_k f(h_k)$ exist but are not necessarily independent of the edge $I_k$. 
Note that $\mathcal D$ contains $\mathcal D_{\mathcal L_c}$ for all choices of $c(x,y)$. Define $\mathcal L_c$ on $\mathcal D$ by applying the differential operator ($L_k$ plus an additional drift corresponding to $c(x,y)$) on each edge separately, with the result not being necessarily continuous at the interior vertices.

As mentioned before, we need to construct a family of auxiliary processes that, on the one hand, have a common invariant measure for all $\e>0$ and, on the other hand, are close to the processes of interest. 
The auxiliary process on $M$ can in fact be obtained by choosing a special $c(x,y)$ in \eqref{eq:auxiliary}.
We denote this particular choice of $c(x,y)$ as $\Tilde c(x,y)$. 
Now we find $\Tilde c(x,y)$ such that $\lambda\times\mu$ is the invariant measure for the process with every $\e$, where $\lambda$ is the area measure w.r.t. $\omega$ and $\mu$ is the invariant measure for $\bxi_t^\e$ in $\mathbb T^m$. 
Let $\Tilde L^\e$ be the generator of the process $(\tbx_t^\e,\tbxi_t^\e)$:
\[
\T L^\e f(x,y)=b(x,y)\cdot\nabla_x f(x,y)+\Tilde c(x,y)\cdot\nabla_y f(x,y)+\frac{1}{\e}L f(x,y).
\]
Hence, $\lambda\times\mu$ is the invariant measure if ${\T L^{\e*}}p(y)=0$, where ${\T L^{\e*}}$ is the adjoint operator of $\T L^\e$ and $p$ is the density of $\mu$, i.e.
\begin{equation}
\label{eq:adjointoperator}
\begin{aligned}
    0={\T L^{\e*}}p(y)=&-\mathrm{div}_x (b(x,y)p(y))-\mathrm{div}_y(\Tilde c(x,y)p(y))+\frac{1}{\e}L^{*}p(y),
\end{aligned}
\end{equation}
where $L^{*}$ is the adjoint operator of $L$. Since $\mu$ is the invariant measure for $\bxi^\e_t$, the last term vanishes. Hence \eqref{eq:adjointoperator} reduces to 
\begin{equation}
\label{eq:added term}
    \mathrm{div}_x b(x,y)p(y)+\mathrm{div}_y(\Tilde c(x,y)p(y))=0.
\end{equation}
To see the existence of the solution, we need the following lemma (cf. Lemma~2.1 in \cite{Freidlin2021}).
\begin{lemma}
\label{lem:solution}
    Let $\Tilde g(x,y)$ be a bounded function on $\mathbb R^2\times\mathbb T^m$ that is infinitely differentiable, and let $\Tilde L$ be the generator of a non-degenerate diffusion on $\mathbb T^m$ with the unique invariant measure $\Tilde\mu$ and suppose that $\int_{\mathbb T^m}\Tilde g(x,y)d\Tilde{\mu}(y)=0$ for each $x\in \mathbb R^2$. Then there exists a unique solution $\Tilde{u}(x,y)$ to the equation
    \begin{equation}
    \label{eq:theequation}
        \Tilde L\Tilde{u}(x,y)=-\Tilde{g}(x,y),~~~\int_{\mathbb T^m}\Tilde u(x,y)d\Tilde{\mu}(y)=0,
    \end{equation}
    and $\Tilde{u}(x,y)$ is also bounded and infinitely differentiable.
    Moreover, if $\tilde g(x,y)$ has uniformly bounded derivatives up to order $K$ in $x$ (or $y)$, the same holds for $\tilde u(x,y)$.
\end{lemma}
\begin{remark}
\label{rmk:existence_of_c}
The same result also holds for functions on $M\times\mathbb T^m$.
Thus, the existence of the solution to \eqref{eq:added term} immediately follows from Lemma~\ref{lem:solution} applied to $\Tilde g(x,y)=\mathrm{div}_x b(x,y)p(y)$ and $\Tilde L=\Delta_y$, and by taking the gradient of the solution in \eqref{eq:theequation} w.r.t. $y$, and dividing it by $p(y)$.
\end{remark}
As in \eqref{eq:rescaled_process1}, we define $(\tx_t^\e,\txi_t^\e)=(\tbx_{t/\e}^\e,\tbxi_{t/\e}^\e)$ in distribution.
Then, a simple corollary can be obtained by using Lemma~\ref{lem:solution} and then applying Ito's formula to the corresponding solution $\Tilde u(\tbx_t^\e,\tbxi_t^\e)$ and $\Tilde u(\tx_t^\e,\txi_t^\e)$ (cf. Lemma~2.3 in \cite{Freidlin2021}).
\begin{corollary}
    \label{cor:avg}
    Let $\Tilde g$ satisfy the all the conditions in Lemma~\ref{lem:solution} with $\Tilde{L}=L$ and $K=1$, then for fixed $T>0$
    \begin{equation}
        \label{eq:avg_coef}
        \E_{(x,y)}\left|\int_0^{\eta} \Tilde g(\tbx_s^\e,\tbxi_s^\e)ds\right|=O(\sqrt\e),~~~~\E_{(x,y)}\left|\int_0^{\eta} \Tilde g(\tx_s^\e,\txi_s^\e)ds\right|=O(\e),
    \end{equation}
    uniformly in $x\in M$, $y\in\mathbb T^m$, and $\eta$ that is a stopping time bounded by $T$.
\end{corollary}

\subsection{Diffusion approximation}
{Since $\int_{\mathbb T^m}(b(x,y)-\nabla^\perp H(x))d\mu(y)=0$, by Lemma~\ref{lem:solution}, there exists a function $u$ that is bounded together with its derivatives such that
\begin{equation}
\label{eq:u}
    L u(x,y)=-(b(x,y)-\nabla^\perp H(x)).
\end{equation}}
The equation is understood element-wise. Apply Ito's formula to $u(\tbx_t^\e,\tbxi_t^\e)$:
\begin{equation}
\label{eq:ito_diffusion_approx}
    \begin{aligned}
        u(\tbx_t^\e,\tbxi_t^\e)&=u(x_0,y_0)+\frac{1}{\e}\int_0^tLu(\tbx_s^\e,\tbxi_s^\e)+\frac{1}{\sqrt{\e}}\int_0^t\nabla_y u(\tbx_s^\e,\tbxi_s^\e)\sigma(\tbxi_s^\e)dW_s\\
        &\quad+\int_0^t [\nabla_x u(\tbx_s^\e,\tbxi_s^\e)b(\tbx_s^\e,\tbxi_s^\e)+\nabla_y u(\tbx_s^\e,\tbxi_s^\e)c(\tbx_s^\e,\tbxi_s^\e)]ds.
    \end{aligned}
\end{equation}
Combining \eqref{eq:auxiliary}, \eqref{eq:u}, and \eqref{eq:ito_diffusion_approx}, we obtain
\begin{equation}
\label{eq:slowx}
    \begin{aligned}
        \tbx_t^\e &= x_0+\int_0^t\nabla^\perp H(\tbx_s^\e)ds+\e\int_0^t [\nabla_x u(\tbx_s^\e,\tbxi_s^\e)b(\tbx_s^\e,\tbxi_s^\e)+\nabla_y u(\tbx_s^\e,\tbxi_s^\e)c(\tbx_s^\e,\tbxi_s^\e)]ds\\
        &\quad+\sqrt{\e}\int_0^t \nabla_y u(\tbx_s^\e,\tbxi_s^\e)\sigma(\tbxi_s^\e)dW_s+\e(u(x_0,y_0)-u(\tbx_t^\e,\tbxi_t^\e)).
    \end{aligned}
\end{equation}
Similarly, by applying Ito's formula to $u(\tx_t^\e,\txi_t^\e)$ and repeating the steps above, we have
\begin{equation}
\label{eq:x}
    \begin{aligned}
        \tx_t^\e &= x_0+\frac{1}{\e}\int_0^t\nabla^\perp H(\tx_s^\e)ds+\int_0^t [\nabla_x u(\tx_s^\e,\txi_s^\e)b(\tx_s^\e,\txi_s^\e)+\nabla_y u(\tx_s^\e,\txi_s^\e)c(\tx_s^\e,\txi_s^\e)]ds\\
        &\quad+\int_0^t \nabla_y u(\tx_s^\e,\txi_s^\e)\sigma(\txi_s^\e)dW_s+\e(u(x_0,y_0)-u(\tx_t^\e,\txi_t^\e)).
    \end{aligned}
\end{equation}
This idea of diffusion approximation is frequently used in the remainder of the paper, and the function $u(x,y)$ always refers to the solution to \eqref{eq:u}.
\section{Averaging principle inside one domain}
\label{sec:averaging}
In this section, we consider a general process $(\tx_t^\e,\txi_t^\e)$ defined after Remark~\ref{rmk:existence_of_c}, which is a faster version of the process in \eqref{eq:auxiliary}. 
The process takes values on $M\times\mathbb T^m$.
As a result of localization, $M$ is separated into three domains, each bounded by the separatrix or a part of it.
This section is devoted to the proof of the averaging principle for $(\tx_t^\e,\txi_t^\e)$ on $M\times\mathbb T^m$ up to time when $\tx_t^\e$ exits from one of the three domains. The domain under consideration will be denoted by $U$. 
Therefore, without any ambiguity, the projection $h$ simply reduces to the Hamiltonian $H$.
Let $U(h_1,h_2)$ be the region in $U$ between $\gamma(h_1)$ and $\gamma(h_2)$, $O$ be the saddle point, and $O'$ be the extremum point, and further define stopping times $\tau(h)=\inf\{t:|H(\tx_t^\e)-H(O)|=h\}$ and $\eta(h)=\inf\{t:|H(\tx_t^\e)-H(O')|=h\}$.
Without loss of generality, we assume that $H(O)=0$ and $H(O')=1$.
\subsection{Averaging principle before \texorpdfstring{$\tau(\e^\alpha)\wedge\eta(\delta)$}{}}
\label{sec:Averaging principle before}
We aim to prove the averaging principle between $\gamma(\e^\alpha)$ and $\gamma(1-\delta)$ with constants $0<\alpha<1/4$ and $0<\delta<1$.
Notice that, for technical reasons, we assume that $0<\alpha<1/4$ in this intermediate result and in the proofs that utilize it in this subsection and the next, while we always assume that $0<\alpha<1/2$ elsewhere.
Let us further define another coordinate $\phi$ inside this domain $U$. 
Let $l$ denote the curve that is tangent to $\nabla H$ at each point and connects the saddle point $O$ and the extremum point $O'$, and let $l(h)$ be the intersection of $l$ and $\gamma(h)$.
Let $Q(h)$ denote the time it takes for the averaged process $\bm x_t$ to make one rotation on $\gamma(h)$ and $q(x)$ denote the time it takes for $\bm x_t$ starting from $l(H(x))$ to arrive at $x$. Now we define the coordinate $\phi(x)=q(x)/Q(H(x))$ whose range is $S^1:=\mathbb R\ (\mathrm{mod}\ 1)$. It is easy to see that $\bm x_t$ has constant speed $1/Q(H(\bm{x}_t))$ in $\phi$ coordinate.
Since there is logarithmic delay near the saddle point, the coordinate $\phi$ has exploding derivatives near the separatrix. However, as shown in Appendix~\ref{sec:derivatives}, the order of its derivatives w.r.t. the Euclidean coordinates is under control. 
Let us denote $\Tilde H_t^\e=H(\tx_t^\e)$ and $\Tilde\Phi_t^\e=\phi(\tx_t^\e)$. Along the same lines leading to \eqref{eq:x}, we have the following equations with $u_h=u\cdot\nabla H$, $u_\phi=u\cdot\nabla\phi$, $h_0=H(x_0)$, and $\phi_0=\phi(x_0)$:
\begin{align}
    \Tilde H_t^\e&=h_0+\int_0^t \nabla_y u_h(\tx_s^\e,\txi_s^\e)^{\mathsf T}\sigma(\txi_s^\e)dW_s+\e(u_h(x_0,y_0)-u_h(\tx_t^\e,\txi_t^\e))\nonumber\\
    &\quad+\int_0^t[\nabla_x u_h(\tx_s^\e,\txi_s^\e)\cdot b(\tx_s^\e,\txi_s^\e)+\nabla_y u_h(\tx_s^\e,\txi_s^\e)\cdot c(\tx_s^\e,\txi_s^\e)]ds,\label{eq:H}\\
    \Tilde\Phi_t^\e&=\phi_0+\int_0^t \nabla_y u_\phi(\tx_s^\e,\txi_s^\e)^{\mathsf T}\sigma(\txi_s^\e)dW_s+\frac{1}{\e}\int_0^t\frac{1}{Q(\tilde H_s^\e)}ds\nonumber\\
    &\quad+\int_0^t[\nabla_x u_\phi(\tx_s^\e,\txi_s^\e)\cdot b(\tx_s^\e,\txi_s^\e)+\nabla_y u_\phi(\tx_s^\e,\txi_s^\e)\cdot c(\tx_s^\e,\txi_s^\e)]ds\nonumber\\
    &\quad+\e(u_\phi(x_0,y_0)-u_\phi(\tx_t^\e,\txi_t^\e)),\label{eq:phi}
\end{align}
for $\e^\alpha\leq h_0\leq 1-\delta$ and $t\leq\tau(\e^\alpha)\wedge\eta(\delta)$. {The term multiplied by $1/\e$ in \eqref{eq:H} disappears since $\nabla H\cdot\nabla^\perp H=0$.} Define the following coefficients using the original coordinates for all $x\in M$:
\begin{equation}
\begin{aligned}
    \label{eq:definition_of_operator_AB_x}
    A(x)&=\int_{\mathbb T^m}|\nabla_y u_h(x,y)^{\mathsf T}\sigma(y)|^2d\mu(y),\\
    {B_c}(x)&=\int_{\mathbb T^m}[\nabla_xu_h(x,y)\cdot b(x,y)+\nabla_y u_h(x,y)\cdot c(x,y)]d\mu(y);
\end{aligned}
\end{equation}
and $(h,\phi)$ coordinates for $x=(h,\phi)$, where $\e^\alpha\leq h\leq 1-\delta$ and $\phi\in S^1$:
\begin{equation}
    \begin{aligned}
    \label{eq:definition_of_operator_AB}
    A(h,\phi)&= A(x),~~~~~~\bar A(h)=\int_{S^1} A(h,\phi)d\phi,\\
    {B_c}(h,\phi)&= {B_c}(x),~~~~\bar {B_c}(h)=\int_{S^1} {B_c}(h,\phi)d\phi.
\end{aligned}
\end{equation}
Define ${\mathcal L_c}$ by ${\mathcal L_c}f=\frac{1}{2}\bar Af''+\bar {B_c}f'$ for $f\in\mathcal D$ in the interior of each edge. In particular, when $c(x,y)=0$, this definition is consistent with that in \eqref{def:operatorL}. Introduce two processes close to $\Tilde H_t^\e,\Tilde\Phi_t^\e$:
\begin{align}
    \hat H_t^\e&=h_0+\int_0^t \nabla_y u_h(\tx_s^\e,\txi_s^\e)^{\mathsf T}\sigma(\txi_s^\e)dW_s\nonumber\\
    &\quad+\int_0^t[\nabla_x u_h(\tx_s^\e,\txi_s^\e)\cdot b(\tx_s^\e,\txi_s^\e)+\nabla_y u_h(\tx_s^\e,\txi_s^\e)\cdot c(\tx_s^\e,\txi_s^\e)]ds,\label{eq:h_hat}\\
    \hat \Phi_t^\e&=\phi_0+\int_0^t \nabla_y u_\phi(\tx_s^\e,\txi_s^\e)^{\mathsf T}\sigma(\txi_s^\e)dW_s+\frac{1}{\e}\int_0^t\frac{1}{Q(\tilde H_s^\e)}ds\nonumber\\
    &\quad+\int_0^t[\nabla_x u_\phi(\tx_s^\e,\txi_s^\e)\cdot b(\tx_s^\e,\txi_s^\e)+\nabla_y u_\phi(\tx_s^\e,\txi_s^\e)\cdot c(\tx_s^\e,\txi_s^\e)]ds.\label{eq:phi_hat}
\end{align}
For each $f\in \mathcal D$, $x\in U(\e^\alpha,1-\delta)$, $y\in\mathbb T^m$, and stopping time {$\sigma'\leq T\wedge\eta(\delta)\wedge\tau(\e^\alpha)$}, by Ito's formula applied to $f(\hat H_{\sigma'}^\e)$, we have
\begin{equation}
    \begin{aligned}
        \E_{(x,y)} f(\hat H_{\sigma'}^\e)&=f(H(x))+\E_{(x,y)}\int_0^{\sigma'}\left(\frac{1}{2} |\nabla_y u_h(\tx_s^\e,\txi_s^\e)^{\mathsf T}\sigma(\txi_s^\e)|^2f''(\hat H_s^\e)\right.\\
        &\quad\left.+ \left[\nabla_x u_h(\tx_s^\e,\txi_s^\e)\cdot b(\tx_s^\e,\txi_s^\e)+\nabla_y u_h(\tx_s^\e,\txi_s^\e)\cdot c(\tx_s^\e,\txi_s^\e)\right]f'(\hat H_s^\e)\right)ds.
    \end{aligned}
\end{equation}
Since $\sup_{0\leq t\leq\sigma'}|\Tilde H_t^\e-\hat H_t^\e|=O(\e)$,
\begin{equation}
\label{eq:f}
    \begin{aligned}
        \E_{(x,y)} f(\Tilde H_{\sigma'}^\e)&=f(H(x))+\E_{(x,y)}\int_0^{\sigma'}\left(\frac{1}{2} |\nabla_y u_h(\tx_s^\e,\txi_s^\e)^{\mathsf T}\sigma(\txi_s^\e)|^2f''(\Tilde H_s^\e)\right.\\
        &\quad\left.+ \left[\nabla_x u_h(\tx_s^\e,\txi_s^\e)\cdot b(\tx_s^\e,\txi_s^\e)+\nabla_y u_h(\tx_s^\e,\txi_s^\e)\cdot c(\tx_s^\e,\txi_s^\e)\right]f'(\Tilde H_s^\e)\right)ds+O(\e).
    \end{aligned}
\end{equation}
Combining this with \eqref{eq:definition_of_operator_AB_x} and \eqref{eq:definition_of_operator_AB}, by Lemma~\ref{lem:solution}, as in Corollary~\ref{cor:avg}, we have
\begin{equation}
\label{eq:the_result}
    \E_{(x,y)}\left[f(\Tilde H_{\sigma'}^\e)-f(H(x))-\int_0^{\sigma'}\left(\frac{1}{2}A(\Tilde H_s^\e,\Tilde \Phi_s^\e)f''(\Tilde H_s^\e)+{B_c}(\Tilde H_s^\e,\Tilde \Phi_s^\e)f'(\Tilde H_s^\e)\right)ds\right]=O(\e).
\end{equation}
\begin{lemma}
\label{lem:avg_rotation}
Let $g(h,\phi)$ be either $A(h,\phi)f''(h)$ or ${B_c}(h,\phi)f'(h)$, and $\bar g(h)=\int_{S^1}g(h,\phi)d\phi$. 
Then, for every $T>0$,
\begin{equation}
    \sup_{\substack{x\in U(\e^\alpha, 1-\delta)\\ y\in\mathbb T^m}}\sup_{\sigma'\leq T\wedge\eta(\delta)\wedge\tau(\e^\alpha)}\E_{(x,y)}\left|\int_0^{\sigma'} \left[g(\Tilde H_s^\e,\Tilde \Phi_s^\e)-\bar g(\Tilde H_s^\e)\right]ds\right|\to0,~~\text{as }\e\downarrow0,
\end{equation}where the first supremum is taken over all stopping times {$\sigma'\leq T\wedge\eta(\delta)\wedge\tau(\e^\alpha)$}.
\end{lemma}
\begin{proof}
Fix $\kappa>0$. Since, for fixed $h$, $g(h,\phi)-\bar g(h)$ is a function on $S^1$, we can approximate it by a finite sum of its Fourier series with error less than $\frac{\kappa}{2T}$:
\begin{equation}
    g(h,\phi)-\bar g(h)\approx\sum_{0<|k|\leq K(\e)}g_k(h,\phi):=\sum_{0<|k|\leq K(\e)}G_k(h)\exp(2\pi ik\phi),
\end{equation}for all $\e^\alpha\leq h\leq 1-\delta$ and $\phi\in S^1$,
where
\begin{equation}
    G_k(h)=\int [g(h,\phi)-\bar g(h)]\exp(-2\pi ik\phi)d\phi.
\end{equation}
Since, as shown in Appendix~\ref{sec:derivatives}, $g''_{\phi\phi}=O(|\log h|/h)$, we see that $K(\e)$ can be chosen as $\e^{-\alpha}|\log \e|^2$ for sufficiently small $\e$.
Then it suffices to prove that, for all $0<|k|\leq K(\e)$ and $\e$ sufficiently small,
\begin{equation}
\label{eq:suffice}
    \sup_{\substack{x\in U(\e^\alpha, 1-\delta)\\ y\in\mathbb T^m}}\sup_{\sigma'\leq T\wedge \eta(\delta)\wedge\tau(\e^\alpha)}\E_{(x,y)}\left|\int_0^{\sigma'} g_k(\Tilde H_s^\e,\Tilde \Phi_s^\e)ds\right|=o\left(\frac{\e^\alpha}{|\log \e|^2}\right).
\end{equation}
We define an auxiliary function $v$ for fixed $g_k$, where $0<|k|\leq K(\e)$:
\begin{equation}
    v=\frac{g_k(h,\phi)Q(h)}{2\pi ik},
\end{equation}
which satisfies that $v'_\phi/Q(h)=g_k(h,\phi)$. We formulate the bounds on $\phi$, $v$, $g$, and their derivatives, uniformly in all $\e^\alpha<h<1-\delta$ and $0<|k|\leq K(\e)$ (proved in the Appendix~\ref{sec:derivatives}):
\begin{equation}
\label{eq:bounds}
\begin{aligned}
    \phi&\in[0,1),~\nabla\phi=O(1/h),~\nabla^2\phi=O(1/h^2),\\
    v &= O(|\log h|),~v'_\phi=O(|\log h|),~v''_{\phi\phi}=O(|\log h|^3/h),\\
    v'_h&=O(|\log h|^2/h),~v''_{hh}=O(|\log h|^3/h^3),~v''_{\phi h}=O(|\log h|^2/h),\\
    g'_h&=O(|\log h|/h),~g''_{hh}=O(|\log h|^2/h^3).
\end{aligned}
\end{equation}
By comparing $(\Tilde H_t^\e,\Tilde\Phi_t^\e)$ and $(\hat H_t^\e,\hat \Phi_t^\e)$ in \eqref{eq:H}, \eqref{eq:phi}, \eqref{eq:h_hat}, and \eqref{eq:phi_hat}, and using the bounds in \eqref{eq:bounds}, we know that for all {$\sigma'\leq T\wedge\eta(\delta)\wedge\tau(\e^\alpha)$},
\begin{equation}
\label{eq:gk_close}
    \int_0^{\sigma'}\left|g_k(\Tilde H_s^\e,\Tilde\Phi_s^\e)-\frac{v'_{\phi}(\hat  H_s^\e,\hat \Phi_s^\e)}{Q(\tilde H_s^\e)}\right|ds=\int_0^{\sigma'}\left|\frac{v'_{\phi}(\tilde  H_s^\e,\tilde \Phi_s^\e)-v'_{\phi}(\hat  H_s^\e,\hat \Phi_s^\e)}{Q(\tilde H_s^\e)}\right|ds=O(\e^{1-2\alpha}|\log\e|^3).
\end{equation}
Apply Ito's formula to $v(\hat H_{\sigma'}^\e,\hat \Phi_{\sigma'}^\e)$ and obtain
\begin{align*}
    \frac{1}{\e}\int_0^{\sigma'} \frac{v'_{\phi}(\hat  H_s^\e,\hat \Phi_s^\e)}{Q(\tilde H_s^\e)}ds&=v(\hat H_{\sigma'}^\e,\hat \Phi_{\sigma'}^\e)-v(H(x),\phi(x))-\int_0^{\sigma'}v'_h(\hat H_s^\e,\hat \Phi_s^\e)\nabla_y u_h(\tx_s^\e,\txi_s^\e)^{\mathsf T}\sigma(\txi_s^\e)  dW_s\\
    &\quad-\int_0^{\sigma'} v'_h(\hat H_s^\e,\hat \Phi_s^\e)[\nabla_x u_h(\tx_s^\e,\txi_s^\e)\cdot b(\tx_s^\e,\txi_s^\e)+\nabla_y u_h(\tx_s^\e,\txi_s^\e)\cdot c(\tx_s^\e,\txi_s^\e)]ds\\
    &\quad-\frac{1}{2}\int_0^{\sigma'}v''_{hh}(\hat H_s^\e,\hat \Phi_s^\e)|\nabla_y u_h(\tx_s^\e,\txi_s^\e)^{\mathsf T}\sigma(\txi_s^\e)|^2ds\\
    &\quad-\int_0^{\sigma'}v'_\phi(\hat H_s^\e,\hat \Phi_s^\e)\nabla_y u_\phi(\tx_s^\e,\txi_s^\e)^{\mathsf T}\sigma(\txi_s^\e)dW_s\\
    &\quad-\int_0^{\sigma'} v'_\phi(\hat H_s^\e,\hat \Phi_s^\e)[\nabla_x u_\phi(\tx_s^\e,\txi_s^\e)\cdot b(\tx_s^\e,\txi_s^\e)+\nabla_y u_\phi(\tx_s^\e,\txi_s^\e)\cdot c(\tx_s^\e,\txi_s^\e)]ds\\
    &\quad-\frac{1}{2}\int_0^{\sigma'}v''_{\phi\phi}(\hat H_s^\e,\hat \Phi_s^\e)|\nabla_y u_\phi(\tx_s^\e,\txi_s^\e)^{\mathsf T}\sigma(\txi_s^\e)|^2ds\\
    &\quad-\int_0^{\sigma'}v_{\phi h}''\nabla_y u_h(\tx_s^\e,\txi_s^\e)^{\mathsf T}\sigma(\txi_s^\e)\sigma(\txi_s^\e)^{\mathsf T}u_\phi(\tx_s^\e,\txi_s^\e)ds.
\end{align*}
By using the estimates in \eqref{eq:bounds} and the fact that $0<\alpha<1/4$, we know that the expectation of the {right-hand side} is $o(\frac{\e^{\alpha-1}}{|\log\e|^2})$. Combining this with \eqref{eq:gk_close}, we get \eqref{eq:suffice}. {Thus, the desired result follows.}
\end{proof}
Now, applying Lemma~\ref{lem:avg_rotation} to \eqref{eq:the_result}, we get
\begin{lemma}
\label{lem:bounded}
For each $f\in \mathcal D$, $0<\alpha<1/4$, and $0<\delta<1$, as $\e \downarrow 0$,
\begin{equation}
\label{eq:bounded}
    \sup_{\substack{x\in U(\e^\alpha, 1-\delta)\\ y\in\mathbb T^m}}\sup_{\sigma'\leq T\wedge\eta(\delta)\wedge\tau(\e^\alpha)}|\E_{(x,y)}[f(H(\tx_{\sigma'}^\e))-f(H(x))-\int_0^{\sigma'}{\mathcal L_c}f(H(\tx_s^\e))ds]|\to0,
\end{equation}where the first supremum is taken over all stopping times $\sigma'\leq T\wedge\eta(\delta)\wedge\tau(\e^\alpha)$.
\end{lemma}
\begin{remark}
\label{rmk:accessibility}
The diffusion process governed by $\mathcal L_c$ can reach all points inside the edge and the interior vertex but cannot reach the exterior vertex. For example, in the case considered here, the process can reach all points in $[0,1)$ but cannot reach $1$. The reason is that, for each $\delta>0$, on $[0,1-\delta]$, $\bar B_c(h)$ is bounded and $1/\bar A(h)\lesssim|\log h|$ (see Appendix~\ref{sec:derivatives} for details). However, for $\kappa>0$ sufficiently small, on $[1-\kappa,1]$, $B_c$ is uniformly negative while $A(h)\lesssim 1-h$ due to the non-degeneracy of the maximum point (see Lemma~\ref{lemb:non-zero-drift} for details).
\end{remark}
\begin{lemma}
    \label{lem:time_bounded_wo_t}
    For each $\delta>0$ and $0<\alpha<1/4$, $\E_{(x,y)}(\eta(\delta)\wedge\tau(\e^\alpha))$ in uniformly bounded for all $x\in U(\e^\alpha, 1-\delta)$, $y\in\mathbb T^m$, and $\e$ sufficiently small.
\end{lemma}
\begin{proof}
The solution $f^\delta$ to the following equation exists on $[0,1-\delta]$ due to Remark~\ref{rmk:accessibility}:
    \begin{equation}
    \begin{cases}
    {\mathcal L_c}f^\delta=-1\\
    f^\delta(0)=f^\delta(1-\delta)=0
    \end{cases}
    \end{equation}
Let $\Tilde T>3\|f^\delta\|_{\mathrm{sup}}$, then Lemma~\ref{lem:bounded} implies that, for all $x\in U(\e^\alpha, 1-\delta)$, $y\in\mathbb T^m$, and $\e$ small enough,
\begin{equation}
    \E_{(x,y)}(\eta(\delta)\wedge\tau(\e^\alpha)\wedge \Tilde T)<\Tilde T/2.
\end{equation}
Thus, by Markov inequality and strong Markov property, $\E_{(x,y)}(\eta(\delta)\wedge\tau(\e^\alpha))\leq 2\Tilde T$.
\end{proof}
\begin{lemma}
\label{lem:bounded_wo_t}
For each $f\in \mathcal D$, $\delta>0$, and $0<\alpha<1/4$, as $\e\downarrow0$,
\begin{equation}
\label{eq:bounded_wo_t}
    \sup_{\substack{x\in U(\e^\alpha, 1-\delta)\\ y\in\mathbb T^m}}\sup_{\sigma'\leq \eta(\delta)\wedge\tau(\e^\alpha)}|\E_{(x,y)} [f(H(\tx^\e_{\sigma'}))-f(H(x))-\int_0^{\sigma'}{\mathcal L_c} f(H(\tx_s^\e))ds]|\to0,
\end{equation}where the first supremum is taken over all stopping times $\sigma'\leq \eta(\delta)\wedge\tau(\e^\alpha)$.
\end{lemma}
\begin{proof}
The result can be deduced from \lemref{lem:bounded} and \lemref{lem:time_bounded_wo_t} by choosing a sufficiently large $T$ and using the Markov property.
\end{proof}


\subsection{Averaging principle before \texorpdfstring{$\sigma$}{sigma}}
With estimates on the transition times and probabilities between level sets near the critical points, the result from last subsection can be extended to the time when the process reaches the separatrix, {which is denoted by  $\sigma$.} The main result of this subsection is
\begin{proposition}
\label{prop:up_to_separatrix}
For each $f\in \mathcal D$, as $\e\downarrow0$,
\begin{equation}
\label{eq:main_result_before_sigma}
    \sup_{\substack{x\in U, y\in\mathbb T^m}}\sup_{\sigma'\leq\sigma}|\E_{(x,y)} [f(H(\tx^\e_{\sigma'}))-f(H(x))-\int_0^{\sigma'}{\mathcal L_{c}} f(H(\tx_s^\e))ds]|\to0,
\end{equation}where the first supremum is taken over all stopping times $\sigma'\leq \sigma$.
\end{proposition}
We state a simple corollary of Proposition~\ref{prop:exit_time_from_separatrix}.
\begin{corollary}
\label{cor:exit_time_eps}
    For each $0<\alpha<1/2$, uniformly in $x\in U(0,\e^\alpha)$ and $y\in\mathbb T^m$,
    \begin{equation}
        \E_{(x,y)}\sigma\wedge\tau(\e^\alpha)=O(\e^{2\alpha}|\log\e|).
    \end{equation}
\end{corollary}
\begin{lemma}
\label{lem:lin_prob}
    For each $0<\alpha<1/2$, uniformly in $x\in U(0,\e^\alpha)$ and $y\in\mathbb T^m$,
    \begin{equation}
        |\Prob_{(x,y)}(\tau(\e^\alpha)<\sigma)-H(x)\e^{-\alpha}|=O(\e^{\alpha}|\log\e|).
    \end{equation}
\end{lemma}
\begin{proof}
    As in \eqref{eq:H}, write the equation for $H(\tx_t^\e)=\Tilde H_t^\e$ stopped at ${\sigma\wedge\tau(\e^\alpha)}$, 
    \begin{align*}
        H(\tx_{\sigma\wedge\tau(\e^\alpha)}^\e)&=H(x)+\int_0^{\sigma\wedge\tau(\e^\alpha)}\nabla_x u_h(\tx_s^\e,\txi_s^\e)\cdot b(\tx_s^\e,\txi_s^\e)+\nabla_y u_h(\tx_s^\e,\txi_s^\e)\cdot c(\tx_s^\e,\txi_s^\e)ds\\
        &\quad+\int_0^{\sigma\wedge\tau(\e^\alpha)}\nabla_y u_h(\tx_s^\e,\txi_s^\e)^{\mathsf T}\sigma(\txi_s^\e)dW_s+O(\e).
    \end{align*}
    From Corollary~\ref{cor:exit_time_eps}, it follows that
    \begin{equation*}
        |\Prob_{(x,y)}(\tau(\e^\alpha)<\sigma)-H(x)\e^{-\alpha}|=\e^{-\alpha}|\E_{(x,y)} H(\tx_{\sigma\wedge\tau(\e^\alpha)}^\e)-H(x)|=O(\e^\alpha|\log\e|).\qedhere
    \end{equation*}
\end{proof}
We prove that the process spends finite time (in expectation) inside $U$. The idea is to use the fact that the process on the graph spends little time near the vertices and exits the edge with positive probability once it gets close enough to the interior vertex.
\begin{lemma}
\label{lem:time1}
    For each $0<\alpha<1/4$, $\E_{(x,y)}\tau(\e^\alpha)$ is uniformly bounded for all $x\in U$ such that $H(x)\geq\e^\alpha$, $y\in\mathbb T^m$, and $\e$ sufficiently small.
\end{lemma}
\begin{proof}
    By Lemma~\ref{lem:near_extremum}, fix $\delta>0$ such that $\E_{(x,y)}\eta(2\delta)\leq1$  for all $\e$ sufficiently small and all $x$ satisfying $H(x)>1-2\delta$ ; By Lemma~\ref{lem:bounded_wo_t}, fix $\kappa>0$ such that $\Prob_{(x,y)}(\eta(\delta)<\tau(\e^\alpha))<1-\kappa$ all $x$ satisfying $H(x)=1-2\delta$, all $y\in\mathbb T^m$, and all $\e$ sufficiently small; By Lemma~\ref{lem:time_bounded_wo_t}, fix $T>4(1+\sup_{x\in U(\e^\alpha, 1-\delta), y\in\mathbb T^m}\E_{(x,y)}(\tau(\e^\alpha)\wedge\eta(\delta)))/\kappa$.
    For all $x$ with $H(x)>1-2\delta$ and $y\in\mathbb T^m$,
    \begin{equation}
    \begin{aligned}
        &\Prob_{(x,y)}(\tau(\e^\alpha)>2T)\\ &\leq \Prob_{(x,y)}(\eta(2\delta)>T)
        +\sup_{(x',y')\in\gamma(1-2\delta)\times\mathbb T^m}\left(\Prob_{(x',y')}(\tau(\e^\alpha)\wedge\eta(\delta)>T)+\Prob_{(x',y')}(\eta(\delta)<\tau(\e^\alpha))\right)\\
        &\leq 1-\kappa/2.
    \end{aligned}
    \end{equation}
    For all $x\in U$ with $\e^\alpha\leq H(x)\leq 1-2\delta$, the estimate above holds without the first term on the second line. Then the uniform boundedness follows from the Markov property.
\end{proof}
We can apply the similar idea near the separatrix. Namely, we choose $0<\alpha'<\alpha<1/4$. By Corollary~\ref{cor:exit_time_eps}, the process spent little time spent between $\gamma$ and $\gamma(\e^{\alpha'})$, and by Lemma~\ref{lem:lin_prob}, the process is very likely to exit through the separatrix rather than come back to $\gamma(\e^{\alpha'})$ once it reaches $\gamma(\e^\alpha)$. Then, since $\E_{(x,y)}\tau(\e^\alpha)$ is uniformly bounded, one can prove the following result:
\begin{lemma}
\label{lem:time2}
    $\E_{(x,y)}\sigma$ is uniformly bounded for all $x\in U$, $y\in\mathbb T^m$, and $\e$ sufficiently small.
\end{lemma}
\begin{proof}[Proof of Proposition~\ref{prop:up_to_separatrix}]
     Fix $\kappa>0$ and $0<\alpha<1/4$. By Lemma~\ref{lem:time1}, let $T$ be large enough such that $\Prob_{(x,y)}(\tau(\e^\alpha)>T)<\kappa$ for all $x\in U$ satisfying $H(x)\geq\e^\alpha$, $y\in\mathbb T^m$, and $\e$ sufficiently small.
    By Lemma~\ref{lem:near_extremum}, let $\delta>0$ small enough such that for $\e$ sufficiently small
     \begin{equation}
     \label{eq:cov_near_extremum}
         \sup_{\substack{x\in U:H(x)\geq 1-\delta\\ y\in\mathbb T^m}}\sup_{\sigma'\leq\eta(\delta)}|\E_{(x,y)}[f(H(\tx_{\sigma'}^\e))-f(H(x))-\int_0^{\sigma'}\mathcal L_c f(H(\tx_s^\e))ds]|<\kappa,
     \end{equation}where the first supremum is taken over all stopping times $\sigma'\leq \eta(\delta)$.
     By Remark~\ref{rmk:accessibility} and Lemma~\ref{lem:bounded_wo_t}, let $\delta'>0$ small enough such that $\Prob_{(x,y)}(\eta(\delta')<\tau(\e^\alpha))<\kappa$ for all $x\in U(\e^\alpha,1-\delta)$, $y\in\mathbb T^m$, and $\e$ sufficiently small. 
     For stopping time $\sigma'\leq\tau(\e^\alpha)$, $x\in U(\e^\alpha,1-\delta)$, and $y\in\mathbb T^m$,
     \begin{equation}
     \label{eq:cov_to_eps}
         \begin{aligned}
             &|\E_{(x,y)} [f(H(\tx^\e_{\sigma'})-f(H(x))-\int_0^{\sigma'}{\mathcal L_c} f(H(\tx_s^\e))ds]|\\
             &\leq|\E_{(x,y)} [f(H(\tx^\e_{\eta(\delta')\wedge\sigma'})-f(H(x))-\int_0^{\eta(\delta')\wedge\sigma'}{\mathcal L_c} f(H(\tx_s^\e))ds]\\\
             &\quad+\Prob_{(x,y)}(\eta(\delta')<\sigma')\sup_{\substack{x'\in U:H(x')=\delta'\\ y'\in\mathbb T^m}}|\E_{(x',y')} [f(H(\tx^\e_{\sigma'}))-f(H(x'))-\int_0^{\sigma'}{\mathcal L_c} f(H(\tx_s^\e))ds]|.
         \end{aligned}
     \end{equation}
     Note that the first term converges to $0$ as $\e\downarrow0$ by Lemma~\ref{lem:bounded_wo_t}, the probability in the second term is less than $\kappa$, and the supremum is uniformly bounded for all $\e$ by Lemma~\ref{lem:time2}. Thus, the expression on the left-hand side of \eqref{eq:cov_to_eps} converges to $0$ uniformly. Combining this with \eqref{eq:cov_near_extremum}, we obtain
     \begin{equation}
     \label{eq:cov_to_eps2}
         \sup_{\substack{x\in U:H(x)\geq\e^\alpha\\ y\in\mathbb T^m}}\sup_{\sigma'\leq\tau(\e^\alpha)}|\E_{(x,y)}[f(H(\tx_{\sigma'}^\e))-f(H(x))-\int_0^{\sigma'}\mathcal L_c f(H(\tx_s^\e))ds]|\to0.
     \end{equation}
     Finally, let us choose $0<\alpha'<\alpha$. Apply Corollary~\ref{cor:exit_time_eps} and Lemma~\ref{lem:lin_prob} to obtain that $\E_{(x,y)}\sigma\wedge\tau(\e^{\alpha'})<\e^{\alpha'}$ and $\Prob_{(x,y)}(\sigma<\tau(\e^{\alpha'}))>1/2$ for all $x\in\gamma(\e^\alpha)$, $y\in\mathbb T^m$, and $\e$ sufficiently small. As in \eqref{eq:cov_to_eps}, by stopping the process at $\tau(\e^\alpha)\wedge\sigma'$ and $\tau(\e^{\alpha'})\wedge\sigma'$ and using the strong Markov property, we can conclude that 
     \begin{equation}
         \sup_{\substack{x\in U:H(x)\geq\e^\alpha\\ y\in\mathbb T^m}}\sup_{\sigma'\leq\sigma}|\E_{(x,y)}[f(H(\tx_{\sigma'}^\e))-f(H(x))-\int_0^{\sigma'}\mathcal L_c f(H(\tx_s^\e))ds]|\to0.
     \end{equation}
     Now \eqref{eq:main_result_before_sigma} follows from this by applying Corollary~\ref{cor:exit_time_eps} again.
\end{proof}

\subsection{Averaging principle starting from \texorpdfstring{$\gamma(\e^\alpha)$}{gamma(epsilon alpha)}}
Fix $0<\alpha<\alpha_1<\alpha_2<1/2$ and $r>0$ small enough. 
{More delicate results are obtained in this subsection to describe the behavior of processes during one excursion from $\gamma(\e^\alpha)$ to $\gamma$ (such excursions, in different domains, happen during the time intervals $[\tau_n,\sigma_n]$ defined in \eqref{eq:stopping_times}). In particular, Lemma~\ref{lem:eps_avg_prin_to_sp} gives bounds on contribution to (\ref{eq:mg_problem_M}) from each such excursion.}
Recall that $Q(h)$ is the rotation time of $\bm x_t$ on $\gamma(h)$.
Our first lemma concerns the typical deviation during one rotation.
\begin{lemma}
\label{lem4:one_rotation}
    For each $\delta>0$ there is $\kappa>0$ such that for all $x\in U(\e^{\alpha_1},r)$, $y\in\mathbb T^m$, and $\e$ sufficiently small,
    \begin{equation}
        \Prob_{(x,y)}\left(\sup_{t\in[0,Q(H(x))]}|H(\tbx_t^\e)-H(\bm x_t)|>\e^{1/2-\delta}\right)<\e^\kappa.
    \end{equation}
    There exists $\delta'>0$ and $\kappa>0$ such that for all $x\in U(\e^{\alpha_1},r)$, $y\in\mathbb T^m$, and $\e$ sufficiently small,
    \begin{equation}
        \Prob_{(x,y)}\left(\sup_{t\in[0,Q(H(x))]}|\tbx_t^\e-\bm x_t|>\e^{\delta'}\right)<\e^\kappa.
    \end{equation}
\end{lemma}
\begin{proof}
    It suffices to prove the result for $\delta<1/2-\alpha_1$. Fix $0<\delta'<\delta''<1/2-\alpha_1-\delta$. Recall the definition of $q$ in Subsection~\ref{sec:Averaging principle before} and consider the coordinates $H$ and $q$ in $U(\e^{\alpha_2},2r)$. As in \eqref{eq:H} and \eqref{eq:phi}, let $q_0=q(x)$, $u_h=u\cdot\nabla H$, $u_q=u\cdot\nabla q$, and $q_t=q_0+t$, and write the equations with $\tau^0=\inf\{t:|\tbh_t^\e-h_0|>\e^{1/2-\delta}~\mathrm{or}~|\tbq_t^\e-q_t|>\e^{\delta''}\}\wedge Q(h_0)$:
    \begin{align}
    H(\tbx_{\tau^0}^\e)&=H(x)+\sqrt{\e}\int_0^{\tau^0} \nabla_y u_h(\tbx_s^\e,\tbxi_s^\e)^{\mathsf T}\sigma(\tbxi_s^\e)dW_s+\e(u_h(x,y)-u_h(\tbx_{\tau^0}^\e,\tbxi_{\tau^0}^\e))\nonumber\\
    &\quad+\e\int_0^{\tau^0}[\nabla_x u_h(\tbx_s^\e,\tbxi_s^\e)\cdot b(\tbx_s^\e,\tbxi_s^\e)+\nabla_y u_h(\tbx_s^\e,\tbxi_s^\e)\cdot c(\tbx_s^\e,\tbxi_s^\e)]ds,\label{eq4:H_tau}\\
    q(\tbx_{\tau^0}^\e)&=q_{\tau^0}+\sqrt{\e}\int_0^{\tau^0} \nabla_y u_q(\tbx_s^\e,\tbxi_s^\e)^{\mathsf T}\sigma(\tbxi_s^\e)dW_s+\e(u_q(x,y)-u_q(\tbx_{\tau^0}^\e,\tbxi_{\tau^0}^\e))\nonumber\\
    &\quad+\e\int_0^{\tau^0}[\nabla_x u_q(\tbx_s^\e,\tbxi_s^\e)\cdot b(\tbx_s^\e,\tbxi_s^\e)+\nabla_y u_q(\tbx_s^\e,\tbxi_s^\e)\cdot c(\tbx_s^\e,\tbxi_s^\e)]ds.\label{eq4:q_tau}
\end{align}
In Appendix~\ref{sec:derivatives}, we prove that $|\nabla q|=O(|\nabla H|/H)$. Thus, it is not hard to see, by looking at the inverse of the Jacobian of $(H,q)$ w.r.t. $x$, that $|H(\tbx_t^\e)-H(\bm x_t)|\leq\e^{1/2-\delta}$ and $|\tbx_t^\e-\bm x_t|\leq\e^{\delta'}$ for all $t\leq\tau^0$. Let $S_H$ and $S_Q$ denote the stochastic integrals in \eqref{eq4:H_tau} and \eqref{eq4:q_tau}. Since ${\tau^0}\lesssim|\log\e|$,
\[\Prob_{(x,y)}(\tau^0< Q(h(x)))<\Prob_{(x,y)}(|S_H|>\e^{1/2-\delta}/2)+\Prob_{(x,y)}(|S_Q|>\e^{\delta''}/2).\] The variance of $S_H$ and $S_Q$ is small:
\begin{align*}
    \bm{\mathrm{Var}}(S_H)&=\e\E(\int_0^{\tau^0} |\nabla_y u_h(\tx_s^\e,\txi_s^\e)^{\mathsf T}\sigma(\txi_s^\e)|^2ds)\lesssim\e\E(\int_0^{\tau^0}|\nabla H(\tx_s^\e)|^2ds)\lesssim\e|\log\e|,\\
    \bm{\mathrm{Var}}(S_Q)&=\e\E(\int_0^{\tau^0} |\nabla_y u_q(\tx_s^\e,\txi_s^\e)^{\mathsf T}\sigma(\txi_s^\e)|^2ds)\lesssim\e\E(\int_0^{\tau^0}|\nabla q(\tx_s^\e)|^2ds)\lesssim\e^{1-2\alpha_1}|\log\e|.
\end{align*}
Hence both results follow from Chebyshev's inequality with $\kappa<\delta$.
\end{proof}
Let $F(h)$ be the solution to
\begin{equation}
\label{eq:ode}
    \begin{cases}
    {\mathcal L_c}F=-1\\
    F(0)=F(2r)=0
    \end{cases}
\end{equation} 
Let $\tau^1$ and $\tau^2$ be the first times for $\tx_t^\e$ to exit $U(\e^{\alpha_1},r)$ and $U(\e^{\alpha_2},2r)$, respectively. Let $x_t^\e=\bm x_{t/\e}$. 
\begin{lemma}
\label{lem4:function_F}
    There exists a function $g(r)$ with $\lim_{r\to0}g(r)=0$ such that $|F'(h)|<g(r)$ for all $0<h<2r$. There exists $C>0$ such that $|F''(h)|<C|\log h|$ and $|F'''(h)|<C/h$.
\end{lemma}
\begin{proof}
The bounds can be verified with the help of estimates for $Q(h)$ in Appendix~\ref{sec:derivatives}.    
\end{proof}
\begin{lemma}
\label{lem:time_exit_alpha1}
There exists a function $g(r)$ with $\lim_{r\to0}g(r)=0$ such that for all $x\in\gamma(\e^\alpha)$, $y\in\mathbb T^m$, and $\e$ sufficiently small,
\begin{equation}
\label{eq4:epsilon_close_time}
    \E_{(x,y)}\tau^1\leq \e^\alpha g(r).
\end{equation}
\end{lemma}
\begin{proof}
For $(\tx_t^\e,\txi_t^\e)$ starting from $(x,y)$, we define $\bar\tau^2=\e Q(H(x))\wedge\tau^2$. As in \eqref{eq:the_result}:
\begin{equation}
\label{eq4:bartau2}
    \E_{(x,y)}[F(H(\tx_{\bar\tau^2}^\e))-F(H(x))-\int_0^{\bar\tau^2}(\frac{1}{2}A(\tx_s^\e)F''(H(\tx_s^\e))+B(\tx_s^\e)F'(H(\tx_s^\e)))ds]=O(\e),
\end{equation}uniformly in $x\in U(\e^{\alpha_2},2r)$ and $y\in\mathbb T^m$.
By the definition of $\bar A(h)$ and $\bar B(h)$, one can see that $\e Q(H(x))\bar A(H(x))=\int_0^{\e Q(H(x))}A(x^\e_s)ds$ and $\e Q(H(x))\bar B(H(x))=\int_0^{\e Q(H(x))}B(x^\e_s)ds$. Since $F$ solves \eqref{eq:ode}, it follows that
\begin{equation}
\label{eq4:epsilonQ}
    \e Q(H(x))=-\e Q(H(x)){\mathcal L_c}F(H(x))=-\int_0^{\e Q(H(x))}\frac{1}{2}A(x^\e_s)F''(H(x^\e_s))+B(x^\e_s)F'(H(x^\e_s))ds.
\end{equation}
We prove that there exists $K>0$ such that 
\begin{equation}
\label{eq:compare_T}
    K\E_{(x,y)}(F(H(\tx_{\bar\tau^2}^\e))-F(H(x)))\leq-\e Q(H(x))
\end{equation}
uniformly in $x\in U(\e^{\alpha_1},r)$, $y\in\mathbb T^m$, and all $\e$ sufficiently small. Then it follows that
\begin{equation}
\label{eq4:sub_solution}
    \E_{(x,y)}\tau^1\leq KF(H(x)),
\end{equation}
for $x\in U(\e^{\alpha_1},r)$, $y\in\mathbb T^m$, and all $\e$ sufficiently small.
Indeed, we can define $\bar\tau^2_k$, $k\geq0$ recursively: $\bar\tau^2_0=0$, $\bar\tau^2_{k+1}=\inf\{t\geq\bar\tau^2_{k}:\tx_t^\e\not\in U(\e^{\alpha_2},2r)\}\wedge(\e Q(H(\tx^\e_{\bar\tau^2_k}))+\bar\tau^2_k)$, and denote the first $k$ such that $\bar\tau^2_k$ exceeds $\tau^1$ as $\bm n$. Then we have
\begin{equation}
    \begin{aligned}
        &\E_{(x,y)} \left[F(H(\tx_{\bar\tau^2_{\bm n}}))-F(H(x))\right]\\
        &=\E_{(x,y)}\sum_{k=0}^\infty\chi_{\bar\tau^2_{k}<\tau^1}\left[F(H(\tx_{\bar\tau^2_{k+1}}))-F(H(\tx_{\bar\tau^2_{k}}))\right]\\
        &\leq\E_{(x,y)}\sum_{k=0}^\infty\chi_{\bar\tau^2_{k}<\tau^1}\sup_{(x',y')\in U(\e^{\alpha_1},r)\times\mathbb T^m}\E_{(x',y')}\left[F(H(\tx_{\bar\tau^2}^\e))-F(H(x'))\right]\\
        &\leq\frac{1}{K}\E_{(x,y)}\sum_{k=0}^\infty\chi_{\bar\tau^2_{k}<\tau^1}(-\e Q(\tx_{\bar\tau^2_{k}})).
    \end{aligned}
\end{equation}
Hence \begin{equation*}
    \E_{(x,y)}\tau^1\leq\E_{(x,y)}\bar\tau^2_{\bm n}=\E_{(x,y)}\sum_{k=0}^\infty\chi_{\bar\tau^2_{k}<\tau^1}(\bar\tau^2_{k+1}-\bar\tau^2_{k})\leq\e\E_{(x,y)}\sum_{k=0}^\infty\chi_{\bar\tau^2_{k}<\tau^1}Q(\tx_{\bar\tau^2_{k}})\leq KF(H(x)).
\end{equation*}Then \eqref{eq4:epsilon_close_time} follows from \eqref{eq4:sub_solution} and Lemma~\ref{lem4:function_F} by taking $x\in\gamma(\e^\alpha)$.

To prove \eqref{eq:compare_T}, it is enough to see that, for $x\in U(\e^{\alpha_1},r)$, $y\in\mathbb T^m$, and $\e$ sufficiently small,
\begin{align*}
    &\e Q(H(x))+\E_{(x,y)}F(H(\tx_{\bar\tau^2}^\e))-F(H(x))\\
    &=-\E_{(x,y)}\int_0^{\e Q(H(x))}\left(\frac{1}{2}A(x^\e_s)F''(H(x^\e_s))+B(x^\e_s)F'(H(x^\e_s))\right)ds\\
    &\quad+\E_{(x,y)}\int_0^{\bar\tau^2}\left(\frac{1}{2}A(\tx_s^\e)F''(H(\tx_s^\e))+B(\tx_s^\e)F'(H(\tx_s^\e))\right)ds+O(\e)\\
    &=\E_{(x,y)}\int_0^{\bar\tau^2}\left(\frac{1}{2}A(\tx_s^\e)F''(H(\tx_s^\e))-\frac{1}{2}A(x^\e_s)F''(H(x^\e_s))\right)ds\\
    &\quad+\E_{(x,y)}\int_0^{\bar\tau^2}\left(B(\tx_s^\e)F'(H(\tx_s^\e))-B(x^\e_s)F'(H(x^\e_s))\right)ds\\
    &\quad-\E_{(x,y)}\int_{\bar\tau^2}^{\e Q(H(x))}\left(\frac{1}{2}A(x^\e_s)F''(H(x^\e_s))+B(x^\e_s)F'(H(x^\e_s))\right)ds+O(\e)\\
    &=o(\e Q(H(x))),
\end{align*}
where the first equality is due to \eqref{eq4:bartau2} and \eqref{eq4:epsilonQ} and the last equality is due to Lemma~\ref{lem4:one_rotation} and Lemma~\ref{lem4:function_F}. 
\end{proof}
Similarly to Lemma~\ref{lem:time2}, we can look at the transitions between $\gamma(\e^\alpha)$ and $\gamma(\e^{\alpha_1})$. By the transition probabilities given in Lemma~\ref{lem:lin_prob} and transition time given in Corollary~\ref{cor:exit_time_eps} and Lemma~\ref{lem:time_exit_alpha1}, one can obtain the following result using the strong Markov property.
\begin{corollary}
\label{cor:exit_time_r}
There exists a function $g(r)$ with $\lim_{r\to0}g(r)=0$ such that for all $x\in\gamma(\e^\alpha)$, $y\in\mathbb T^m$, and $\e$ sufficiently small,
\begin{equation}
    \label{eq:g(r)}
    \E_{(x,y)}\tau(r)\wedge\sigma\leq \e^\alpha g(r).
\end{equation}
\end{corollary}
\begin{lemma}
\label{lem:eps_avg_prin_to_sp}
For each $f\in\mathcal D$, as $\e\downarrow0$,
    \begin{equation}
    \sup_{(x,y)\in\gamma(\e^\alpha)\times\mathbb T^m}\sup_{\sigma'\leq\sigma}|\E_{(x,y)} [f(H(\tx^\e_{\sigma'}))-f(H(x))-\int_0^{\sigma'}{\mathcal L_c} f(H(\tx_s^\e))ds]|=o(\e^\alpha),
\end{equation}where the first supremum is taken over all stopping times $\sigma'\leq \sigma$.
\end{lemma}
\begin{proof}
Fix $\kappa>0$. By Corollary~\ref{cor:exit_time_r}, we can choose $r$ small enough so that for stopping time $\sigma'\leq\sigma$:  $|\E_{(x,y)}[H(\tx_{\tau(r)\wedge\sigma'}^\e)-H(x)]|<\kappa\e^\alpha$ and
$|\E_{(x,y)}[f(H(\tx_{\tau(r)\wedge\sigma'}^\e))-f(H(x))]|<\kappa\e^\alpha$, and
\begin{align*}
    \sup_{(x,y)\in\gamma(\e^\alpha)\times\mathbb T^m}\sup_{\sigma'\leq\sigma}|\E_{(x,y)}\int_0^{\tau(r)\wedge\sigma'}{\mathcal L_c} f(H(\tx_s^\e))ds|<\kappa\e^\alpha,
\end{align*}for all $\e$ sufficiently small, using similar arguments leading to \eqref{eq:H} and \eqref{eq:f}. It follows that, $\Prob_{(x,y)}(H(\tx_{\tau(r)\wedge\sigma'}^\e)=r)\leq H(x)/r+ \kappa\e^\alpha/r\leq2\e^\alpha/r$.
Therefore, uniformly in all $x\in\gamma(\e^\alpha)$, $y\in\mathbb T^m$, and $\sigma'\leq\sigma$,
\begin{align*}
    &|\E_{(x,y)} [f(H(\tx^\e_{\sigma'}))-f(H(x))-\int_0^{\sigma'}{\mathcal L_c} f(\tx_s^\e)ds]|\\
    &\leq|\E_{(x,y)} f[H(\tx^\e_{\tau(r)\wedge\sigma'})-f(H(x))-\int_0^{\tau(r)\wedge\sigma'}{\mathcal L_c} f(H(\tx_s^\e))ds]|\\
    &\quad+\Prob_{(x,y)}(H(X_{\tau(r)\wedge\sigma'}^\e)=r)\sup_{\substack{x'\in\gamma(r)\\ y'\in\mathbb T^m}}|\E_{(x',y')} [f(H(\tx^\e_{\sigma'}))-f(H(x))-\int_0^{\sigma'}{\mathcal L_c} f(H(\tx_s^\e))ds]|\\
    &\leq 3\kappa\e^\alpha,
\end{align*}
for $\e$ sufficiently small, due to Proposition~\ref{prop:up_to_separatrix} and our choice of $r$. The result follows because $\kappa$ can be chosen arbitrarily small.
\end{proof}
{The last result in this subsection provides estimates that will be used later to control the number of excursions from $\gamma(\e^\alpha)$ to $\gamma$ in finite time (see Corollary~\ref{cor:num_excursion}).}
\begin{lemma}
\label{lem:number_excursion}
    There is a constant $\kappa>0$ such that, for all $\e$ sufficiently small,
    \begin{equation}
        \sup_{(x,y)\in\gamma(\e^\alpha)\times\mathbb T^m}\E_{(x,y)} e^{-\sigma}\leq1-\kappa\e^\alpha.
    \end{equation}
\end{lemma}
\begin{proof}
    By Corollary~\ref{cor:exit_time_r}, as in the proof of Lemma~\ref{lem:lin_prob}, we can fix $0<r<1/3$ such that for all $x\in\gamma(\e^\alpha)$, $y\in\mathbb T^m$, and $\e$ sufficiently small, $\Prob_{(x,y)}(\tau(r)<\sigma)\geq\e^\alpha/2r$.
    Let $F$ be defined as in \eqref{eq:ode} and $t=F(r)/3$, then it follows from Proposition~\ref{prop:up_to_separatrix}, as $\e\downarrow0$,
    \begin{equation}
        \sup_{(x,y)\in\gamma(r)\times\mathbb T^m}\E_{(x,y)}[F(H(\tx_{\sigma\wedge\tau(2r)\wedge t}^\e))-F(H(x))-\int_0^{\sigma\wedge\tau(2r)\wedge t}{\mathcal L_c}F(H(\tx_s^\e))ds]\to 0.
    \end{equation}
    Thus, we have that for all $x\in\gamma(r)$, $y\in\mathbb T^m$, and $\e$ sufficiently small,
    \begin{equation}
        \E_{(x,y)}F(H(\tx_{\sigma\wedge\tau(2r)\wedge t}^\e)>F(r)/2,
    \end{equation}
    and it follows that,
    \begin{equation}
        \Prob_{(x,y)}(\sigma>t)\geq\Prob_{(x,y)}(\sigma\wedge\tau(2r)>t)>\frac{\E_{(x,y)}F(H(\tx_{\sigma\wedge\tau(2r)\wedge t}^\e)}{\sup_{[0,2r]}F(h)}>\frac{F(r)}{2\sup_{[0,2r]}F(h)}=:c_1(r).
    \end{equation}
    Then, for all $x\in\gamma(r)$, $y\in\mathbb T^m$, and $\e$ sufficiently small,
    \begin{equation}
        \E_{(x,y)} e^{-\sigma}\leq\Prob_{(x,y)}(\sigma\leq t)+\Prob_{(x,y)}(\sigma>t)e^{-t}\leq 1-\Prob_{(x,y)}(\sigma>t)(1-e^{-t})\leq 1-c(r),
    \end{equation}
    with $c(r)=(1-\exp(-F(r)/3))c_1(r)>0$, and therefore,
    \begin{align*}
        \E_{(x,y)} e^{-\sigma}&\leq \Prob_{(x,y)}(\sigma<\tau(r))+\Prob_{(x,y)}(\sigma>\tau(r))\sup_{{x'\in\gamma(r),y'\in\mathbb T^m}}\E_{(x',y')} e^{-\sigma}\\
        &\leq 1-\Prob_{(x,y)}(\sigma>\tau(r))(1-\sup_{{x'\in\gamma(r),y'\in\mathbb T^m}}\E_{(x',y')} e^{-\sigma})\\
        &\leq 1-\frac{1}{2}c(r)\frac{\e^\alpha}{r}.
    \end{align*}
    The result holds with $\kappa=c(r)/2r$.
\end{proof}

\section{Exponential convergence on the separatrix}
\label{sec:exponenitalconvergence}
We fix $0<\alpha<1/2$. As in \eqref{eq:stopping_times} and Figure \ref{fig:markov_chain}, we define inductively two sequences of stopping times $\sigma_n$, $n\geq0$, and $\tau_n$, $n\geq 1$, but now for the general process $(\tx_t^\e,\txi_t^\e)$ on $M\times\mathbb T^m$ with additional drift $c(x,y)$. Without loss of generality, we assume that the saddle point $O$ satisfies $H(O)=0$. Let $V^\e=\{x:|H(x)|<\e^\alpha\}$ and $U_1$, $U_2$, $U_3$ be the three domains separated by $\gamma$. We aim to prove that the distribution of Markov chain $(\tx^\e_{\sigma_n},\txi^\e_{\sigma_n})$ converges in total variation exponentially fast, uniformly in $\e$ and in the initial distribution. Namely, we have the following lemma.

\begin{lemma}
\label{lem5:expo_ergodicity}
Let $\nu^{n,\e}_{x,y}$ denote the measure on $\gamma\times\mathbb T^m$ induced by $(\tx_{\sigma_n}^\e,\txi_{\sigma_n}^\e)$ with starting point $(x,y)\in\gamma\times\mathbb T^m$. Then there exist a probability measure $\nu^\e$ on $\gamma\times\mathbb T^m$ and constants $\Xi>0$ and $0<c<1$ such that, for all $\e$ sufficiently small,
\begin{equation}
\label{eq:exponentialconvergence}
    \sup_{(x,y)\in\gamma\times\mathbb T^m}{\mathrm{TV}}(\nu^{n,\e}_{x,y},\nu^\e)<\Xi \cdot (1-c)^n,
\end{equation}
where TV is the total variation distance of probability measures. 
\end{lemma}
The rest of this section is devoted to the proof of Lemma~\ref{lem5:expo_ergodicity}. 
Let $\bm\sigma_n$, $n\geq0$, and $\bm\tau_n$, $n\geq1$, be the stopping times w.r.t. $(\tbx_t^\e,\tbxi_t^\e)$ that are analogous to $\sigma_n$, $\tau_n$ w.r.t. $(\tx_t^\e,\txi_t^\e)$.
The lemma is equivalent to the exponential convergence in total variation of $(\tbx_{\bm\tau_n}^\e,\tbxi_{\bm\tau_n}^\e)$ on $\gamma'\times\mathbb T^m$, uniformly in $\e$ and in the initial distribution. 
The proof consists of three steps:
\begin{enumerate}
    \item[\textbf{1.}] The process starting on $\gamma'\times\mathbb T^m$ hits $I\times\mathbb T^m$ before $\bm\tau_1$ with uniformly positive probability, {where $I$ is a fixed interval on the separatrix.}
    \item[\textbf{2.}] Let the process starting on $I\times\mathbb T^m$ evolve for a certain period of time. Then, by a local limit theorem, we can estimate from below the probabilities of hitting $O(\e)$-sized boxes in a certain $O(\sqrt{\e})$-sized region, uniformly in the starting point on $I\times\mathbb T^m$.
    \item[\textbf{3.}] By the parabolic H\"ormander condition \hyperlink{H5}{\textit{(H5)}}, we prove a common lower bound for the density of the distribution of the process starting from each of the $O(\e)$-sized boxes after a short time.
\end{enumerate}

Let us take care of these steps in order.

\textbf{Step 1}. Let $0<\beta<1$, which will be specified later. We prove in the next two results that the process has a uniformly positive probability of following along the averaged process $\bm x_t$ and going through a neighborhood of the saddle point without making a deviation more than $\beta\sqrt{\e}$ in terms of the value of $H$.
\begin{lemma}
\label{lem:stay_close_to_averaged}
For each fixed $\hat t>0$, $\beta'>0$,
\[\Prob_{(x,y)}\left(\sup_{0\leq t\leq\hat t}|\tbx_t^\e-\bm x_t|\leq\beta'\sqrt{\e}\right),\]
is uniformly positive for all $(x,y)\in M\times\mathbb T^m$ and $\e$ sufficiently small.
\end{lemma}
\begin{proof}
     Let the eigenvalues of $\nabla^2 H$ be bounded by $K$.
    Recall formula \eqref{eq:slowx}. By the boundedness of the coefficients, the event
    \[E:=\left\{\sup_{0\leq t\leq\Tilde t}|\int_0^t \nabla_y u(\tbx_s^\e,\tbxi_s^\e)\sigma(\tbxi_s^\e)dW_s|\leq\frac{1}{2}\beta'e^{-K\Tilde t}\right\}\]
    has positive probability, uniformly in the starting points.
    By \eqref{eq:slowx}, we have that on the event $E$, for $t\leq\Tilde t$ and $\e$ sufficiently small,
    \begin{align*}
        |\tbx_t^\e-\bm x_t|&\leq |\int_0^t(\nabla^\perp H(\tbx_s^\e)-\nabla^\perp H(\bm x_s)ds|+\sqrt{\e}|\int_0^t \nabla_y u(\tbx_s^\e,\tbxi_s^\e)\sigma(\tbxi_s^\e)dW_s|\\
        +\e|\int_0^t &[\nabla_x u(\tbx_s^\e,\tbxi_s^\e)b(\tbx_s^\e,\tbxi_s^\e)+\nabla_y u(\tbx_s^\e,\tbxi_s^\e)c(\tbx_s^\e,\tbxi_s^\e)]ds|+\e|u(x,y)-u(\tbx_t^\e,\tbxi_t^\e)|\\
        &\leq K\int_0^t|\tbx_s^\e-\bm x_s|ds+\beta'e^{-K\Tilde t}\sqrt{\e}.
    \end{align*}
    Then Gr\"onwall's inequality implies that $|\tbx_t^\e-\bm x_t|\leq\beta'\sqrt{\e}$ for all $t\leq\Tilde{t}$. Therefore, $E$ implies $\{\sup_{0\leq t\leq\hat t}|\tbx_t^\e-\bm x_t|\leq\beta'\sqrt{\e}\}$, and the uniform positivity follows.
\end{proof}
\begin{lemma}
\label{lem:throughsaddlepoint}
     For any given $0<c<1$, there exist curves $\Gamma_1$ and $\Gamma_2$ in $U_1$ such that 
    \begin{enumerate}[(i)]
        \item $\Gamma_1$ and $\Gamma_2$ have their tangent vectors as $\nabla H$. They intersect with the separatrix on different sides of the saddle point and the averaged motion on the separatrix spends finite time from $\Gamma_2$ to $\Gamma_1$. 
        \item Let $x\in\Gamma_1$ satisfy $2\beta\sqrt{\e}\leq |H(x)|\leq 2\sqrt{\e}$ and $\tau_x=\inf\{t:\tbx_t^\e\in\Gamma_2\}$. Then for all $y\in\mathbb T^m$, $\Prob_{(x,y)}(\sup_{0\leq t\leq \tau_x}|H(\tbx^\e_t)-H(x)|\leq\beta\sqrt{\e})>c$ for all $\e$ sufficiently small.
    \end{enumerate}
\end{lemma}
\begin{figure}[!htbp]
\centering
\begin{subfigure}{.4\textwidth}
  \centering
  \includegraphics[width=.9\linewidth]{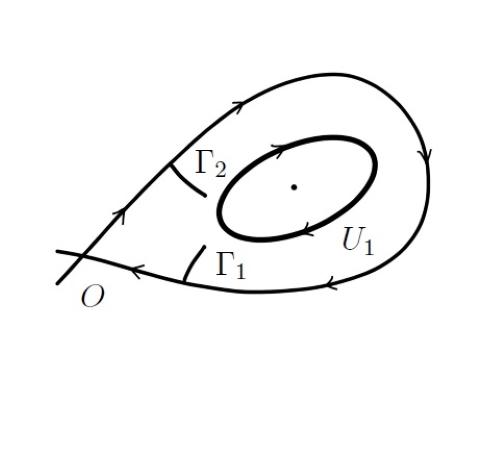}
\end{subfigure}%
\begin{subfigure}{.5\textwidth}
  \centering
  \includegraphics[width=.9\linewidth]{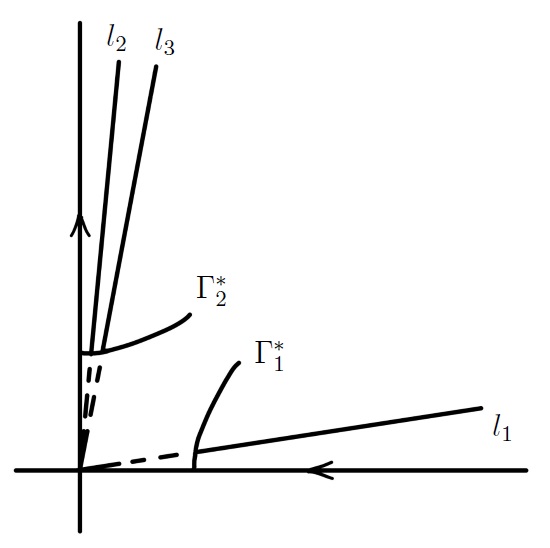}
\end{subfigure}
\caption{Curves in different coordinates.}
\label{fig:Curves in different coordinates}
\end{figure}
\begin{proof}
     Suppose $H(x)>0$ for all $x\in U_1$. By the Morse lemma, there exist neighborhoods $U$ and $V$ of the saddle point $O$ and the origin, respectively, and a diffeomorphism $\psi$ from $U$ to $V$ such that $H(x)=G(\psi(x))$, where $G(z)=z_1z_2$. Then consider a random change of time by dividing the generator by $D(x):=\mathrm{det}(\nabla_x\psi(x))$: 
    \begin{equation}
    \label{eq:randomchangeoftime}
    \begin{aligned}
        d\st{\tbx_t}&=\frac{b(\st{\tbx_t},\st{\tbxi_t})}{D(\st{\tbx_t})}dt,\\
        d\st{\tbxi_t}&=\frac{1}{\e}\frac{v(\st{\tbxi_t})}{D(\st{\tbx_t})}dt+\frac{1}{\sqrt\e}\frac{\sigma(\st{\tbxi_t})}{\sqrt{D(\st{\tbx_t})}}dW_t+\frac{c(\st{\tbx_t},\st{\tbxi_t})}{D(\st{\tbx_t})}dt.
    \end{aligned}
    \end{equation}
    Write the equation for $\st{\tz_t}:=\psi(\st{\tbx_t})$:
    \[
        d\st{\tz_t}=\frac{1}{D(\psi^{-1}(\st{\tz_t}))}\cdot \nabla_x\psi(\psi^{-1}(\st{\tz_t}))b(\psi^{-1}(\st{\tz_t}),\st{\tbxi_t})dt=:b^*(\st{\tz_t},\st{\tbxi_t})dt.
    \]
    It is not hard to verify that $b^*(z,y)$ satisfies
    $$\int_{\mathbb T^m}b^*(z,y)d\mu(y)=\nabla^\perp G(z).$$
    Hence, by Lemma~\ref{lem:solution}, there exists a bounded solution $u^*(z,y)$ to 
    \[Lu^*(z,y)=-(b^*(z,y)-\nabla^\perp G(z))\cdot D(\psi^{-1}(z)).\]
    Consider a local coordinate $G=z_1z_2$ and $\phi^*=\frac{1}{2}\log({z_2}/{z_1})$ in $V$.
    The averaged motion has constant speed: $0$ in $G$ and $1$ in $\phi^*$. 
    As in \eqref{eq:H} and \eqref{eq:phi}, we have the equations for $\st{\Tilde G_t}=G(\st{\tz_t})$, $\st{\Tilde \Phi_t}=\phi^*(\st{\tz_t})$, by applying Ito's formula to $u^*_g=u^*\cdot\nabla G$ and $u^*_{\phi}=u^*\cdot\nabla \phi^*$, with 
    $z=\psi(x)$, $g_0=G(z)$, and $\phi^*_0=\phi^*(z)$:
    \begin{align}
        \st{\Tilde G_t}=~&g_0+\sqrt{\e}\int_0^t \nabla_y u^*_g(\st{\tz_s},\st{\tbxi_s})^{\mathsf T}\frac{\sigma(\st{\tbxi_s})}{\sqrt{D(\psi^{-1}(\st{\tz_s}))}}dW_s-\e(u^*_g(\st{\tz_t},\st{\tbxi_t})-u^*_g(z,y))\nonumber\\
        &+\e\int_0^t\left[\nabla_z u^*_g(\st{\tz_s},\st{\tbxi_s})\cdot b^*(\st{\tz_s},\st{\tbxi_s})+\nabla_y u^*_g(\st{\tz_s},\st{\tbxi_s})\cdot \frac{c(\psi^{-1}(\st{\tz_s}),\st{\tbxi_s})}{D(\psi^{-1}(\st{\tz_s}))}\right]ds,\label{eq:G}\\
        \st{\Tilde \Phi_t}=~&\phi^*_0+t+\sqrt{\e}\int_0^t \nabla_y u^*_\phi(\st{\tz_s},\st{\tbxi_s})^{\mathsf T}\frac{\sigma(\st{\tbxi_s})}{\sqrt{D(\psi^{-1}(\st{\tz_s}))}}dW_s-\e(u^*_\phi(\st{\tz_t},\st{\tbxi_t})-u^*_\phi(z,y))\nonumber\\
        &+\e\int_0^t\left[\nabla_z u^*_\phi(\st{\tz_s},\st{\tbxi_s})\cdot b^*(\st{\tz_s},\st{\tbxi_s})+\nabla_y u^*_\phi(\st{\tz_s},\st{\tbxi_s})\cdot \frac{c(\psi^{-1}(\st{\tz_s}),\st{\tbxi_s})}{D(\psi^{-1}(\st{\tz_s}))}\right]ds.\label{eq:Phi*}
    \end{align}
    
    To get the lower bound for the desired probability, we will choose the curves $\Gamma_1$ and $\Gamma_2$ that are close enough to the saddle point. 
    The time it takes to get from $\Gamma_1$ to $\Gamma_2$ is still of order $|\log\e|$ since they are chosen independently of $\e$.
    In this way, the process starting on $\Gamma_1$ and stopped on $\Gamma_2$ will be shown to have small variance, hence it is unlikely for the process to have deviations larger than what we wish. 
    With $C>0$ to be specified later, let $l_1=\{z:\phi^*(z)=\frac{1}{4}\log\e+\frac{1}{2}\log\beta+C\}$, $l_2=\{z:\phi^*(z)=-(\frac{1}{4}\log\e+\frac{1}{2}\log\beta+C)\}$, and $l_3=\{z:\phi^*(z)=-(\frac{1}{4}\log\e+\frac{1}{2}\log\beta+C)-2\}$.
    The idea is to look at event that the process stays close to the averaged motion before the latter reaches $l_2$, which implies that the process does not make a large deviation in $G$, or equivalently, in $H$, before reaching $l_3$.
    Let $\Gamma_1^*$, $\Gamma_2^*$ be the curves that have tangent vectors as $\nabla_x\psi\circ\psi^{-1}({\nabla_x\psi\circ\psi^{-1}})^{\mathsf T}\nabla G$ and go through the points $(e^{-C},e^C\beta\sqrt{\e})$, $(e^{C+1}\beta\sqrt{\e},e^{-C-1})$, respectively.
    Since $\psi$ is a diffeomorphism, it is easy to see that each $z$ on $\Gamma_1^*$ or $\Gamma_2^*$ with $G(z)\geq\beta\sqrt{\e}$ satisfies that $\frac{1}{4}\log\e+\frac{1}{2}\log\beta+C\leq\phi^*(z)\leq-(\frac{1}{4}\log\e+\frac{1}{2}\log\beta+C)-2$.
    Let $\Gamma_1$ and $\Gamma_2$ be the pre-images of $\Gamma_1^*$ and $\Gamma_2^*$ in $U_1${, as shown in Figure \ref{fig:Curves in different coordinates}}.
    They have $\nabla H$ as tangent vectors due to the specific way we construct $\Gamma_1^*$ and $\Gamma_2^*$.
    Consider the process in \eqref{eq:randomchangeoftime} starting at $x\in\Gamma_1$ satisfying that $2\beta\sqrt{\e}\leq H(x)\leq 2\sqrt{\e}$ with an arbitrary $y\in\mathbb T^m$. 
    Let $\phi_t^*=\phi^*_0+t$.
    Define $t_x=\inf\{t:\phi_t^*=\phi^*(l_2)\}$ and $\tau_x^*=\inf\{t:|\st{\Tilde G_t}-g_0|=\beta\sqrt{\e}\}\wedge \inf\{t:|\st{\Tilde\Phi_t}-\phi_t^*|=1\}\wedge t_x$. 
    Then it is clear that $\Prob_{(x,y)}(\sup_{0\leq t\leq \tau^*_x}|H(\tbx^\e_t)-H(x)|\leq\beta\sqrt{\e})\geq\Prob_{(x,y)}(\tau_x^*=t_x)$. 
    Let $S_G$ and $S_\phi$ denote the stochastic integrals in \eqref{eq:G} and \eqref{eq:Phi*}, respectively, with $t$ replaced by $\tau_x^*$. 
    Since $\tau^*_x\lesssim|\log\e|$, $\nabla G$ is bounded, and $\nabla \phi^*\lesssim\e^{-1/2}$ before $\tau^*_x$, we see that the unwanted deviations happen only if $S_G$ and $S_\phi$ are large.
    Namely,
    \[\Prob_{(x,y)}(\tau_x^*<t_x)\leq\Prob_{(x,y)}(|S_G|\geq\beta\sqrt{\e}/2)+\Prob_{(x,y)}(|S_\phi|\geq1/2).\]
    Both terms on the right-hand side can be controlled by Chebyshev's inequality. 
    Note that there exists a constant $K>0$ independent of $\e$ such that
    \begin{align*}
        \bm{\mathrm{Var}}(S_G)&\leq\e K\E\int_0^{\tau_x^*}|\nabla G(\st{\tz_s})|^2ds\\
        &=\e K\E\int_0^{\tau_x^*} \st{\Tilde G_s}(e^{2\st{\Tilde \Phi_s}}+e^{-2\st{\Tilde \Phi_s}})ds\\
        &\leq \e K\int_0^{\tau_x^*}(2+\beta)\sqrt{\e}e^2(e^{2{\phi_s^*}}+e^{-2{\phi_s^*}})ds\\
        &\leq  3K\sqrt{\e^3}e^2\int_0^{-2(\frac{1}{4}\log\e+\frac{1}{2}\log\beta+C)}(e^{2{\phi_s^*}}+e^{-2{\phi_s^*}})ds\\
        &=\ 3K\sqrt{\e^3}e^2\int_{\frac{1}{4}\log\e+\frac{1}{2}\log\beta+C}^{-(\frac{1}{4}\log\e+\frac{1}{2}\log\beta+C)}(e^{2\varphi}+e^{-2\varphi})d\varphi\\
        &\leq \ \frac{3}{\beta}Ke^{2-2C}\e,
    \end{align*}and, similarly,
    \[
        \bm{\mathrm{Var}}(S_\phi)\leq \e K\E\int_0^{\tau^*_x}|\nabla \Phi(\st{\tz_s})|^2ds
        \leq\e K\E\int_0^{\tau^*_x} \frac{1}{\st{\Tilde G}_s}(e^{2\st{\Tilde \Phi_s}}+e^{-2\st{\Tilde\Phi_s}})ds\leq  \frac{1}{\beta^2}Ke^{2-2C}.
    \]
    Then, $C$ can be chosen large enough such that both variances are small enough, and hence $\Prob_{(x,y)}(|H(\tbx^\e_{\tau^*_x})-H(x)|\leq\beta\sqrt{\e})>c$. 
\end{proof}
\begin{figure}[!htbp] 
    \centering
    \includegraphics[width=0.4\textwidth]{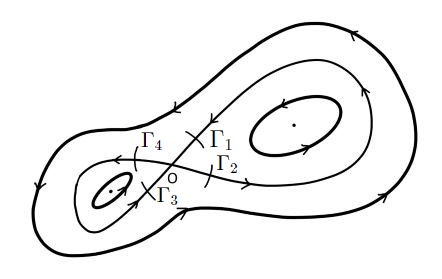}
    \caption{Four curves on four directions.}
    \label{fig:four curves}
\end{figure}
We can choose the corresponding curves in the other regions. As a result, we have four curves corresponding to four different directions, all with positive distance to the saddle point, as shown in Figure~\ref{fig:four curves}. Moreover, the corresponding transition probabilities near the saddle point have lower bounds analogous to that given in Lemma~\ref{lem:throughsaddlepoint} (ii). For the rotations happening away from those curves, we will prove that, before the time when the process comes back to the curves, the deviation of $H$ can be large enough to cross the separatrix with positive probability. Let $\Gamma_i(h_1,h_2)$ be the set $\{x\in\Gamma_i:h_1\leq H(x)\leq h_2\}$.
\begin{lemma}
    \label{lem:hit_separatrix}
    For each fixed $\hat t>0$,
    \[\Prob_{(x,y)}\left(\inf_{0\leq t\leq\hat t}H(\tbx_{t}^\e)\leq -\sqrt{\e},\sup_{0\leq t\leq\hat t}|\tbx_t^\e-\bm x_t|\leq\e^{\frac{1+2\alpha}{4}}\right)\]
    is positive uniformly in {$x\in\Gamma_2(0,2\sqrt{\e})$}, $y\in\mathbb T^m$, and all $\e$ sufficiently small.
\end{lemma}
\begin{proof}
    By Lemma~\ref{lem:stay_close_to_averaged} and the Markov property, it is enough to consider small $\hat t$ such that $\bm x_t$ does not reach $\Gamma_1$ before $\hat t$. Using formula \eqref{eq:slowx} again, we see that $\Prob_{(x,y)}(\sup_{0\leq t\leq\hat t}|\tbx_t^\e-\bm x_t|>\e^{\frac{1+2\alpha}{4}})\to 0$ as $\e\downarrow0$ uniformly in $(x,y)$. Use formula \eqref{eq:H} on a shorter time scale:
    \begin{align*}
        H(\tbx_{t}^\e)&=H(x)+\sqrt{\e}\int_0^{t} \nabla_y u_h(\tbx_s^\e,\tbxi_s^\e)^{\mathsf T}\sigma(\tbxi_s^\e)dW_s+\e(u_h(x,y)-u_h(\tbx_{t}^\e,\tbxi_{t}^\e))\\\
    &\quad+\e\int_0^{t}[\nabla_x u_h(\tbx_s^\e,\tbxi_s^\e)\cdot b(\tbx_s^\e,\tbxi_s^\e)+\nabla_y u_h(\tbx_s^\e,\tbxi_s^\e)\cdot c(\tbx_s^\e,\tbxi_s^\e)]ds.
    \end{align*}
    So, it suffices to show the uniform positivity of
    \[\Prob_{(x,y)}\left(\inf_{0\leq t\leq\hat t}\int_0^{t} \nabla_y u_h(\tbx_s^\e,\tbxi_s^\e)^{\mathsf T}\sigma(\tbxi_s^\e)dW_s\leq -4,\sup_{0\leq t\leq\hat t}|\tbx_t^\e-\bm x_t|\leq\e^{\frac{1+2\alpha}{4}}\right).\]
    Note that there exists another Brownian motion $\Tilde W$ such that
    \begin{equation}
    \label{eq5:time_changed_brownian_motion}
    \int_0^{t} \nabla_y u_h(\tbx_s^\e,\tbxi_s^\e)^{\mathsf T}\sigma(\tbxi_s^\e)dW_s=\Tilde W\left(\int_0^{t} |\nabla_y u_h(\tbx_s^\e,\tbxi_s^\e)^{\mathsf T}\sigma(\tbxi_s^\e)|^2ds\right).
    \end{equation}
    Recall that in Subsection~\ref{sec:Averaging principle before} we defined $A(x)=\int_{\mathbb T^m}|\nabla_y u_h(x,y)\sigma(y)|^2d\mu(y)$. Then, by Corollary~\ref{cor:avg},
    \begin{equation}
    \label{eq:close_in_L1}
        \E_{(x,y)}\left|\int_0^{\hat t} |\nabla_y u_h(\tbx_s^\e,\tbxi_s^\e)^{\mathsf T}\sigma(\tbxi_s^\e))|^2ds-\int_0^{\hat t} A(\tbx_s^\e)ds\right|=O(\sqrt\e).
    \end{equation}
    Note that on the event $\{\sup_{0\leq t\leq\hat t}|\tbx_t^\e-\bm x_t|\leq\e^{\frac{1+2\alpha}{4}}\}$, $A(\tbx_t^\e)$ is uniformly positive for $0\leq t\leq\hat t$. Let us denote this lower bound as $m$, which is independent of $x$, $y$, and $\e$. Then
    \begin{equation}
        \Prob_{(x,y)}(\int_0^{\hat t} A(\tbx_s^\e)ds>m\hat t,\sup_{0\leq t\leq\hat t}|\tbx_t^\e-\bm x_t|\leq\e^{\frac{1+2\alpha}{4}})\to1.
    \end{equation}
    By the $L^1$ convergence in \eqref{eq:close_in_L1}, we obtain
    \begin{equation}
    \label{eq5:positive_variance_one_rotation}
        \Prob_{(x,y)}\left(\int_0^{\hat t} |\nabla_y u_h(\tbx_s^\e,\tbxi_s^\e)^{\mathsf T}\sigma(\tbxi_s^\e))|^2ds>m\hat t/2,\sup_{0\leq t\leq\hat t}|\tbx_t^\e-\bm x_t|\leq\e^{\frac{1+2\alpha}{4}}\right)\to1.
    \end{equation}
    Suppose $0<c<\Prob(\inf_{0\leq t\leq m\hat t/2}\Tilde W_t<-4)$.
    Then, for all $\e$ sufficiently small,
    \begin{align*}
        &\Prob_{(x,y)}\left(\inf_{0\leq t\leq\hat t}\int_0^{t} \nabla_y u_h(\tbx_s^\e,\tbxi_s^\e)^{\mathsf T}\sigma(\tbxi_s^\e)dW_s\leq -4,\sup_{0\leq t\leq\hat t}|\tbx_t^\e-\bm x_t|\leq\e^{\frac{1+2\alpha}{4}}\right)\\
        &=\Prob_{(x,y)}\left(\inf_{0\leq t\leq\hat t}\Tilde W\left(\int_0^{t} |\nabla_y u_h(\tbx_s^\e,\tbxi_s^\e)^{\mathsf T}\sigma(\tbxi_s^\e))|^2ds\right)\leq -4,\sup_{0\leq t\leq\hat t}|\tbx_t^\e-\bm x_t|\leq\e^{\frac{1+2\alpha}{4}}\right)\\
        &\geq\Prob_{(x,y)}\left(\inf_{0\leq t\leq m\hat t/2}\Tilde W_t\leq -4,\int_0^{\hat t} |\nabla_y u_h(\tbx_s^\e,\tbxi_s^\e)^{\mathsf T}\sigma(\tbxi_s^\e))|^2ds>m\hat t/2,\sup_{0\leq t\leq\hat t}|\tbx_t^\e-\bm x_t|\leq\e^{\frac{1+2\alpha}{4}}\right)\\
        &\geq c/2.\qedhere
    \end{align*}
\end{proof}
\begin{remark}
\label{rmk:hit_separatrix}
    The result in Lemma~\ref{lem:hit_separatrix} also holds for $x\in\Gamma_4(0,2\sqrt{\e})$. Similarly, for each fixed $\hat t>0$,
    \[\Prob_{(x,y)}\left(\sup_{0\leq t\leq\hat t}H(\tbx_{t}^\e)\geq \sqrt{\e},\sup_{0\leq t\leq\hat t}|\tbx_t^\e-\bm x_t|\leq\e^{\frac{1+2\alpha}{4}}\right)\]
    is positive uniformly in $x\in\Gamma_2(-2\sqrt{\e},0)\cup\Gamma_4(-2\sqrt{\e},0)$, $y\in\mathbb T^m$, and $\e$ sufficiently small.
\end{remark}
Now we can choose $\beta=1/10$. By the results in Lemma~\ref{lem:stay_close_to_averaged}, Lemma~\ref{lem:throughsaddlepoint}, Lemma~\ref{lem:hit_separatrix}, and Remark~\ref{rmk:hit_separatrix}, using the strong Markov property, we obtain the following lemma:
\begin{lemma}
\label{lem:rotation}
There exists a closed interval $I$ on $\gamma$ that does not contain the saddle point and a constant $0<c<1$ satisfying the following property: if the system \eqref{eq:auxiliary} starts at $(x,y)\in\gamma'\times\mathbb T^m$, then for all $\e$ sufficiently small
    \begin{equation}
        \Prob_{(x,y)}(\Tilde\eta_1<\bm\tau_1)\geq c
    \end{equation}
    where $\Tilde\eta_1=\inf\{t:\tbx_t^\e\in I\}$.
\end{lemma}
\begin{remark}
\label{rmk:positive_any_subset_separatrix}
    In order for us to apply Lemma~\ref{lem:hit_separatrix}, we need to choose $I$ that contains the intersection of $\Gamma_2$ and $\gamma$ in its interior. In fact, it is not difficult to show that Lemma~\ref{lem:rotation} holds for any subset of $\gamma$ with non-empty interior.
\end{remark}

\textbf{Step 2}. 
Without loss of generality, we assume that, if $\bm x_t$ starts at one endpoint of $I$, then the other endpoint is $\bm x_{1/2}$.
In the remainder of this section, $\bm x_t$ always denotes this deterministic motion, irrespective of where $\tbx_t^\e$ starts. 
We aim to study the distribution of the process $(\tbx_t^\e,\tbxi_t^\e)$ starting on $I\times\mathbb T^m$ with certain $t>0$. The choice of $t$ will depend on the initial point $x$ being considered (see Figure~\ref{fig:local_limit_theorem}), and this will be convenient as we use the strong Markov property later when combining all three steps.
\begin{figure}[!htbp]
    \centering
    \includegraphics[width=0.7\textwidth]{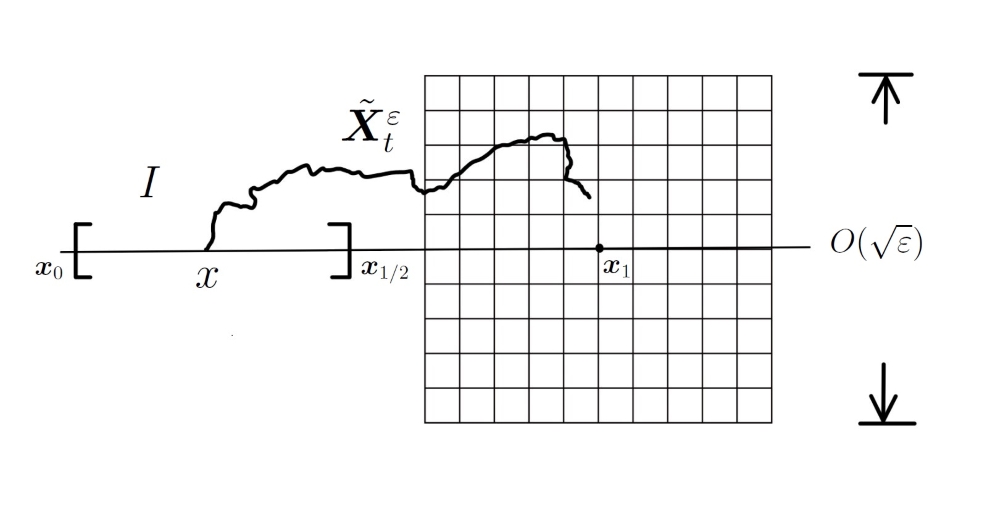}
    \caption{Higher order estimate on the distribution of $\tbx_t^\e$.}
    \label{fig:local_limit_theorem}
\end{figure}
To be more precise, for $x\in I$, let $s(x)$ be such that $\bm x_{s(x)}=x$ (so $0\leq s(x)\leq1/2$). We introduce a process $\tilde\zeta_{t}^{\e}$, $0\leq t$, as the second term in the expansion of $\tbx_{t}^{\e}$ around the deterministic motion $\bm x_{(s(x)+t)}$:
\begin{equation}
\label{eq:linearize_process_tilde}
\begin{aligned}
d\tilde\zeta_{t}^{\e}&=\nabla(\nabla^\perp H)(\bm x_{s(x)+t})\tilde\zeta_{t}^{\e}dt+[b(\bm x_{s(x)+t},\tbxi_{t}^{\e})-\nabla^\perp H(\bm x_{s(x)+t})]dt,~\tilde\zeta_{0}^{\e}=0.
\end{aligned}
\end{equation}
(Note that, for finite $t$, $\tilde\zeta_{t}^{\e}$ is of order $\sqrt{\e}$.) Then, by standard perturbation arguments and Gr\"onwall's inequality, it can be shown that, uniformly in $x$ such that $0\leq s(x)\leq1/2$, $0\leq t\leq 1-s(x)$, and $y\in\mathbb T^m$,
\begin{equation}
\label{eq:xtildecloseinL1}
    \E_{(x,y)}|\tbx_{t}^{\e}-\bm x_{s(x)+t}-\tilde\zeta_{t}^{\e}|=O(\e).
\end{equation}
Therefore, understanding of the distribution of $\tilde\zeta_{1-s(x)}^{\e}$ would help one to understand the distribution of $\tbx_{1-s(x)}^{\e}$.
However, it is not straightforward to study $\tilde\zeta_{t}^{\e}$ since $(\tbxi_{t}^{\e},\tilde\zeta_{t}^{\e})$ is not a Markov process. 
We introduce a related process $\zeta_{t}^{\e}$ defined using the original Markov process $\bxi_t^\e$, apply the local limit theorem to $( \bxi_{t}^{\e}, \zeta_{t}^{\e})$, and use the Girsanov theorem to get the desired estimate. Namely, let $\zeta_{t}^{\e}$, $s\leq t$, be defined by:
\begin{equation}
\label{eq:linearizeprocess}
\begin{aligned}
d \zeta_{t}^{\e}&=\nabla(\nabla^\perp H)(\bm x_{s(x)+t}) \zeta_{t}^{\e}dt+[b(\bm x_{s(x)+t}, \bxi_{t}^{\e})-\nabla^\perp H(\bm x_{s(x)+t})]dt,~ \zeta_{0}^{\e}=0.
\end{aligned}
\end{equation}
The following result is a version of the local limit theorem \cite{LLT} adapted to our case. 
\begin{theorem}
\label{thm:locallimittheorem}
    Let $g:[0,1]\times \mathbb T^m\to\mathbb R^2$ be a $C^\infty$ function such that $g(t,\cdot)$ spans $\mathbb R^2$ and $\int_{\mathbb T^m} g(t,y)d\mu(y)=0$ for all $t\geq 0$, where $\mu$ is the invariant measure of $ \bxi_{t}^{\e}$. 
    Then a local limit theorem holds for the following random variable as $\e\to0$ uniformly in $(x,y)\in I\times\mathbb T^m$,
    \[S^\e:=\frac{1}{\e}\int_0^{1-s(x)} g(s(x)+t, \bxi_{t}^{\e})dt.\]
    Namely, there exists an invertible covariance matrix $B(s)$ continuous in $s$ such that
    \begin{equation}
    \label{eq5:locallimittheorem}
        \lim_{\e\to0}\left| \frac{2\pi}{\e}\sqrt{\mathrm{det}B(s(x))}\cdot\Prob_{(x,y)}(S^\e-u\in[0,1)^2)-\exp({-\frac{\e \langle B(s(x))^{-1}u,u\rangle}{2}})\right|=0,
    \end{equation}uniformly in $u\in\mathbb R^2 $, $x\in I$, and $y\in\mathbb T^m$.
\end{theorem}
The second term in \eqref{eq5:locallimittheorem} is non-trivial even when $u$ takes large values (of order $1/\sqrt{\e}$), which is exactly the situation we are dealing with.
Following \eqref{eq:linearizeprocess}, we solve explicitly
\begin{equation}
     \zeta_{1-s(x)}^{\e}=\int_0^{1-s(x)}U_{s(x)+t,1}(b(\bm x_{s(x)+t}, \bxi_{t}^{\e})-\nabla^\perp H(\bm x_{s(x)+t}))dt,
\end{equation}
where $U_{t,s}$ solves the differential equation
\[
dU_{t,s}=\nabla(\nabla^\perp H)(\bm x_{s})U_{t,s}ds,
\]
and $U_{t,t}$ is the identity matrix.
Since $\bm x_t$ is deterministic, the integrand can be treated as a function only of time $t$ and $ \bxi_{t}^{\e}$. 
Moreover, for each $t$, the integrand has zero mean w.r.t. the invariant measure and spans $\mathbb R^2$, since $U_{t,1}$ is deterministic and non-singular and, for each $x$, $\{b(x,y)-\nabla^\perp H(x):y\in\mathbb T^m\}$ spans $\mathbb R^2$ by assumption \hyperlink{H4'}{\textit{(H4$'$)}}.
Then Theorem~\ref{thm:locallimittheorem} implies that
\begin{equation}
\label{eq:A_jk}
    \Prob_{(x,y)}\left(\frac{1}{\e} \zeta_{1-s(x)}^{\e}\in[j,j+1)\times[k,k+1)\right)\geq\frac{\e}{4\pi\sqrt{\mathrm{det}B(s(x))}}\exp\left(-\frac{\e \langle B(s(x))^{-1}(j,k),(j,k)\rangle}{2}\right)
\end{equation}
for all $\e$ small enough, $-1/\sqrt{\e}\leq j,k\leq 1/\sqrt{\e}$, $x\in I$, and  $y\in\mathbb T^m$. 
Finally, we compare $(\tbx_{t}^{\e},\tbxi_{t}^{\e},\tilde\zeta_{t}^{\e})$ with $( \bx_{t}^{\e}, \bxi_{t}^{\e}, \zeta_{t}^{\e})$. 
Since the added drift $c(x,y)$ in the equation of $\tbxi_{t}^{\e}$ is small compared to the diffusion term $\frac{1}{\sqrt{\e}}\sigma(y)$, it is not hard to verify that, using the Girsanov theorem, for all $\e$ small enough, $-1/\sqrt{\e}\leq j,k\leq 1/\sqrt{\e}$, $x\in I$, and  $y\in\mathbb T^m$, 
\begin{equation}
    \label{eq:A_jk1}
    \Prob_{(x,y)}\left(\frac{1}{\e}\tilde\zeta_{1-s(x)}^{\e}\in[j,j+1)\times[k,k+1)\right)\geq \frac{1}{2}\Prob_{(x,y)}\left(\frac{1}{\e} \zeta_{1-s(x)}^{\e}\in[j,j+1)\times[k,k+1)\right).
\end{equation}

\textbf{Step 3}. We proved that $\bm x_1+\tilde\zeta_{1-s(x)}^{\e}$ reaches the $O(\e)$-sized boxes with probabilities bounded from below. 
We also know that $\tbx_{1-s(x)}^{\e}$ is $O(\e)$-close to $\bm x_1+\tilde\zeta_{1-s(x)}^{\e}$ in $L^1$. 
Let us take one generic pair $(j,k)$, let $B^{\e,K}_{j,k}=\bm x_1+[(j-K)\e,(j+1+K)\e)\times[(k-K)\e,(k+1+K)\e)$, and study the distribution of $(\tbx_t^\e,\tbxi_t^\e)$ with the initial point in $B^{\e,K}_{j,k}\times\mathbb T^m$ after time of order $O(\e)$ {(see Figure \ref{fig:common_component})}.
\begin{figure}[hbtp]
    \centering
    \includegraphics[width=.5\textwidth]{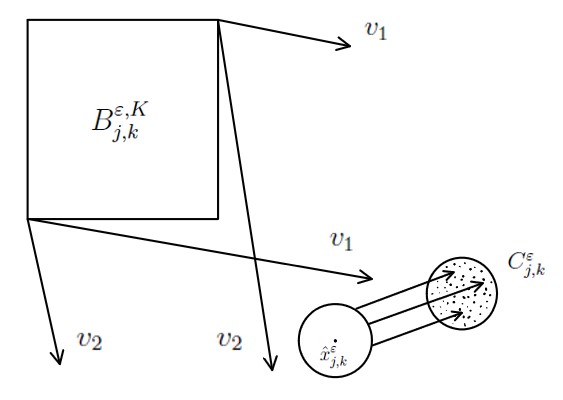}
    \caption{Common component of the distributions.}
    \label{fig:common_component}
\end{figure}
\begin{lemma}
\label{lemma:forcetosmallerbox}
    For each $\kappa>0$, $K>0$, and $\hat y\in\mathbb T^m$, there exist $t_2>0$, $c>0$, and, for each pair $(j,k)$, a point $\hat x_{j,k}^\e$ such that, for each $(x,y)\in B^{\e,K}_{j,k}\times\mathbb T^m$ and all $\e$ sufficiently small,
    $$\Prob_{(x,y)}(\tau_\kappa<t_2\e)\geq c,$$
    where $\tau_\kappa=\inf\{t:\tbx_t^\e\in B(\hat x_{j,k}^\e,\kappa\e),~\tbxi_t^\e\in B(\hat y,\kappa)\}$.
\end{lemma}
\begin{proof}
    Recall the definition of $\bm x_1$ at the beginning of Step 2 (see Figure~\ref{fig:local_limit_theorem}).
    By assumption \hyperlink{H4'}{\textit{(H4$'$)}}, $\{b(\bm x_1,y):y\in\mathbb T^m\}$ spans $\mathbb R^2$. So there exist $y_1,y_2\in\mathbb T^m$ such that $v_1:=b(\bm x_1,y_1)$ and $v_2:=b(\bm x_1,y_2)$ span $\mathbb R^2$. 
    Let us consider the set $S_{j,k}=\bigcap_{x\in B_{j,k}^{\e,K}}\{x+av_1+bv_2:a,b\geq0\}$.
    Then it is easy to see that, there exist a constant $t_2>0$ and, for each pair $(j,k)$, a point $\hat x_{j,k}^\e\in S_{j,k}$  such that for all $x\in B_{j,k}^{\e,K}$, $\hat x_{j,k}^\e=x+ a_x \e v_1+b_x \e v_2$ and $0<a_x,b_x<t_2/5$.
    There exists $\delta>0$ such that for each $x\in B(\bm x_1,2\delta)$ and each $y$ in $B(y_i,2\delta)$, $|b(x,y)-v_i|<\kappa/t_2$, $i=1,2$.
    Let $M$ be the upper bound of vector $b(x,y)$.
    For all $\e$ sufficiently small, the probability of the following event, denoted by $E$, has a lower bound, denoted by $c$, that only depends on $t_2$, $\kappa$, $M$, $y_1$, $y_2$, $\hat y$, $\delta$, and not on the starting point $(x,y)\in B(\bm x_1,\delta)\times\mathbb T^m$, thus not on $(j,k)$: 
    \[E=\begin{Bmatrix}
    \tau_1<(t_2\wedge \kappa/M)\e/5;~\tbxi^\e_{\tau_1+t}\in B(y_1,2\delta),~t\in[0,a_x\e];~\tau_2<\tau_1+a_x+(t_2\wedge \kappa/M)\e/5; \\
    \tbxi^\e_{\tau_2+t}\in B(y_2,2\delta),~t\in[0,b_x \e];~\tau_3<\tau_2+b_x\e+(t_2\wedge \kappa/M)\e/5
    \end{Bmatrix},
    \]where $\tau_1=\inf\{t\geq0:\tbxi^\e_t\in B(y_1,\delta)\}$, $\tau_2=\inf\{t\geq\tau_1+a_{x}\e:\tbxi^\e_t\in B(y_2,\delta)\}$, and $\tau_3=\inf\{t\geq\tau_2+b_{x}\e:\tbxi^\e_t\in B(\hat{y},\kappa)\}$. If $E$ is a subset of the event  $\{\tau_\kappa<t_2\e\}$, then the lemma is proved. To show the inclusion, note that on $E$,
    \begin{align*}
        |\tbx_{\tau_3}^\e-\hat x_{j,k}^\e|&=|\tbx_{\tau_3}^\e-(x+a_{x}\e v_1+b_{x}\e v_2)|\\
        &\leq |\tbx_{\tau_3}^\e-\tbx_{\tau_2+b_{x}\e}^\e|+|\tbx_{\tau_2+b_{x}\e}^\e-(\tbx_{\tau_2}^\e+ b_{x}\e v_2)|+|\tbx_{\tau_2}^\e-\tbx_{\tau_1+a_{x}\e}^\e|\\
        &\quad\quad +|\tbx_{\tau_1+a_{x}\e}^\e-(\tbx_{\tau_1}^\e+ a_{x}\e v_1)|+|\tbx_{\tau_1}^\e-x|\\
        &\leq \kappa\e.
    \end{align*}
    Besides, by the definition of $\tau_3$, $\tbxi_{\tau_3}^\e\in B(\hat{y},\kappa)$. Thus $\tau_\kappa\leq\tau_3<t_2\e$ on $E$.
\end{proof}
From now on, let $\hat y$ be the point in assumption \hyperlink{H5}{\textit{(H5)}} such that the parabolic H\"ormander condition holds at $(\bm x_1,\hat y)$ and let $p_t^\e((x,y),\cdot)$ be the density of $(\tbx_{t\e}^\e,\tbxi_{t\e}^\e)$ starting at $(x,y)$.
\begin{lemma}
\label{lem5:density_xy}
    There exists $\kappa>0$ such that for each $\hat x\in B(\bm x_1,\kappa)$ and all $\e$ sufficiently small, there is a domain $C^\e_{\hat x,\hat y}\subset V^\e\times\mathbb T^m$ with $\lambda(C^\e_{\hat x,\hat y})>\kappa\e^2$ and $p_{1}^\e((x,y),\cdot)>\kappa/\e^2$ on $C^\e_{\hat x,\hat y}$ for $(x,y)\in B(\hat x,\kappa\e)\times B(\hat y,\kappa)$.
\end{lemma}
\begin{proof}
Consider the stochastic processes that depend on the parameters $(\e,x,y)$:
\label{eq:theta_process}
\begin{equation}
    \label{eq5:difference}
    \begin{aligned}
    d\theta_t^{\e,x,y}&=b(x+\e\theta_t^{\e,x,y},y+\eta_t^{\e,x,y})dt,~\theta_0^{\e,x,y}=0\in\mathbb R^2,\\
    d\eta_t^{\e,x,y}&= v(y+\eta_t^{\e,x,y})dt+\e c(x+\e\theta_t^{\e,x,y},y+\eta_t^{\e,x,y})dt+\sigma(y+\eta_t^{\e,x,y}) dW_t,~\eta_0^{\e,x,y}={0}\in\mathbb R^m.
    \end{aligned}
\end{equation}
Since, by assumption \hyperlink{H5}{\textit{(H5)}}, the parabolic H\"ormander condition for equation \eqref{eq:theprocess1} holds at $(\bm x_1,\hat y)$, it is not hard to see that, if $(x,y)$ is close to $(\bm x_1,\hat y)$ and $\e$ is small, the parabolic H\"ormander condition holds for \eqref{eq5:difference} at $0$ and the distribution of $(\theta_t^{\e,x,y},\eta_t^{\e,x,y})$ is absolutely continuous w.r.t. the Lebesgue measure (\cite{Nualart}). 
Moreover, if the density function, denoted by $\tilde p_1^{\e,x,y}(\theta,\eta)$, exists, it is continuous in $\e,x,y,\theta$, and $\eta$. 
Let $\hat\theta$ and $\hat\eta$ satisfy that $\tilde p_1^{0,\bm x_1,\hat y}(\hat\theta,\hat\eta)>0$. 
Then there exists $0<\delta<1$ such that $\tilde p_1^{\e,x,y}(\theta,\eta)$ exists and is greater than $\delta$ for all $0<\e<\delta$, $x\in B(\bm x_1,\delta)$, $y\in B(\hat y,\delta)$, $\theta\in B(\hat\theta,\delta)$, and $\eta\in B(\hat\eta,\delta)$. 
For $\hat x\in B(\bm x_1,\delta/2)$, define $C^\e_{\hat x,\hat y}=B(\hat x+\e\hat\theta,\e\delta/2)\times B(\hat y+\hat\eta,\delta/2)$. 
Then, for $(x,y)\in B(\hat x,\e\delta/2)\times B(\hat y,\delta/2)$, and $(x',y')\in C^\e_{\hat x,\hat y}$, and $0<\e<\delta$, we have that 
\[p^\e_1((x,y),(x',y'))=\frac{1}{\e^2}\tilde p^{\e,x,y}\left(\frac{x'-x}{\e},y'-y\right)>\frac{\delta}{\e^2}.\]
The result holds with $\kappa=(\delta/2)^{m+2}$.
\end{proof}
\begin{lemma}
\label{lem:pijk}
    For each $K>0$, there exist constants $c>0$ and $t_1>0$ such that for all $-1/\sqrt{\e}\leq j,k\leq1/\sqrt{\e}$, there exists a measure $\pi^\e_{j,k}$ and a stopping time $\Tilde\eta_3^{j,k}<t_1\e$ such that for each $(x,y)\in B^{\e,K}_{j,k}\times\mathbb T^m$, the distribution of $(\tbx_{\Tilde\eta_3^{j,k}}^\e,\tbxi_{\Tilde\eta_3^{j,k}}^\e)$ starting at $(x,y)$ has $\pi^\e_{j,k}$ as a component and $\pi^\e_{j,k}(V^\e\times\mathbb T^m)>c$ for all $\e$ sufficiently small.
\end{lemma}
\begin{proof}
    We fix constant $\kappa>0$ such that the statements in Lemma~\ref{lem5:density_xy} hold. 
    Then, for the fixed $\kappa$, by Lemma~\ref{lemma:forcetosmallerbox}, we fix $t_2>0$, $c'>0$, and the point $\hat x_{j,k}^\e$ for each pair $(j,k)$ such that for all $(x,y)\in B^{\e,K}_{j,k}\times\mathbb T^m$ and $\e$ small, $\Prob_{(x,y)}(\tau_\kappa<t_2\e)\geq c'$, where $\tau_\kappa=\inf\{t:\tbx_t^\e\in B(\hat x_{j,k}^\e,\kappa\e),~\tbxi_t^\e\in B(\hat y,\kappa)\}$.
   It follows from Lemma~\ref{lem5:density_xy} that there is a domain $C^\e_{j,k}\subset V^\e\times\mathbb T^m$ with $\lambda(C^\e_{j,k})>\kappa\e^2$ and $p_{1}^\e((x,y),\cdot)>\kappa/\e^2$ on $C^\e_{j,k}$ for all $(x,y)\in B(\hat x_{j,k}^\e,\kappa\e)\times B(\hat y,\kappa)$. 
   Then the result follows if we define $c=c'\kappa^2$, $\pi^\e_{j,k}=c'\kappa/\e^2\cdot\chi_{\{C^\e_{j,k}\}}\lambda$, $t_1=t_2+2$, and $\Tilde\eta_3^{j,k}=\tau_{\kappa}\wedge t_2\e+\e<t_1\e$.
\end{proof}
Now let us combine Step 2 and Step 3 together to get the following result concerning the total variation distance of $(\tbx_{\bm\tau_1},\tbxi_{\bm\tau_1})$ with different starting points on $I\times\mathbb T^m$:
\begin{lemma}
\label{lem5:total_variation_on_I}
    For each $(x,y)\in I\times\mathbb T^m$, let $\tilde\mu_{x,y}^\e$ be the measure induced by $(\tbx_{\bm\tau_1},\tbxi_{\bm\tau_1})$ starting at $(x,y)$. Then there exists $c>0$ such that $\mathrm{TV}(\tilde\mu_{x,y}^\e,\tilde\mu_{x',y'}^\e)<1-c$ for any $(x,y),(x',y')\in I\times\mathbb T^m$ and all $\e$ sufficiently small.
\end{lemma}
\begin{proof}
    It suffices to show that there exist $c>0$ and a stopping time $\tilde\eta\leq\bm\tau_1$ such that the total variation distance of $(\tbx_{\tilde\eta},\tbxi_{\tilde\eta})$ with different starting points on $I\times\mathbb T^m$ is no more than $1-c$. 
    Recall the definitions of $s(x)$ and $\tilde\zeta_t^\e$ in Step 2.
    For the process $(\tbx_t^\e,\tbxi_t^\e)$ starting at $(x,y)\in I\times\mathbb T^m$, define
    \[A_{j,k}^{\e}=\{\frac{1}{\e}\tilde\zeta_{1-s(x)}^{\e}\in[j,j+1)\times[k,k+1)\},\]
    \[E_{K}^{\e}=\{|\tbx_{1-s(x)}^{\e}-\bm x_1-\tilde\zeta_{1-s(x)}^{\e}|> K\e\}\cup\{\sup_{0\leq t\leq 1-s(x)}|H(\tbx_{t}^{\e})|>K\sqrt{\e}\}.\]
    Using \eqref{eq:A_jk} and \eqref{eq:A_jk1}, we can find a constant $c'>0$ such that, for all $x\in I$, $y\in \mathbb T^m$, $\e$ sufficiently small, and $-1/\sqrt{\e}\leq j,k\leq1/\sqrt{\e}$, $\Prob_{(x,y)}(A_{j,k}^{\e})\geq c'\e$.
    And using \eqref{eq:slowx} and \eqref{eq:xtildecloseinL1} we can choose $K$ large enough such that, for all $x\in I$, $y\in \mathbb T^m$, and $\e$ sufficiently small, $\Prob_{(x,y)}(E_{K}^{\e})<c'/100$.
    Let $\tilde\eta_2=1-s(x)\wedge\bm\tau_1$.
    Then it is not hard to see that \[\sum_{-1/\sqrt{\e}\leq j,k\leq1/\sqrt{\e}}\Prob_{(x,y)}(A_{j,k}^\e\cap\{\tbx_{\tilde\eta_2}\not\in B^{\e,K}_{j,k}\})<c'/100.\]
    Now let us define, for $(x,y)\in I\times\mathbb T^m$,
    \[
    R_{x,y}^\e=\{(j,k):-1/\sqrt{\e}\leq j,k\leq1/\sqrt{\e},\Prob_{(x,y)}(A_{j,k}^\e\cap\{\tbx^\e_{\Tilde\eta_2}\in B^{\e,K}_{j,k}\})<c'\e/2\}.
    \]
    Then we know that $|R_{x,y}^\e|<\frac{1}{50\e}$ since, for every $(j,k)\in R_{x,y}^\e$, 
    \[\Prob_{(x,y)}(A_{j,k}^{\e}\cap \{\tbx^\e_{\Tilde\eta_2}\not\in B^{\e,K}_{j,k}\})\geq \Prob_{(x,y)}(A_{j,k}^{\e})-\Prob_{(x,y)}(A_{j,k}^\e\cap\{\tbx^\e_{\Tilde\eta_2}\in B^{\e,K}_{j,k}\})\geq c'\e/2.\]
    Let the constants $c''>0$, $t_1>0$, the stopping time $\Tilde\eta_3^{j,k}<t_1\e$, and $\pi^\e_{j,k}$ be defined as in Lemma~\ref{lem:pijk}. 
    Define
    \[
        \pi^\e=\frac{1}{2}c'\e\sum_{-1/\sqrt{\e}\leq j,k\leq1/\sqrt{\e}}\pi^\e_{j,k},~~~~~~~~~
        \hat\pi_{x,y,x',y'}^\e=\frac{1}{2}c'\e\sum_{(j,k)\in R^\e_{x,y}\cup R^\e_{x',y'}}\pi^\e_{j,k}.
    \]
    In order to define the desired stopping time, we first run the process starting on $I\times\mathbb T^m$ for time $\tilde\eta_2$ (with overwhelming probability, it is the time for the deterministic motion with the same stating point to reach $\bm x_1$).
    Then we use the locations of both $\tilde\zeta_{\tilde\eta_2}^\e$ and $\tx_{\tilde\eta_2}^\e$ to determine whether the process continues and, if it continues, we choose the stopping time based on Lemma~\ref{lem:pijk}.
    Namely, we define 
    \begin{equation}
        \Tilde\eta=\Tilde\eta_2+\sum_{-1/\sqrt{\e}\leq j,k\leq1/\sqrt{\e}}\chi(A_{j,k}^{\e}\cap\{\tbx^\e_{\Tilde\eta_2^{j,k}}\in B_{j,k}^{\e,K}\})\cdot \Tilde\eta_3^{j,k}(\tbx^\e_{\tilde\eta_2},\tbxi^\e_{\tilde\eta_2}),
    \end{equation}where $\Tilde\eta_3^{j,k}(x,y)$ denotes the stopping time with initial condition $(x,y)$.
    Then it follows from previous results that, for any pair $(x,y),(x',y')\in I\times\mathbb T^m$, there is a common component $\pi^\e-\hat\pi^\e_{x,y,x',y'}$ of the distributions of $(\tbx^\e_{\Tilde\eta},\tbxi^\e_{\Tilde\eta})$ starting from $(x,y)$ and $(x',y')$, respectively. Moreover, $(\pi^\e-\hat\pi^\e_{x,y,x',y'})(V^\e\times\mathbb T^m)>c'c''$ since
    $|R^\e_{x,y}|<\frac{1}{50\e}$ and $|R^\e_{x',y'}|<\frac{1}{50\e}$. Therefore, the total variation is no more than $1-c'c''$.
\end{proof}
Finally, we combine the result we just obtained with Step 1 to prove Lemma~\ref{lem5:expo_ergodicity}.
\begin{proof}[Proof of Lemma~\ref{lem5:expo_ergodicity}]
    As we discussed, the result is equivalent to the exponential convergence in total variation of $(\tbx_{\bm\tau_n}^\e,\tbxi_{\bm\tau_n}^\e)$ on $\gamma'\times\mathbb T^m$, uniformly in $\e$ and in the initial distribution. 
    Let $\mu^{\e}_{x,y}$ denote the measure on $\gamma'\times\mathbb T^m$ induced by $(\tbx_{\bm\tau_1}^\e,\tbxi_{\bm\tau_1}^\e)$ with the starting point $(x,y)\in\gamma'\times\mathbb T^m$.
    Then it suffices to prove that there exists $c>0$ such that, for every pair $(x,y)$, $(x',y')\in \gamma'\times\mathbb T^m$ and all $\e$ sufficiently small, $\mathrm{TV}(\mu^{\e}_{x,y},\mu^{\e}_{x',y'})<1-c$, which follows from Lemma~\ref{lem:rotation} and Lemma~\ref{lem5:total_variation_on_I}.
\end{proof}

\section{Proof of the main result}
\label{sec:proofofthemainresult}
Since we deal with both the original and the auxiliary processes in this section, certain notation needs clarifying to avoid possible ambiguity: the process $(\tx_t^\e,\txi_t^\e)$ represents not a generic process with arbitrary bounded $c(x,y)$ but only the auxiliary process with $\tilde c(x,y)$ satisfying \eqref{eq:added term}; 
$\sigma_n$, $\tau_n$ are defined as in \eqref{eq:stopping_times}{, but for the process $(X_t^\e,\xi_t^\e)$ on $M\times\mathbb T^m$}, and 
$\Tilde\sigma_n$, $\Tilde{\tau}_n$ represent the corresponding stopping time w.r.t. $(\tx_t^\e,\txi_t^\e)$. In \eqref{eq:definition_of_operator_AB}, we defined $\mathcal L_{c}$ on each edge for a generic $c(x,y)$. Here we give a more explicit definition of $\mathcal L_{\tilde c}=\tilde L_k$ on the edge $I_k$:
\begin{align*}
    \tilde L_kf(h)&=\frac{1}{2}A_k(h)f''(h)+\tilde B_k(h)f'(h),\\
    A_k(h)&=\frac{2}{Q_k(h)}\int_{\gamma_k(h)}\frac{1}{|\nabla H(x)|}\int_0^\infty\E_\mu b_h(x,\xi_s)b_h(x,\xi_0)dsdl,\\
    \Tilde B_k(h)&=\frac{1}{Q_k(h)}\int_{\gamma_k(h)}\frac{1}{|\nabla H(x)|}\int_0^\infty\E_\mu\mathrm{div}_x (b_h(x,\xi_s)(b(x,\xi_0)-\nabla^{\perp}H(x)))dsdl.
\end{align*}
One can easily check that this is consistent with the definitions of $\bar A$ and $\bar B_c$ in \eqref{eq:definition_of_operator_AB}, which are the generalizations of the coefficients defined in \eqref{def:llimit_operator} and, moreover,
\begin{equation}
\label{eq6:derivative}
    \frac{1}{2}[A_k(h_k)Q_k(h_k)f'(h_k)]'=\frac{1}{2}A_k(h_k)Q_k(h_k)f''(h_k)+\Tilde B_k(h_k)Q_k(h_k)f'(h_k).
\end{equation}
\begin{lemma}
\label{lem:zeroexpectation}
    For each $f\in \mathcal D$ and all $\e$ sufficiently small, we have $\Expe_{\nu^\e} \int_0^{\sigma_1}\mathcal L_{\Tilde c}f(h(\tx_t^\e))dt=0$.
\end{lemma}
\begin{proof}
    {This is the place where the gluing condition \eqref{eq2:gluing_condition} plays a role.}
    Since the process $(\tx_t^\e,\txi_t^\e)$ on $M\times\mathbb T^m$ is recurrent, and the measure $\lambda\times\mu$ is the invariant measure, by Theorem~2.1 in \cite{Khasminskii}, we have that for any measurable set $A\subset M$,
    \[
        \int_M\chi_A(x)d\lambda(x)=\lambda(A)=(\lambda\times\mu)(A\times\mathbb T^m)=\E_{\nu^\e}\int_0^{\sigma_1}\chi_A(\tx_t^\e)dt.
    \]
    Thus,
    \[
        \int_M \mathcal L_{\Tilde c}f(h(x))d\lambda(x)=\E_{\nu^\e}\int_0^{\sigma_1}\mathcal L_{\Tilde c}f(h(\tx_t^\e))dt.
    \]
    So, it suffices to show that the left-hand side is zero. By \eqref{eq6:derivative} and \eqref{eq2:gluing_condition},
\begin{align*}
       \int_{M}\mathcal L_{\Tilde c}f(h(x))d\lambda(x)&=\sum_{k=1}^3\int_{I_k} \tilde L_kf(h_k)Q_k(h_k)dh_k\\
       &=\sum_{k=1}^3\int_{I_k}(\frac{1}{2}A_k(h_k)f''(h_k)+\Tilde B_k(h_k)f'(h_k))Q_k(h_k)dh_k\\
       &=\sum_{k=1}^3\int_{I_k}\frac{1}{2}[A_k(h_k)Q_k(h_k)f'(h_k)]'dh_k\\
       &=\frac{1}{2}\sum_{k=1}^3p_k\lim_{h_k\to O}f'(h_k)\\
       &=0.\qedhere
\end{align*}
\end{proof}

Let us verify the analogue of \eqref{eq:mg_problem_M} in the case of the auxiliary process $(\tx_t^\e,\txi_t^\e)$.
\begin{proposition}
\label{prop:auxiliary}
For each $f\in \mathcal D$ and $T>0$,
    \begin{equation}
        \E_{(x,y)}[f(h(\tx_{\eta}^\e))-f(h(x))-\int_0^{\eta}\mathcal L_{\Tilde c}f(h(\tx_t^\e))dt]\to0
    \end{equation}
    as $\e\to0$, uniformly in $x\in M$, $y\in\mathbb T^m$, and $\eta$ is a stopping time bounded by $T$.
\end{proposition}
\begin{proof}
We divide the time interval $[0,\eta]$ into visits to the separatrix. Since $\sigma_n\to\infty$,
    \begin{align}
        &|\E_{(x,y)}[f(h(\tx_{\eta}^\e))-f(h(x))-\int_0^{\eta}\mathcal L_{\Tilde c}f(h(\tx_t^\e))dt]|\nonumber\\
        &\leq|\lim_{n\to\infty}\E_{(x,y)}[f(h(\tx_{\tilde\sigma_n}^\e))-f(h(x))-\int_0^{\tilde\sigma_n}\mathcal L_{\Tilde c}f(h(\tx_t^\e))dt]|\nonumber\\
        &\quad+|\lim_{n\to\infty}\E_{(x,y)}\E_{(\tx_{\eta}^\e,\txi_{\eta}^\e)}[f(h(\tx_{\tilde\sigma_n}^\e))-f(h(\tx_0^\e))-\int_0^{\tilde\sigma_n}\mathcal L_{\Tilde c}f(h(\tx_t^\e))dt]|\nonumber\\
        &\leq 2\sup_{(x,y)\in M\times\mathbb T^m}|\E_{(x,y)}[f(H(\tx^\e_{\tilde\sigma}))-f(H(x))-\int_0^{\sigma}{\mathcal L_{\tilde c}} f(H(\tx_s^\e))ds]|\label{eq6:to_separatrx}\\
        &\quad+2\lim_{n\to\infty}\sup_{(x,y)\in\gamma\times\mathbb T^m}|\E_{(x,y)}[f(h(\tx_{\tilde\sigma_n}^\e))-f(h(x))-\int_0^{\sigma_n}\mathcal L_{\tilde c}f(h(\tx_t^\e))dt]|.\label{eq6:between_separatrices}
    \end{align}
    Note that \eqref{eq6:to_separatrx} converges to $0$ due to Proposition~\ref{prop:up_to_separatrix}, and \eqref{eq6:between_separatrices} also converges to $0$ since
    \begin{align*}
        &\lim_{n\to\infty}\sup_{(x,y)\in\gamma\times\mathbb T^m}|\E_{(x,y)}[f(h(\tx_{\tilde\sigma_n}^\e))-f(h(x))-\int_0^{\tilde\sigma_n}\mathcal L_{\tilde c}f(h(\tx_t^\e))dt]|\\
        &\leq\lim_{n\to\infty}\sup_{(x,y)\in\gamma\times\mathbb T^m}\sum_{k=0}^{n-1}|\E_{(x,y)}\int_{\tilde\sigma_k}^{\tilde\sigma_{k+1}}\mathcal L_{\tilde c}f(h(\tx_t^\e))dt|\\
        &\leq\lim_{n\to\infty}\sup_{(x,y)\in\gamma\times\mathbb T^m}\sum_{k=0}^n \left(2\cdot\mathrm{TV}(\nu_{x,y}^{k,\e},\nu^\e)\cdot\sup_{(x',y')\in\gamma\times\mathbb T^m}|\E_{(x',y')}\int_0^{\tilde\sigma_1}\mathcal L_{\tilde c}f(h(\tx_t^\e))dt|\right)\\
        &=0,
    \end{align*}
    where the second inequality is due to Lemma~\ref{lem:zeroexpectation} and the last equality follows from Proposition~\ref{prop:up_to_separatrix}, Lemma~\ref{lem5:expo_ergodicity}, and Proposition~\ref{prop:exit_time_from_separatrix}.
    Thus, the desired result holds.
\end{proof}
{To generalize the result to the original process $(X_t^\e,\xi_t^\e)$ on $M\times\mathbb T^m$, we need the next two technical results.
We start with a simple corollary of Lemma~\ref{lem:number_excursion}, which controls the number of excursions or, equivalently, the number of stopping times $\sigma_n$ and $\tau_n$ in finite time.}
\begin{corollary}
\label{cor:num_excursion}
    For a given $t>0$, the expected number of excursions before $t$ is $O(\e^{-\alpha})$:
    \begin{equation}
    \label{eq:num_excursion}
    \sum_{n=0}^\infty\Prob_{(x,y)}(\tau_{n+1}<t)\leq\sum_{n=0}^\infty\Prob_{(x,y)}(\sigma_n<t)\leq\frac{e^t}{\kappa}\e^{-\alpha},
    \end{equation}where $\kappa$ is the constant chosen in Lemma~\ref{lem:number_excursion}.
\end{corollary}
\begin{proof}
    By Lemma~\ref{lem:number_excursion} and the strong Markov property,
\begin{equation}
    \sup_{(x,y)\in M\times\mathbb T^m}\E_{(x,y)} e^{-\sigma_n}\leq(\sup_{(x,y)\in\gamma'\times\mathbb T^m}\E_{(x,y)} e^{-\sigma})^n\leq(1-\kappa\e^\alpha)^n.
\end{equation}
Thus, by Markov's inequality, for all $n>0$,
\begin{equation}
    \Prob_{(x,y)}(\tau_{n+1}<t)\leq\Prob_{(x,y)}(\sigma_n<t)\leq e^t\E_{(x,y)} e^{-\sigma_n}\leq e^t(1-\kappa\e^\alpha)^n,
\end{equation}
and \eqref{eq:num_excursion} follows by taking the sum.
\end{proof}
\begin{lemma}
\label{lem:aux_to_ori}
    For each $f\in\mathcal D$ and $\delta>0$ there is $0<\rho<1$ such that, for all $x\in\gamma$, $y\in \mathbb T^m$, and all $\e$ sufficiently small,
    \begin{equation}
    \label{eq:auxiliary_to_main}
    \begin{aligned}
        &\sup_{{\sigma'}\leq\rho}|\E_{(x,y)}\sum_{n=0}^\infty\chi_{\{\sigma_n<{\sigma'}\}}[f(h(X_{\tau_{n+1}}^\e))-f(h(X_{\sigma_n}^\e))-\int_{\sigma_n}^{\tau_{n+1}}\mathcal Lf(h(X_s^\e))ds]|\\ 
        &\quad\leq\delta\rho+\e^\alpha\delta\sum_{n=0}^\infty\Prob_{(x,y)}(\sigma_n<\rho),
    \end{aligned}
    \end{equation}
    where $\sigma'$ is a stopping time w.r.t. $\mathcal F^{X_\cdot^\e}_t$.
\end{lemma}
\begin{proof}
    The result holds either with or without the integral term since nearly all of the time is spent from $\tau_n$ to $\sigma_n$.
    To be precise, by the strong Markov property, Corollary~\ref{cor:num_excursion}, and Proposition~\ref{prop:exit_time_from_separatrix},
    \begin{equation}
    \label{eq:integral_not_matter}
    \begin{aligned}
        &\sup_{(x,y)\in\gamma\times\mathbb T^m}\sup_{{\sigma'}\leq\rho}|\E_{(x,y)}\sum_{n=0}^\infty\chi_{\{\sigma_n<{\sigma'}\}}\int_{\sigma_n}^{\tau_{n+1}}\mathcal Lf(h(X_s^\e))ds|\\
        &\lesssim \sup_{(x,y)\in\gamma\times\mathbb T^m}\sup_{{\sigma'}\leq\rho}\sum_{n=0}^\infty|\E_{(x,y)}\chi_{\{\sigma_n<{\sigma'}\}}\E_{(X_{\sigma_n}^\e,\xi_{\sigma_n}^\e)}\tau_1|=O(\e^\alpha|\log\e|).
    \end{aligned}
    \end{equation}
    Thus, it suffices to prove for all $\e$ sufficiently small
    \begin{equation}
    \label{eq:without_integral}
        \sup_{{\sigma'}\leq\rho}|\E_{(x,y)}\sum_{n=0}^\infty\chi_{\{\sigma_n<{\sigma'}\}}[f(h((X_{\tau_{n+1}}^\e))-f(h((X_{\sigma_n}^\e))|\leq\delta\rho+\e^\alpha\delta\sum_{n=0}^\infty\Prob_{(x,y)}(\sigma_n<\rho).
    \end{equation}

    Let us prove this for $\tx_t^\e$ first using Proposition~\ref{prop:auxiliary}, then apply the Girsanov theorem to get the result for $X_t^\e$.
    Let $\tilde\sigma'$ be the analogue of $\sigma'$ w.r.t. $\mathcal F^{\tx_\cdot^\e}_t$.
    Divide the time interval $[0,\tilde\sigma']$ into excursions using stopping times $\tilde\sigma_n$ and $\tilde\tau_n$:
    \begin{align}
    &\E_{(x,y)}[f(h(\tx_{\Tilde\sigma'}^\e))-f(h(x))-\int_0^{\Tilde\sigma'}\mathcal Lf(h(\tx_t^\e))dt]\label{eq:auxleft}\\
    &\quad = \E_{(x,y)}[f(h(\tx_{\Tilde\sigma'\wedge\Tilde\sigma}^\e))-f(h(x))-\int_0^{\Tilde\sigma'\wedge\Tilde\sigma}\mathcal Lf(h(\tx_t^\e))dt]\label{eq:auxfirstsum}\\
    &\quad+\sum_{n=0}^\infty \E_{(x,y)}\left(\chi_{\{\Tilde\sigma_n<\Tilde\sigma'\}}[f(h(\tx_{\Tilde\tau_{n+1}\wedge \Tilde\sigma'}^\e))-f(h(\tx_{\Tilde\sigma_n}^\e))-\int_{\Tilde\sigma_n}^{\Tilde\tau_{n+1}\wedge \Tilde\sigma'}\mathcal Lf(h(\tx_t^\e))dt]\right)\label{eq:auxsecondsum}\\
    &\quad+\sum_{n=1}^\infty \E_{(x,y)}\left(\chi_{\{\Tilde\tau_n<\Tilde\sigma'\}}[f(h(\tx_{\Tilde\sigma_{n}\wedge \Tilde\sigma'}^\e))-f(h(\tx_{\Tilde\tau_n}^\e))-\int_{\Tilde\tau_n}^{\Tilde\sigma_{n}\wedge \Tilde\sigma'}\mathcal Lf(h(\tx_t^\e))dt]\right)\label{eq:auxthirdsum}.
\end{align}
    Thus, \eqref{eq:auxsecondsum} converges to $0$ uniformly for all $x\in\gamma$ and $\Tilde\sigma'\leq\rho$ due to the convergence of \eqref{eq:auxleft}, \eqref{eq:auxfirstsum}, and \eqref{eq:auxthirdsum}, by Proposition~\ref{prop:auxiliary}, Proposition~\ref{prop:up_to_separatrix}, and Lemma~\ref{lem:eps_avg_prin_to_sp} with Corollary~\ref{cor:num_excursion}, respectively. Note that \eqref{eq:integral_not_matter} also holds for $\tx_t^\e$, hence we conclude that
    \begin{equation}
    \label{eq:aux_stop_at_sigma'}
        \sup_{(x,y)\in\gamma\times\mathbb T^m}\sup_{{\Tilde\sigma'}\leq\rho}\sum_{n=0}^\infty \E_{(x,y)}\left(\chi_{\{\Tilde\sigma_n<\Tilde\sigma'\}}[f(h(\tx_{\Tilde\tau_{n+1}\wedge \Tilde\sigma'}^\e))-f(h(\tx_{\Tilde\sigma_n}^\e))]\right)
        \to0.
    \end{equation}
    
    To apply the Girsanov theorem, we choose a sufficiently small time interval and use the fact that the transition probability of $(X_t^\e,\xi_t^\e)$ is similar to that of $(\tx_t^\e,\txi_t^\e)$ in the sense that they are absolutely continuous with density close to $1$.
    More precisely, for any fixed $\delta'>0$, by the Girsanov theorem, we can choose a constant $\rho_1$ such that for all $0<\rho<\rho_1$,
    \begin{equation}
        \mu_{x,y}^\e\left(\left|\frac{d\Tilde\mu_{x,y}^\e}{d\mu_{x,y}^\e}-1\right|<\delta'\right)\geq 1-\rho^2,
    \end{equation}
    where $\mu_{x,y}^\e$ and $\Tilde\mu_{x,y}^\e$ are the measures on $\bm{\mathrm C}[0,\rho]$ induced by $(X_t^\e,\xi_t^\e)$ and $(\tx_t^\e,\txi_t^\e)$. Define
    \[C'=\left\{\left|\frac{d\Tilde\mu_{x,y}^\e}{d\mu_{x,y}^\e}-1\right|<\delta\right\}\subset\bm{\mathrm C}[0,\rho],~\Omega'=\{(X_t^\e,~t\in[0,\rho])\in C'\}.\]
    Note that the quantity in \eqref{eq:without_integral} primarily depends on the behavior of the processes on time interval $[0,\sigma']$ and event $\Omega'$. 
    Indeed, we can replace the stopping times $\tau_n$ by $\tau_n\wedge\sigma'$ with $O(\e^\alpha)$ errors.
    To replace $\Omega$ with $\Omega'$, we need several additional results that control the difference.
    
    As in Corollary~\ref{cor:num_excursion}, we fix $\kappa>0$ and choose a large constant $C>0$ independent of $\rho$ such that
    \begin{equation}
        \sum_{n=[C\log(C/\rho)\e^{-\alpha}]}^\infty\Prob_{(x,y)}(\sigma_n<\rho)\leq\sum_{n=[C\log(C/\rho)\e^{-\alpha}]}^\infty e^{\rho}(1-\kappa\e^\alpha)^n\leq\delta'\rho\e^{-\alpha}.
    \end{equation}
    Now we choose $\rho_2>0$ such that, for all $0<\rho<\rho_2$, $C\rho\log(C/\rho)<\delta'$. Hence, for all $\sigma'\leq\rho$,
    \begin{equation}
        \sum_{n=0}^\infty\Prob_{(x,y)}(\{\sigma_n<\sigma'\}\setminus\Omega')\leq C\rho^2\log(C/\rho)\e^{-\alpha}+\delta'\rho\e^{-\alpha}\leq2\delta'\rho\e^{-\alpha}.
    \end{equation}
    Thus, with $K:=\max_{I_k\sim O}|\lim_{{h_k\in I_k, h_k\to O}}f'(h_k)|$, we obtain
    \begin{equation}
    \label{eq6:Omega'}
        |\E_{(x,y)}\sum_{n=0}^\infty\chi_{\{\sigma_n<{\sigma'}\}\setminus\Omega'}[f(h(X_{\tau_{n+1}\wedge\sigma'}^\e))-f(h(X_{\sigma_n}^\e))]|\leq 2(K+1)\delta'\rho.
    \end{equation}
    By following the same steps, we can choose $\rho_3>0$ such that for all $0<\rho<\rho_3$,
    \begin{equation}
    \label{eq6:Omega'_tilde}
        |\E_{(x,y)}\sum_{n=0}^\infty\chi_{\{\tilde\sigma_n<{\tilde\sigma'}\}\setminus\Omega'}[f(h(\tx_{\tilde\tau_{n+1}\wedge\tilde\sigma'}^\e))-f(h(\tx_{\tilde\sigma_n}^\e))]|\leq 2(K+1)\delta'\rho.
    \end{equation}
    It remains to consider
    \begin{equation}
    \label{eq6:omega_expc}
        |\E_{(x,y)}\sum_{n=0}^\infty\chi_{\{\sigma_n<{\sigma'}\}\cap\Omega'}[f(h(X_{\tau_{n+1}\wedge\sigma'}^\e))-f(h(X_{\sigma_n}^\e))]|,
    \end{equation}
    which can be written and estimated as, with $F$ denoting the functional on $\bm{\mathrm C}[0,\rho]$ found inside the expectation in \eqref{eq6:omega_expc}, 
    \begin{equation}
    \begin{aligned}
        \left|\int_{C'} F d\mu_{x,y}^\e\right|&=\left|\int_{C'} Fd\Tilde\mu_{x,y}^\e-\int_{C'} F\left(\frac{d\Tilde\mu_{x,y}^\e}{d\mu_{x,y}^\e}-1\right) d\mu_{x,y}^\e\right|\\
        &\leq \left|\int_{C'} Fd\Tilde\mu_{x,y}^\e\right|+\delta'\int_{C'} |F|d\mu_{x,y}^\e.
    \end{aligned}
    \end{equation}
    The first term is bounded by $2(K+2)\delta'\rho$ due to \eqref{eq:aux_stop_at_sigma'} and \eqref{eq6:Omega'_tilde}. The second term is simply bounded by $(K+1)\delta'\e^\alpha\sum_{n=0}^\infty\Prob_{(x,y)}(\sigma_n<\rho)$. Thus, we see that the left-hand side in \eqref{eq:without_integral} is no more than $(4K+6)\delta'(\rho+\e^\alpha\sum_{n=0}^\infty\Prob_{(x,y)}(\sigma_n<\rho))$ with finite $K$ independent of $\delta'$ and $\rho$ for all $\e$ sufficiently small. It remains to take $\delta'=\delta/(4K+6)$.
\end{proof}

\begin{proof}[Proof of Proposition~\ref{prop:main_result}]
Fix arbitrary $\delta>0$. We divide the time interval $[0,\eta]$ into excursions from $\gamma$ to $\gamma'$ and from $\gamma'$ to $\gamma$ by using stopping times $\sigma_n$ and $\tau_n$:
\begin{align}
    &\E_{(x,y)}[f(h(X_{\eta}^\e))-f(h(x))-\int_0^{\eta}\mathcal Lf(h(X_t^\e))dt]\nonumber\\
    &\quad = \E_{(x,y)}[f(h(X_{\eta\wedge\sigma}^\e))-f(h(x))-\int_0^{\eta\wedge\sigma}\mathcal Lf(h(X_t^\e))dt]\label{eq:firstsum}\\
    &\quad+\sum_{n=0}^\infty \E_{(x,y)}\left(\chi_{\{\sigma_n<\eta\}}[f(h(X_{\tau_{n+1}\wedge \eta}^\e))-f(h(X_{\sigma_n}^\e))-\int_{\sigma_n}^{\tau_{n+1}\wedge \eta}\mathcal Lf(h(X_t^\e))dt]\right)\label{eq:secondsum}\\
    &\quad+\sum_{n=1}^\infty \E_{(x,y)}\left(\chi_{\{\tau_n<\eta\}}[f(h(X_{\sigma_{n}\wedge \eta}^\e))-f(h(X_{\tau_n}^\e))-\int_{\tau_n}^{\sigma_{n}\wedge \eta}\mathcal Lf(h(X_t^\e))dt]\right)\label{eq:thirdsum}.
\end{align}
Here \eqref{eq:firstsum} converges to $0$ by Proposition~\ref{prop:up_to_separatrix} and \eqref{eq:thirdsum} converges to $0$ by Lemma~\ref{lem:eps_avg_prin_to_sp} and Corollary~\ref{cor:num_excursion}. It remains to consider \eqref{eq:secondsum} and it suffices to consider instead
\begin{equation}
\label{eq:secondsum1}
    \sum_{n=0}^\infty \E_{(x,y)}\left(\chi_{\{\sigma_n<\eta\}}[f(h(X_{\tau_{n+1}}^\e))-f(h(X_{\sigma_n}^\e))-\int_{\sigma_n}^{\tau_{n+1}}\mathcal Lf(h(X_t^\e))dt]\right)
\end{equation}
because the difference converges to $0$ by Proposition~\ref{prop:exit_time_from_separatrix}.
By Proposition~\ref{prop:auxiliary}, we choose $0<\rho<1$ such that \eqref{eq:auxiliary_to_main} holds for $\delta$ and all $\e$ sufficiently small. We introduce the stopping times $\hat\sigma_n$ by letting $\hat\sigma_0=\sigma$ and $\hat\sigma_n$ be the first of $\sigma_k$ such that $\sigma_k-\hat\sigma_{n-1}\geq\rho$. It is clear that $\hat\sigma_{[T/\rho]}\geq T\geq\eta$. {Hence, we can replace \eqref{eq:secondsum1} by}
\begin{equation}
\label{eq:secondsum2}
    \sum_{n=0}^{[T/\rho]-1}\E_{(x,y)}(\chi_{\{\hat\sigma_n<\eta\}}\E_{(X_{\hat\sigma_n}^\e,\xi_{\hat\sigma_n}^\e)}\sum_{k=0}^\infty\chi_{\{\sigma_k<\rho\}}[f(h(X_{\tau_{n+1}}^\e))-f(h(X_{\sigma_n}^\e))-\int_{\sigma_n}^{\tau_{n+1}}\mathcal Lf(h(X_t^\e))dt])
\end{equation}
and, by the strong Markov property,  the difference is no more than 
\begin{equation}
\label{eq:secondsum3}
    \sup_{(x,y)\in\gamma\times\mathbb T^m}\sup_{{\sigma'}\leq\rho}|\E_{(x,y)}\sum_{n=0}^\infty\chi_{\{\sigma_n<{\sigma'}\}}[f(h((X_{\tau_{n+1}}^\e))-f(h((X_{\sigma_n}^\e))-\int_{\sigma_n}^{\tau_{n+1}}\mathcal Lf(h(X_t^\e))dt]|,
\end{equation}where $\sigma'$ is a stopping time w.r.t. $\mathcal F^{X_\cdot^\e}_t$.
Both of them can be bounded by $O(\delta)$ due to Lemma~\ref{lem:aux_to_ori} and Corollary~\ref{cor:num_excursion}.
\end{proof}

\appendix
\section{Derivatives of the action-angle-type coordinates}
\label{sec:derivatives}
In this section, we carefully estimate the first and second derivatives of $q(x)$, $Q(h)$, $\phi(x)$, $A(h,\phi)$, and ${B_c}(h,\phi)$ in order to prove \eqref{eq:bounds}. 
Our main tool is the Morse lemma.
Note that we only need to verify the bounds near the separatrix, since the derivatives are uniformly bounded inside the domain.
For $x,y\in\mathbb R^2$, let $x\to y$ denote the line segment connecting $x$ and $y$. For $x,y\in\gamma(h)$ for certain $h$, let $x\xrightarrow{\gamma} y$ denote the piece of $\gamma(h)$ connecting $x$ to $y$ along the direction of $\nabla^\perp H$.
Recall that $q(x)=\int_{l(H(x))\xrightarrow{\gamma} x}\frac{1}{|\nabla H|}dl$. 
To start with, using the Morse lemma, one can compute that $q(x)\lesssim |\log H(x)|$ and $Q(h)\lesssim|\log h|$.
Then we make use of two special deterministic motions in the directions of $\nabla^\perp H$ and $\nabla H$ to calculate the first derivatives of $q(x)$ precisely:
\begin{equation}
    \begin{aligned}
        d\bm x_t&=\nabla^\perp H(\bm x_t)dt,\\
        d\bm y_t&=\nabla H(\bm y_t)dt.
    \end{aligned}
\end{equation}
\begin{figure}[!ht] 
    \centering
    \includegraphics[width=0.4\textwidth]{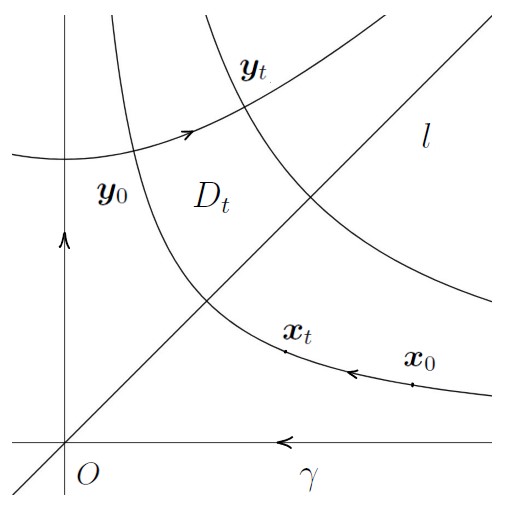}
    \caption{Motions tangent and perpendicular to the level curve.}
    \label{fig:coordinate}
\end{figure}
It follows that
\begin{equation}
\label{eqa:qxt}
    q(\bm x_t)
    =q(\bm x_0)+t,
\end{equation}
\begin{equation}
\label{eqa:qyt}
    \begin{aligned}
        q(\bm y_t)
        =q(\bm y_0)+\int_{\partial D_t}\frac{\nabla H}{|\nabla H|^2}\cdot\bm ndl
        =q(\bm y_0)+\int_{D_t}\mathrm{div}(\frac{\nabla H}{|\nabla H|^2})dS,
    \end{aligned}
\end{equation}
where $D_t$ is the region bounded by $l$, trajectory of $\bm {y}_s$, $0\leq s\leq t$, $\gamma(H(\bm y_0))$, and $\gamma(H(\bm y_t))$, as shown in Figure~\ref{fig:coordinate}.
Thus, by differentiating \eqref{eqa:qxt} and \eqref{eqa:qyt} in $t$, we have the following equations:
\begin{equation}
\label{eqa:linear_sys}
    \begin{aligned}
        \nabla q(x)\cdot\nabla^\perp H(x)&=1,\\
        \nabla q(x)\cdot\nabla H(x)&=|\nabla H(x)|^2\int_{l(H(x))\xrightarrow{\gamma} x}\mathrm{div}(\frac{\nabla H}{|\nabla H|^2})\frac{1}{|\nabla H|}dl.
    \end{aligned}
\end{equation}
Therefore, with subscripts denoting the partial derivatives, by solving the linear system,
\begin{equation}
\label{eqa:q1q2}
    \begin{aligned}
        q'_1&=\frac{-H'_2}{{H'_1}^2+{H'_2}^2}+H'_1p,\\
        q'_2&=\frac{H'_1}{{H'_1}^2+{H'_2}^2}+H'_2p,
    \end{aligned}
\end{equation}
where $p(x)=\int_{l(H(x))\xrightarrow{\gamma} x}\mathrm{div}(\frac{\nabla H}{|\nabla H|^2})\frac{1}{|\nabla H|}dl$. Using the Morse lemma, one can compute $p(x)=O(1/H(x))$, since
\begin{equation}
    \left|\mathrm{div}(\frac{\nabla H}{|\nabla H|^2})\right|\lesssim\frac{1}{|\nabla H|^2}.
\end{equation}
Note that the non-degeneracy of the saddle point implies that $|H|\lesssim |\nabla H|^2$, and hence $\nabla q=O(|\nabla H|/H)$. The next step is to estimate $\nabla p$. For all $x,y$ close enough, with $D$ denoting the region bounded by $l$, $x\to y$, $\gamma(H(x))$, and $\gamma(H(y))$, we have
\begin{align*}
    |p(x)-p(y)|&=\left|\int_{\partial D}\mathrm{div}(\frac{\nabla H}{|\nabla H|^2})\frac{\nabla H}{|\nabla H|^2}\cdot \mathbf{n}dl-\int_{x\to y}\mathrm{div}(\frac{\nabla H}{|\nabla H|^2})\frac{\nabla H}{|\nabla H|^2}\cdot \mathbf{n}dl\right|\\
    &\leq\int_D \left|\mathrm{div}\left[\mathrm{div}(\frac{\nabla H}{|\nabla H|^2})\frac{\nabla H}{|\nabla H|^2}\right]\right|dS+\int_{x\to y}\left|\mathrm{div}(\frac{\nabla H}{|\nabla H|^2})\frac{1}{|\nabla H|}\right|dl\\
    &\lesssim\int_D \frac{1}{|\nabla H|^4}dS+\int_{x\to y}\frac{1}{|\nabla H|^3}dl.
\end{align*}
Then one can obtain $|\nabla p|\lesssim|\nabla H|/H^2$ by using the Morse lemma again and, as a result, $|q''_{ij}|\lesssim |\nabla H|^2/H^2$, $1\leq i,j\leq 2$. Similarly, we can estimate the derivatives of $Q(h)$. In fact, $Q'(h)=\int_{\gamma(h)}\mathrm{div}(\frac{\nabla H}{|\nabla H|^2})\frac{1}{|\nabla H|}dl=O(1/h)$ because of the estimate we had on $p(x)$. In addition,
\begin{equation}
    \left|Q''(h)\right|=\left|\int_{\gamma(h)}\mathrm{div}\left[\mathrm{div}(\frac{\nabla H}{|\nabla H|^2})\frac{\nabla H}{|\nabla H|^2}\right]\frac{1}{|\nabla H|}dl\right|\lesssim\int_{\gamma(h)}\frac{1}{|\nabla H|^5}dl=O(1/h^2).
\end{equation}
Since $\phi(x)=q(x)/Q(H(x))$, $|\nabla\phi|\lesssim|\nabla H|/H$ and $|\phi''_{ij}|\lesssim|\nabla H|^2/H^2$, $1\leq i,j\leq2$. Finally, we estimate the derivatives w.r.t. $h$ of a general function $F(h,\phi)=\Tilde F(x_1,x_2)$ with $\Tilde F$ having bounded first and second derivatives. By computing the inverse of the Jacobian of $(H,\phi)$ and using the first equation in \eqref{eqa:linear_sys},
\begin{equation}
\label{eqa:der_in_h}
    F'_h=\frac{\Tilde F'_1\phi'_2-\Tilde F'_2\phi'_1}{H'_1\phi'_2-H'_2\phi'_1}=(\Tilde F'_1\phi'_2-\Tilde F'_2\phi'_1)Q(H(x_1,x_2)).
\end{equation}
We deduce that $F'_h=O(|\log h|/h)$ and, using $F_h'$ instead of $F$ in {\eqref{eqa:der_in_h}, $F''_{hh}=O(|\log h|^2/h^3)$.} Similarly, one can obtain $F'_\phi=O(|\log h|)$ and $F''_{\phi\phi}=O(|\log h|/h)$.

\section{Exit from neighborhoods of critical points}
\label{sec:exitfromneighborhood}
In this section, we obtain estimates for the exit time from the neighborhoods of the critical points, including extremum points and saddle points. Recall the notation in Section~\ref{sec:averaging}: $O$ is a saddle point, $O'$ is an extremum point, $U$ is a domain bounded by the separatrix, and $\eta(h)=\inf\{t:|H(\tx_t^\e)-H(O')|=h\}$. Recall the function $u$ defined in \eqref{eq:u} and let us define 
\begin{equation}
    \label{eqb:nondegenerate-diffusion}
    \bm A(x)=\int_{\mathbb T^m}\nabla_y u(x,y)\sigma(y)\sigma(y)^\mathsf{T}\nabla_y u(x,y)^\mathsf{T}d\mu(y).
\end{equation}
Using assumption \hyperlink{H4'}{\textit{(H4$'$)}}, one can see that $\bm A(x)$ is positive-definite. 
Indeed, if there exist a point $x$ and a vector $v\not=0$ such that $v^\mathsf{T}\bm A(x)v=0$, then, since $\sigma\sigma^\mathsf{T}$ is positive-definite, $v^\mathsf{T}\nabla_y u(x,y)=0$ for all $y\in\mathbb T^m$.
Namely, $v^\mathsf{T} u(x,y)$ is constant, and $v^\mathsf{T}(b(x,y)-\bar b(x))=L(v^\mathsf{T}u(x,y))=0$ on $\mathbb T^m$, which contradicts with \hyperlink{H4'}{\textit{(H4$'$)}}. 
\begin{lemma}
\label{lemb:non-zero-drift}
    Recall the definition of $B_c(x)$ in \eqref{eq:definition_of_operator_AB_x}. If $O'$ is a minimum point, then $B_c(O')>0$; if $O'$ is a maximum point, then $B_c(O')<0$.
\end{lemma}
\begin{proof}
We prove the result in the case of minimum point. The other case can be treated similarly. Since $O'$ is a critical point and $Lu(x,y)=-(b(x,y)-\bar b(x))$, we have $\nabla H(O')=0$ and 
    \begin{align*}
        {B_c}(O')&=\int_{\mathbb T^m}[\nabla_xu_h(O',y)b(O',y)+\nabla_y u_h(O',y)c(x,y)]d\mu(y)\\
        &=\int_{\mathbb T^m}u(O',y)^\mathsf{T}\nabla^2 H(O')b(O',y)d\mu(y)\\
        &=-\int_{\mathbb T^m}u(O',y)^\mathsf{T}\nabla^2 H(O')Lu(O',y)d\mu(y)\\
        &=-\sum_{1\leq i,j\leq2}\frac{\partial^2}{\partial x_i\partial x_j}H(O')\int_{\mathbb T^m}u_i(O',y)Lu_j(O',y)d\mu(y)\\
        &=-\sum_{1\leq i,j\leq2}\frac{\partial^2}{\partial x_i\partial x_j}H(O')\int_{\mathbb T^m}\frac{1}{2}(u_iLu_j+u_jLu_i)(O',y)d\mu(y)\\
        &=\frac{1}{2}\sum_{1\leq i,j\leq2}\frac{\partial^2}{\partial x_i\partial x_j}H(O')\left(\bm A(O')_{i,j}-\int_{\mathbb T^m}L(u_iu_j)d\mu(y)\right)\\
        &=\frac{1}{2}\sum_{1\leq i,j\leq2}\frac{\partial^2}{\partial x_i\partial x_j}H(O')\bm A(O')_{i,j}.
\end{align*}
This is positive since both $\nabla^2 H$ and $\bm A$ are positive definite at $O'$.
\end{proof}
\begin{lemma}
\label{lem:near_extremum}
    For each $\kappa>0$, there exists $\delta>0$ such that
    \begin{equation}
    \label{eq:near_extremum}
        \E_{(x,y)}\eta(\delta)\leq\kappa
    \end{equation}
    for all $x$ satisfying $|H(x)-H(O')|<\delta$, $y\in\mathbb T^m$, and $\e$ sufficiently small.
\end{lemma}
\begin{proof}
Without loss of generality, we assume $O'$ to be a minimum point. Similarly to \eqref{eq:H}, we apply Ito's formula to $u_h(\tx^\e_{\eta(\delta)\wedge1},\txi^\e_{\eta(\delta)\wedge1})$ and we obtain
\begin{align*}
        H(\tx^\e_{\eta(\delta)\wedge1})&=H(x)+\int_0^{\eta(\delta)\wedge1} \nabla_y u_h(\tx_s^\e,\txi_s^\e)^{\mathsf T}\sigma(\txi_s^\e)dW_s+\e(u_h(x,y)-u_h(\tx_{\eta(\delta)\wedge1}^\e,\txi_{\eta(\delta)\wedge1}^\e))\\
    &\quad+\int_0^{\eta(\delta)\wedge1}[\nabla_x u_h(\tx_s^\e,\txi_s^\e)\cdot b(\tx_s^\e,\txi_s^\e)+\nabla_y u_h(\tx_s^\e,\txi_s^\e)\cdot c(\tx_s^\e,\txi_s^\e)]ds.
    \end{align*}
By taking the expectation on both sides and using Corollary~\ref{cor:avg}, we obtain
\begin{equation}
\label{eqb:exit_from_extremum}
    \E_{(x,y)}\int_0^{\eta(\delta)\wedge1}B_c(\tx_s^\e)ds<\delta+O(\e).
\end{equation}
Due to Lemma~\ref{lemb:non-zero-drift}, $B_c(O')>0$. {Hence, we can choose $\delta$ to be small enough such that $B_c(\tx_s^\e)>B_c(O')/2>0$ before $\eta(\delta)$.} Thus, it follows from \eqref{eqb:exit_from_extremum} that, for all $\e$ sufficiently small,
\[\E_{(x,y)}(\eta(\delta)\wedge1)\leq 4\delta/B_c(O').\]
Then $\Prob_{(x,y)}(\eta(\delta)\geq1)\leq4\delta/B_c(O')$ for all $x$ satisfying $|H(x)-H(O')|<\delta$, $y\in\mathbb T^m$ by Markov's inequality. Then, one can obtain the desired result using the Markov property by the fact that $\E_{(x,y)}\eta(\delta)\leq \E_{(x,y)}(\eta(\delta)\wedge1)+\E_{(x,y)}(\eta(\delta)\chi_{\{\eta(\delta)>1\}})$.
\end{proof}

\begin{proposition}
\label{prop:exit_time_from_separatrix}
    Let $0<\alpha<1/2$, $V^\e=\{x:|H(x)-H(O)|<\e^\alpha\}$, and $\tau=\inf\{t:\tx_t^\e\not\in V^\e\}$. Then, uniformly in $0<\alpha<1/2$ and $(x,y)\in V^\e\times\mathbb T^m$, as $\e\downarrow0$,
    \begin{equation}
        \E_{(x,y)}\tau=O(\e^{2\alpha}|\log\e|).
    \end{equation}
\end{proposition} 
To prove Proposition~\ref{prop:exit_time_from_separatrix}, it is more convenient to consider the process $(\tbx_t^\e,\tbxi_t^\e)$, define the stopping time $\bm\tau=\inf\{t:\tbx_t^\e\not\in V^\e\}$, and then prove that $\E_{(x,y)}\bm\tau=O(\e^{2\alpha-1}|\log\e|)$.
We need careful analysis of the behavior of the processes near the saddle point and away from the saddle point. 
The latter is easier to deal with since there is no degeneracy, while the former needs us to, again, use the Morse Lemma to make concrete computations. 
For simplicity, we prove the result in the case of $(\bx_t^\e,\bxi_t^\e)$ without the additional drift $c(x,y)$ since it can be seen in the proof that the additional terms induced by $c(x,y)$ are always relatively small. 
We prove that there exist two neighborhoods, $D_1\subset D_2$, of $O$ (as shown in Figure~\ref{fig:D1D2}), such that, in $V^\e$, it takes the process $O(|\log\e|)$ time to escape from $D_2$, and $O(1)$ time to return to $D_1$.
Since $x\in V^\e$ is two-dimensional, we denote $x=(p,q)$. To avoid confusion brought by convoluted formulas, we assume the saddle point is the origin and the Hamiltonian $H(x)=pq$ in a small neighborhood of $O$.
    This assumption is not restrictive because, as in the proof of Lemma~\ref{lem:throughsaddlepoint}, we can use the Morse Lemma and perform a random change of time with the multiplier bounded from below and above, which will not change the order of the expected time. 
    For $r>0$, we denote $D_r$ to be the region in $V^\e$ with $|p|\leq r$ and $|q|\leq r$, $(\partial D_r)_{\textrm{in}}=\{|p|=r\}\cap V^\e$, $(\partial D_r)_{\textrm{out}}=\{|q|=r\}\cap V^\e$, and choose $r_3>0$ small enough that $H(x)=pq$ in $D_{r_3}$.
\begin{figure}[!htbp]
\centering
\begin{subfigure}{.5\textwidth}
  \centering
  \includegraphics[width=.9\linewidth]{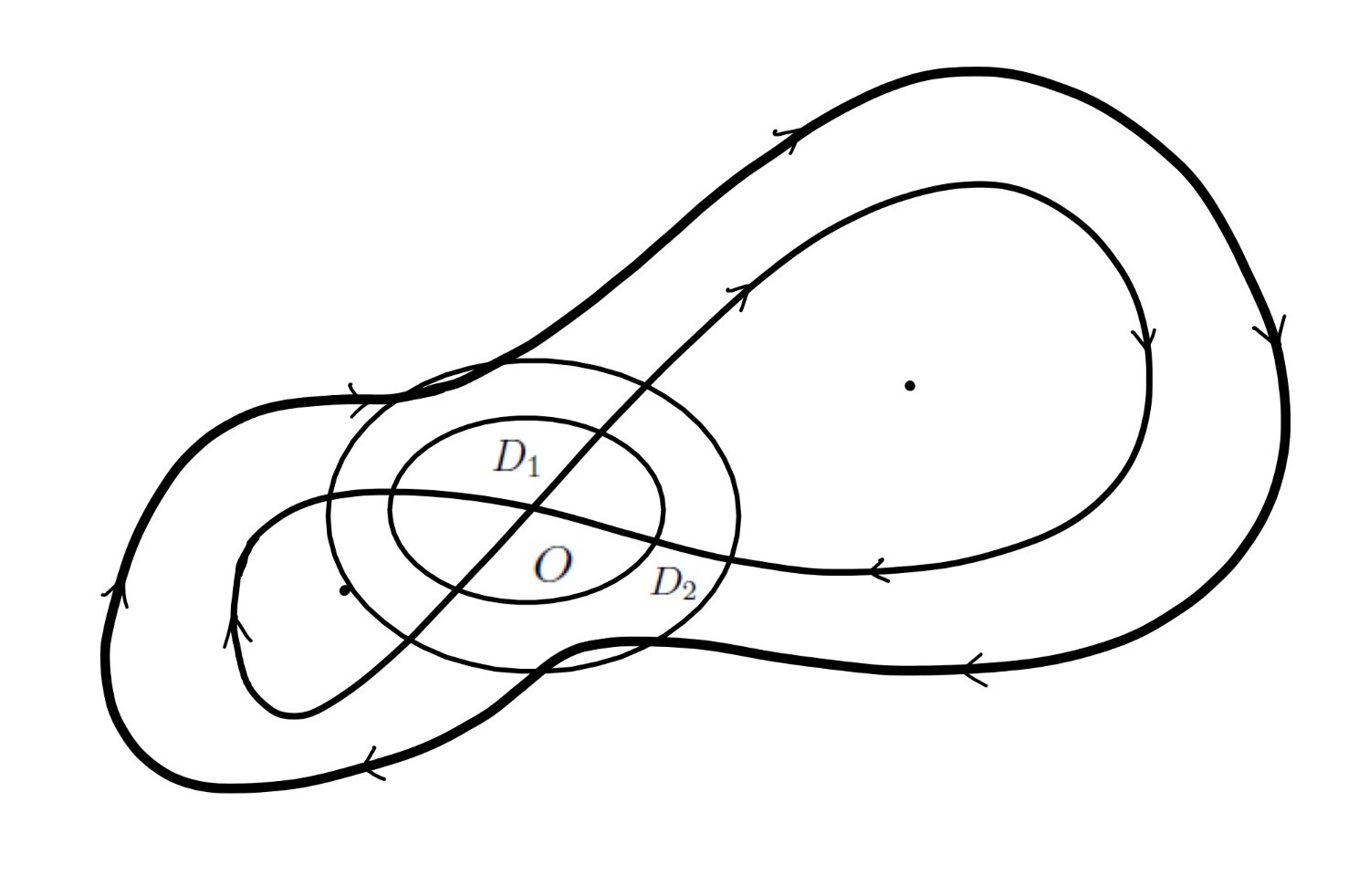}
  \caption{Two neighborhoods of $O$.}
  \label{fig:D1D2}
\end{subfigure}%
\begin{subfigure}{.5\textwidth}
  \centering
  \includegraphics[width=.9\linewidth]{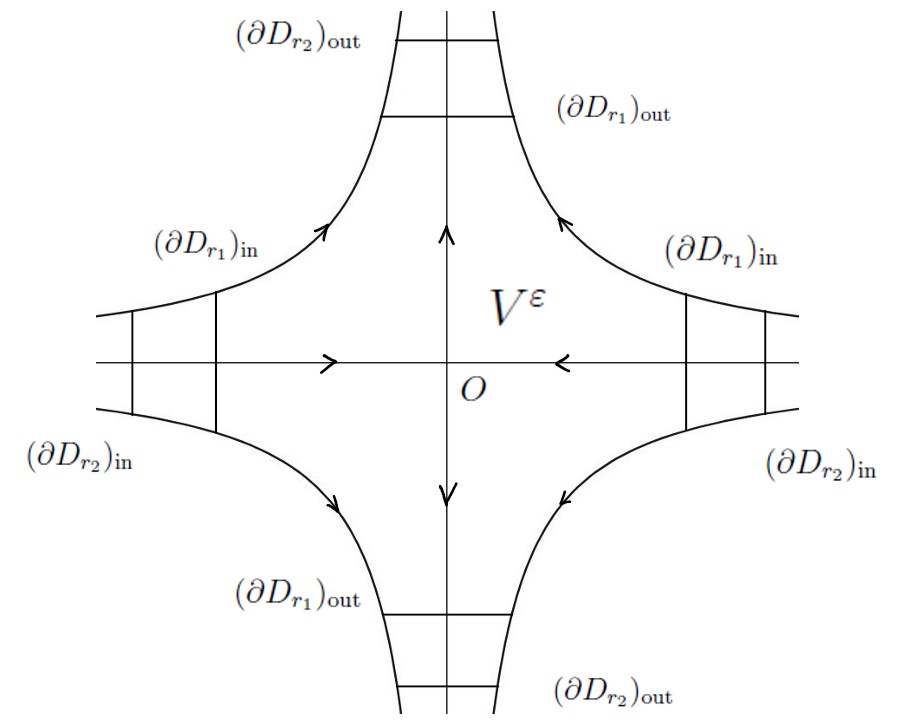}
  \caption{$D_{r_1}$ and $D_{r_2}$.}
  \label{fig:MorseLemma}
\end{subfigure}
\caption{Transitions between two boundaries.}
\label{fig:appendixB}
\end{figure}
\begin{lemma}
\label{lem:near}
    There exist $r_1,r_2>0$ such that, uniformly in $(x,y)\in D_{r_1}\times\mathbb T^m$, as $\e\downarrow0$,
    \begin{equation}
        \E_{(x,y)}\bar\tau=O(|\log\e|),
    \end{equation}
    where $\bar\tau=\inf\{t:\bx_t^\e\not\in D_{r_2}\}$.
\end{lemma}
\begin{proof}
    We denote $\eta_t^\e=(\bx_t^\e)_2$ and study the behavior of $\eta_t^\e$ inside $B(0,r_3)$. All the computations below concern $\bx_t^\e$ before leaving $D_{r_3}$. As in \eqref{eq:slowx}, we can write the equation for $\eta_t^\e$,
    \begin{equation}
    \begin{aligned}
    \label{eq:etat}
        {\eta}_t^\e&=q+\int_0^t {\eta}_s^\e ds+\sqrt{\e}\int_0^t\nabla_y u_2(\bx_s^\e,\bxi_s^\e)^{\mathsf T}\sigma(\bxi_s^\e)dW_s\\
        &\quad+\e\int_0^t\nabla_x u_2(\bx_s^\e,\bxi_s^\e)\cdot b(\bx_s^\e,\bxi_s^\e)ds+\e(u_2(x,y)-u_2(\bx_t^\e,\bxi_t^\e)).
    \end{aligned}
    \end{equation}
    Introduce $\hat\eta_t^\e$, which is close to $\eta_t^\e$:
    \begin{equation}
        \label{eq:eta}
        \hat\eta_t^\e = q+\int_0^t {\eta}_s^\e ds+\sqrt{\e}\int_0^t\nabla_y u_2(\bx_s^\e,\bxi_s^\e)^{\mathsf T}\sigma(\bxi_s^\e)dW_s+\e\int_0^t\nabla_x u_2(\bx_s^\e,\bxi_s^\e)\cdot b(\bx_s^\e,\bxi_s^\e)ds.
    \end{equation}
    Let $F(q)=\int_0^q e^{-z^2}\int_0^{z}e^{w^2}dwdz$, which satisfies $ F(q)\sim \frac{1}{2}\log q$, as $p\to\infty$, and has bounded derivatives up to the third order.
    Then we can choose a large $C>0$, such that $|F'|,|F''|,|F'''|$, $|u(x,y)|,|\nabla u(x,y)|,|\nabla^2 u(x,y)|$ are bounded by $C$.
    Recall from \eqref{eqb:nondegenerate-diffusion} that the vector-valued function $\nabla_y u_2(x,y)^{\mathsf T}\sigma(y)$ has non-degenerate average w.r.t. $\mu$, in the sense that $\bm A_{22}(x)>0$. 
    Let $A_0=\bm A_{22}(O)>0$ and let $0<r_1<r_2<r_3$ be such that $A_0(1-1/(2C))<\bm A_{22}(x)<A_0(1+1/(2C))$ in $D_{r_2}$, as shown in Figure~\ref{fig:MorseLemma}.
    Let $\bar r_2=\frac{r_3+r_2}{2}$. Define function $f(q)$ (that depends on $\e$) and compute its derivatives:
    \begin{equation}
    \label{eqb:derivatives_f}
    \begin{aligned}
        f(q)=2(F(\frac{\bar r_2}{\sqrt{A_0\e}})-F(\frac{q}{\sqrt{A_0\e}})),~~~&f'(q)=-\frac{2}{\sqrt{A_0\e}}F'(\frac{q}{\sqrt{A_0\e}}),\\
        f''(q)=-\frac{2}{A_0\e}F''(\frac{q}{\sqrt{A_0\e}}),~~~~~~~~~~~~~~&f'''(q)=-\frac{2}{(\sqrt{A_0\e})^3}F'''(\frac{q}{\sqrt{A_0\e}}).
    \end{aligned}
    \end{equation}
    Furthermore, $f$ satisfies the differential equations:
    \begin{equation}
    \label{eqb:equation_f}
        \begin{cases}
        \frac{1}{2}A_0\e f''+qf'=-1\\
        f(-\bar r_2)=f(\bar r_2)=0
        \end{cases}.
    \end{equation}
    By Lemma~\ref{lem:solution}, there is a function $v_2(x,y)$ that is bounded together with its derivatives such that
    \begin{equation}
        Lv_2(x,y)=\left|\nabla_y u_2(x,y)\sigma(y)\right|^2-\bm A_{22}(x).
    \end{equation}
    where $L$ is the generator of the process $\xi_t$ (see \eqref{eqb:def_operator_L}). 
    Since $|\eta_t^\e-\hat\eta_t^\e|=O(\e)$ and $\bar r_2>r_2$, we can apply Ito's formula to $v_2(\bx_t^\e,\bxi_t^\e)f''(\eta_t^\e)$ for $0\leq t\leq\bar\tau$:
    \begin{equation}
    \label{eq:avg_step_eg}
        \begin{aligned}
            v_2(\bx_t^\e,\bxi_t^\e)f''(\eta_t^\e)=&v_2(x,y)f''(q)+\int_0^t\nabla_x(v_2(\bx_s^\e,\bxi_s^\e)f''(\eta_s^\e))\cdot b(\bx_s^\e,\bxi_s^\e)ds\\
            &+\frac{1}{\e}\int_0^t Lv_2(\bx_s^\e,\bxi_s^\e)f''(\eta_s^\e)ds+\frac{1}{\sqrt{\e}}\int_0^t\nabla_yv_2(\bx_s^\e,\bxi_s^\e)^{\mathsf T}f''(\eta_s^\e)\sigma(\bxi_s^\e)dW_s.
        \end{aligned}
    \end{equation}
    Hence it follows that
    \begin{equation}
    \label{eqb:averaged_a22}
    \begin{aligned}
        &\int_0^t (\left|\nabla_y u_2(\bx_s^\e,\bxi_s^\e)\sigma(\bxi_s^\e)\right|^2-\bm A_{22}(\bx_s^\e))f''(\eta_s^\e)ds\\
        &\quad=\e\left(v_2(\bx_t^\e,\bxi_t^\e)f''(\eta_t^\e)-v_2(x,y)f''(q)\right) -\e\int_0^t\nabla_x(v_2(\bx_s^\e,\bxi_s^\e)f''(\eta_s^\e))\cdot b(\bx_s^\e,\bxi_s^\e)ds\\
        &\quad\quad -\sqrt{\e}\int_0^tf''(\eta_s^\e)\nabla_yv_2(\bx_s^\e,\bxi_s^\e)^{\mathsf T}\sigma(\bxi_s^\e)dW_s\\
        &\quad= O(1)+O(\frac{1}{\sqrt{\e}})\cdot t-\sqrt{\e}\int_0^tf''(\eta_s^\e)\nabla_yv_2(\bx_s^\e,\bxi_s^\e)^{\mathsf T}\sigma(\bxi_s^\e)dW_s.
    \end{aligned}
    \end{equation}
    Now apply Ito's formula to $f(\hat\eta_t^\e)$ for $0\leq t\leq \bar\tau$:
    \begin{align*}
        f(\hat\eta_t^\e)&=f(q)+\int_0^tf'(\hat\eta_s^\e)\eta_s^\e ds+\frac{\e}{2}\int_0^t f''(\hat\eta_s^\e)|\nabla_y u_2(\bx_s^\e,\bxi_s^\e)^{\mathsf T}\sigma(\bxi_s^\e)|^2ds\\
        &\quad+\e\int_0^t f'(\hat\eta_s^\e)\nabla_x u_2(\bx_s^\e,\bxi_s^\e)\cdot b(\bx_s^\e,\bxi_s^\e)ds+\sqrt{\e}\int_0^tf'(\hat\eta_s^\e)\nabla_y u_2(\bx_s^\e,\bxi_s^\e)^{\mathsf T}\sigma(\bxi_s^\e)dW_s\\
        &= f(q)+\int_0^tf'(\hat\eta_s^\e)\hat\eta_s^\e ds+O(\sqrt{\e})\cdot t+\frac{\e}{2}\int_0^t f''(\hat\eta_s^\e)\bm A_{22}(\bx_s^\e)ds\\
        &\quad+\frac{\e}{2}\int_0^t (|\nabla_y u_2(\bx_s^\e,\bxi_s^\e)^{\mathsf T}\sigma(\bxi_s^\e)|^2-\bm A_{22}(\bx_s^\e))f''(\eta_s^\e)ds+O(\sqrt{\e})\cdot t\\
        &\quad+\sqrt{\e}\int_0^tf'(\hat\eta_s^\e)\nabla_y u_2(\bx_s^\e,\bxi_s^\e)^{\mathsf T}\sigma(\bxi_t^\e)dW_s\\
        &= f(q)+\int_0^t [f'(\hat\eta_s^\e)\hat\eta_s^\e+\frac{A_0\e}{2}f''(\hat\eta_s^\e)]ds+\frac{\e}{2}\int_0^t f''(\hat\eta_s^\e)(\bm A_{22}(\bx_s^\e)-A_0)ds\\
        &\quad+\frac{\e}{2}( O(1)+O(\frac{1}{\sqrt{\e}})\cdot t-\sqrt{\e}\int_0^tf''(\eta_s^\e)\nabla_yv_2(\bx_s^\e,\bxi_s^\e)^{\mathsf T}\sigma(\bxi_s^\e)dW_s)+O(\sqrt{\e})\cdot t\\
        &\quad+\sqrt{\e}\int_0^tf'(\hat\eta_s^\e)\nabla_y u_2(\bx_s^\e,\bxi_s^\e)^{\mathsf T}\sigma(\bxi_s^\e)dW_s\\
        &\leq f(q)-t+\frac{t}{2}+O(\e)+O(\sqrt{\e})\cdot t-\frac{\sqrt{\e^3}}{2}\int_0^tf''(\eta_s^\e)\nabla_yv_2(\bx_s^\e,\bxi_s^\e)^{\mathsf T}\sigma(\bxi_s^\e)dW_s\\
        &\quad+\sqrt{\e}\int_0^tf'(\hat\eta_s^\e)\nabla_y u_2(\bx_s^\e,\bxi_s^\e)^{\mathsf T}\sigma(\bxi_s^\e)dW_s,
    \end{align*}where the equalities follow from \eqref{eqb:derivatives_f} and \eqref{eqb:averaged_a22}, and the last inequality holds since $f$ solves \eqref{eqb:equation_f} and $|\bm A_{22}(\bx_s^\e)-A_0|<A_0/(2C)$.
    Let $\Tilde \tau=\bar\tau\wedge 1/\e$. Then $\hat\eta_{\Tilde\tau}^\e\in[-\bar r_2,\bar r_2]$ because $|\hat\eta_{\Tilde\tau}^\e-\eta_{\Tilde\tau}^\e|=O(\e)$. The previous calculation reduces to
    \begin{align*}
        f(\hat\eta_{\Tilde\tau}^\e)&\leq f(q)-\Tilde\tau/2+O(\e)+O(\sqrt\e)\cdot\Tilde\tau-\frac{\sqrt{\e^3}}{2}\int_0^{\Tilde\tau}f''(\eta_s^\e)\nabla_yv_2(\bx_s^\e,\bxi_s^\e)^{\mathsf T}\sigma(\bxi_s^\e)dW_s\\
        &\quad+\sqrt{\e}\int_0^{\Tilde{\tau}}f'(\hat\eta_s^\e)\nabla_y u_2(\bx_s^\e,\bxi_s^\e)^{\mathsf T}\sigma(\bxi_s^\e)dW_s.
    \end{align*}
    By taking the expectation, we have for all $x\in D_{r_2}$, $y\in\mathbb T^m$, and $\e$ small enough
    \begin{equation}
        \E_{(x,y)}\Tilde{\tau}\leq 5\sup_{-\bar r_2\leq q'\leq\bar r_2}f(q')=O(|\log\e|).
    \end{equation}
    Then, by Markov's inequality and the Markov property, we obtain that $\E_{(x,y)}\bar\tau=O(|\log\e|)$.
\end{proof}
\begin{lemma}
\label{lem:exit_location}
    Let $r_1,r_2,\bar\tau$ be defined as in Lemma~\ref{lem:near}. Then, uniformly in $(x,y)\in D_{r_1}\times\mathbb T^m$,
    \begin{equation}
        \Prob_{(x,y)}(\bx_{\bar\tau}^\e\in(\partial D_{r_2})_{\mathrm{in}})\to0~\text{as }\e\downarrow0.
    \end{equation}
\end{lemma}
\begin{proof}
    Again, we denote $x=(p,q)$ and, for simplicity, we assume that the saddle point is the origin and that $H(x) = pq$ in a small neighborhood of $O$.
    We extend the function $b(x,y)$ in the vertical direction in such a way that it is bounded together with its partial derivatives and the first component of $\bar b(x)$ is $-p$ in the region $\{x:|p|\leq r_2\}$.
    We denote $\zeta_t^\e=(\bx_t^\e)_1$ and show that it takes significantly longer than $|\log\e|$ for $\zeta_t^\e$ to reach $\pm r_2$, hence it is unlikely for $\bx_t^\e$ to exit $D_{r_2}$ through ${(\partial D_{r_2})}_{\mathrm{in}}$. All the computations below concern $\bx_t^\e$ before $\zeta_t^\e$ reaches $\pm r_2$. As in \eqref{eq:slowx}, we can write the equation for $\zeta_t^\e$:
    \begin{equation}
    \begin{aligned}
    \label{eq:zetat}
        {\zeta}_t^\e&=p-\int_0^t {\zeta}_s^\e ds+\sqrt{\e}\int_0^t\nabla_y u_1(\bx_s^\e,\bxi_s^\e)^{\mathsf T}\sigma(\bxi_s^\e)dW_s\\
        &\quad+\e\int_0^t\nabla_x u_1(\bx_s^\e,\bxi_s^\e)\cdot b(\bx_s^\e,\bxi_s^\e)ds+\e(u_1(x,y)-u_1(\bx_t^\e,\bxi_t^\e)).
    \end{aligned}
    \end{equation}
    Introduce $\hat\zeta_t^\e$, which is close to $\zeta_t^\e$:
    \begin{equation}
        \label{eq:zeta}
        \hat\zeta_t^\e = p-\int_0^t {\zeta}_s^\e ds+\sqrt{\e}\int_0^t\nabla_y u_1(\bx_s^\e,\bxi_s^\e)^{\mathsf T}\sigma(\bxi_s^\e)dW_s+\e\int_0^t\nabla_x u_1(\bx_s^\e,\bxi_s^\e)\cdot b(\bx_s^\e,\bxi_s^\e)ds.
    \end{equation}
    Since $b(x,y)$ and its partial derivatives are bounded in $\{x:|p|\leq r_2\}\times\mathbb T^m$, we can choose $C>0$ such that
    \begin{equation}
    \label{eqb:supremums}
        \sup_{x:|p|\leq r_2,y\in\mathbb T^m}\left(|\nabla_y u_1(x,y)^{\mathsf T}\sigma(y)|^2\vee2|u_1(x,y)|\vee|\nabla_x u_1(x,y)\cdot b(x,y)|\right)<C/2.
    \end{equation}
    Let us define $\bar r_2=\frac{r_1+r_2}{2}$, $\hat{\tau}_2=\inf\{t:|\hat\zeta_t^\e|>\bar r_2\}$, and function $f(p)=\exp(p^2/(C\e))$. Then it follows that
    \begin{equation}
    \label{eq:ode_exit_loc}
        \frac{C\e}{2}f''-pf'-f=0.
    \end{equation}
    Note that $|\zeta_t^\e|\leq r_2$ for $0\leq t\leq\hat\tau_2$ since $|\zeta_t^\e-\hat\zeta_t^\e|\leq C\e/2$. Apply Ito's formula to $\exp(-t/2)f(\hat\zeta_t^\e)$ for $0\leq t\leq\hat\tau_2$ and obtain using \eqref{eqb:supremums}:
    \begin{align*}
        e^{-t/2}f(\hat\zeta_t^\e)&=f(p)-\frac{1}{2}\int_0^t e^{-s/2}f(\hat\zeta_s^\e)ds-\int_0^t e^{-s/2}f'(\hat\zeta_s^\e)\zeta_s^\e ds\\
        &\quad+\e\int_0^t e^{-s/2} f'(\hat\zeta_s^\e)\nabla_x u_1(\bx_s^\e,\bxi_s^\e)\cdot b(\bx_s^\e,\bxi_s^\e)ds\\
        &\quad+\frac{\e}{2}\int_0^t e^{-s/2}f''(\hat\zeta_s^\e)|\nabla_y u_1(\bx_s^\e,\bxi_s^\e)^{\mathsf T}\sigma(\bxi_s^\e)|^2ds\\
        &\quad+\sqrt{\e}\int_0^t e^{-s/2}f'(\hat\zeta_s^\e)\nabla_y u_1(\bx_s^\e,\bxi_s^\e)^{\mathsf T}\sigma(\bxi_s^\e)dW_s\\
        &= f(p)-\frac{1}{2}\int_0^t e^{-s/2}f(\hat\zeta_s^\e)ds\\
        &\quad+\int_0^t e^{-s/2}f'(\hat\zeta_s^\e)\left(-\frac{1}{2}\hat\zeta_s^\e+\left[(\hat\zeta_s^\e-\zeta_s^\e)+\e \nabla_x u_1(\bx_s^\e,\bxi_s^\e)\cdot b(\bx_s^\e,\bxi_s^\e)-\frac{1}{2}\hat\zeta_s^\e\right]\right) ds\\
        &\quad+\frac{\e}{2}\int_0^t e^{-s/2}f''(\hat\zeta_s^\e)|\nabla_y u_1(\bx_s^\e,\bxi_s^\e)^{\mathsf T}\sigma(\bxi_s^\e)|^2ds\\
        &\quad+\sqrt{\e}\int_0^t e^{-s/2}f'(\hat\zeta_s^\e)\nabla_y u_1(\bx_s^\e,\bxi_s^\e)^{\mathsf T}\sigma(\bxi_s^\e)dW_s\\
        &\leq f(p)-\frac{1}{2}\int_0^t e^{-s/2}f(\hat\zeta_s^\e)ds-\frac{1}{2}\int_0^t e^{-s/2}f'(\hat\zeta_s^\e)\hat\zeta_s^\e ds+\frac{C\e}{4}\int_0^t e^{-s/2}f''(\hat\zeta_s^\e)ds\\
        &\quad+\int_0^t e^{-s/2}f'(\hat\zeta_s^\e)\left[(\hat\zeta_s^\e-\zeta_s^\e)+\e \nabla_x u_1(\bx_s^\e,\bxi_s^\e)\cdot b(\bx_s^\e,\bxi_s^\e)-\frac{1}{2}\hat\zeta_s^\e\right] ds\\
        &\quad+\sqrt{\e}\int_0^t e^{-s/2}f'(\hat\zeta_s^\e)\nabla_y u_1(\bx_s^\e,\bxi_s^\e)^{\mathsf T}\sigma(\bxi_s^\e)dW_s\\
        &\leq  f(p)+18C\e(1-e^{-t/2})+\sqrt{\e}\int_0^t e^{-s/2}f'(\hat\zeta_s^\e)\nabla_y u_1(\bx_s^\e,\bxi_s^\e)^{\mathsf T}\sigma(\bxi_s^\e)dW_s.
    \end{align*}
    The last inequality follows from \eqref{eq:ode_exit_loc} and the fact that the integrand on the second line is either negative, when $|\hat\zeta_s^\e|\geq 2C\e$, or small and bounded by $9C\e e^{-s/2}$, when $|\hat\zeta_s^\e|\leq 2C\e$. By replacing $t$ by the stopping time $\hat\tau_2$ and taking expectation, we obtain
    \begin{equation}
        \E_{(x,y)} e^{-\hat\tau_2/2}\leq 2 e^{(r_1^2-{\bar r_2}^2)/(C\e)}.
    \end{equation}
    Let $\bar{\tau}_2=\inf\{t:|\zeta_t^\e|>r_2\}$. Then, since $|\zeta_t^\e-\hat\zeta_t^\e|\leq C\e/2$, it follows that
    \begin{equation}
        \Prob_{(x,y)}(\bar{\tau}_2<|\log\e|/\sqrt{\e})\leq \Prob_{(x,y)}(\hat{\tau}_2<|\log\e|/\sqrt{\e})\leq 2\exp(\frac{r_1^2-{\bar r_2}^2}{C\e}+\frac{|\log\e|}{2\sqrt{\e}})\to 0,
    \end{equation}
    as $\e\downarrow0$. However, by Lemma~\ref{lem:near} and Markov's inequality, we have
    \begin{equation}
        \Prob_{(x,y)}(\bar\tau>|\log\e|/\sqrt{\e})\to0,
    \end{equation}
    as $\e\downarrow0$. Thus, the desired result follows.
\end{proof}
\begin{lemma}
    \label{lem:away}
    Let $\bar{\bar\tau}=\inf\{t:\bx_t^\e\in D_{r_1}\}\wedge\bm\tau$. Then there exists $a>0$ such that, uniformly in $(x,y)\in(\partial {D_{r_2}})_{\mathrm{out}}\times\mathbb T^m$,
    \begin{equation}
    \label{eq:away_non_deg}
        \Prob_{(x,y)}(\bar{\bar\tau}<\bm\tau,\int_0^{\bar{\bar\tau}}|\nabla_y u_h(\bx_s^\e,\bxi_s^\e)^{\mathsf T}\sigma(\bxi_s^\e)|^2ds<a)\to 0~\text{as }\e\to0.
    \end{equation}
    Furthermore, $\E_{(x,y)}\bar{\bar\tau}$ is bounded uniformly in $(x,y)\in(\partial {D_{r_2}})_{\mathrm{out}}\times\mathbb T^m$.
\end{lemma}
\begin{proof}
Let $\hat t>0$ and $\check t>0$ be the lower bound and the upper bound of time spent by $\bm x_t$ to get from $(\partial {D_{r_2}})_{\mathrm{out}}$ to $D_{r_1}$, respectively. Then, similarly to \eqref{eq5:positive_variance_one_rotation}, there exists $a>0$ such that
\begin{equation}
    \Prob_{(x,y)}\left(\int_0^{\hat t/2} |\nabla_y u_h(\bx_s^\e,\bxi_s^\e)^{\mathsf T}\sigma(\bxi_s^\e))|^2ds>a,\sup_{0\leq t\leq2\hat t}|\bx_t^\e-\bm x_t|\leq\e^{\frac{1+2\alpha}{4}}\right)\to1.
\end{equation}
Hence
\begin{align*}
    &\Prob_{(x,y)}\left(\bar{\bar\tau}<\bm\tau,\int_0^{\bar{\bar\tau}}|\nabla_y u_h(\bx_s^\e,\bxi_s^\e)^{\mathsf T}\sigma(\bxi_s^\e)|^2ds<a\right)\\
    &\leq\Prob_{(x,y)}\left(\hat t/2\leq\bar{\bar\tau}<\bm\tau,\int_0^{\bar{\bar\tau}}|\nabla_y u_h(\bx_s^\e,\bxi_s^\e)^{\mathsf T}\sigma(\bxi_s^\e)|^2ds<a\right)+\Prob_{(x,y)}(\bar{\bar\tau}<\bm\tau,\bar{\bar\tau}<\hat t/2)\\
    &\leq\Prob_{(x,y)}\left(\int_0^{\hat t/2}|\nabla_y u_h(\bx_s^\e,\bxi_s^\e)^{\mathsf T}\sigma(\bxi_s^\e)|^2ds<a\right)+\Prob_{(x,y)}\left(\sup_{0\leq t\leq2\hat t}|\bx_t^\e-\bm x_t|>\e^{\frac{1+2\alpha}{4}}\right)\\
    &\to0.
\end{align*}
Similarly, it is easy to see that $\Prob_{(x,y)}(\bar{\bar\tau}>2\check t)<\Prob_{(x,y)}\left(\sup_{0\leq t\leq2\check t}|\bx_t^\e-\bm x_t|>\e^{\frac{1+2\alpha}{4}}\right)\to0$, and the desired result follows from the Markov property.
\end{proof}

\begin{proof}[Proof of Proposition~\ref{prop:exit_time_from_separatrix}]
    As in~\eqref{eq:H}:
    \begin{align*}
        H(\bx_t^\e) &= H(x)+\e\int_0^t \nabla_x u_h(\bx_s^\e,\bxi_s^\e)\cdot b(\bx_s^\e,\bxi_s^\e)ds\\
    &\quad+\sqrt{\e}\int_0^t \nabla_y u_h(\bx_s^\e,\bxi_s^\e)^{\mathsf T}\sigma(\bxi_s^\e)dW_s+\e(u_h(x,y)-u_h(\bx_t^\e,\bxi_t^\e)).
    \end{align*}
    The change in $H(\bx_t^\e)$ is mainly due to the stochastic integral while the other terms are of order $O(\e)$ and $O(t\cdot\e)$ and can be controlled. For each $t(\e)>0$,
    \begin{equation}
    \label{eq:set}
    \begin{aligned}
        \{\bm\tau<t(\e)\}\supset&\left\{\sup_{[0,t(\e)]}\left|\sqrt{\e}\int_0^{t(\e)} \nabla_y u_h(\bx_s^\e,\bxi_s^\e)^{\mathsf T}\sigma(\bxi_s^\e)dW_s\right|>3\e^\alpha\right\}\\
        &\quad\bigcap\left\{\e\int_0^{t(\e)} |\nabla_x u_h(\bx_s^\e,\bxi_s^\e)\cdot b(\bx_s^\e,\bxi_s^\e)|ds<\e^\alpha/2\right\}.
    \end{aligned}
    \end{equation}
    Note that if we choose $t(\e)=o(\e^{\alpha-1})$, then the second event is always true. Now we recursively define stopping times:
    \begin{align*}
        \theta^1_0&=0,\\
        \theta^2_k&=\inf\{t\geq\theta^1_{k-1}:\bx_t^\e\in\partial D_{r_2}\}\wedge\bm\tau,\\
        \theta^1_k&=\inf\{t\geq\theta^2_{k}:\bx_t^\e\in\partial D_{r_1}\}\wedge\bm\tau.
    \end{align*}
    We denote $D(x,y)=\nabla_y u_h(x,y)^{\mathsf T}\sigma(y)$. Note that once the process leaves $V^\e$, the stopping times stay constant afterwards. The main idea of the proof is to show that after a sufficiently long time $t(\e)$, the stochastic integral will accumulate enough variance to exit from $V^\e$. Let us bound the probability of variance being small:
    \begin{equation}
    \label{eq:top_inequality}
        \begin{aligned}
            &\Prob_{(x,y)}\left(\bm\tau\geq t(\e),\int_0^{t(\e)}|D(\bx_s^\e,\bxi_s^\e)|^2ds<9\e^{2\alpha-1}\right)\\
            &\quad\leq\Prob_{(x,y)}\left(\bm\tau\geq t(\e)>\theta^1_{n(\e)},\int_0^{t(\e)}|D(\bx_s^\e,\bxi_s^\e)|^2ds<9\e^{2\alpha-1}\right)+\Prob_{(x,y)}(\theta^1_{n(\e)}\geq t(\e)),
        \end{aligned}
    \end{equation}
    where the integer $n(\e)$ will be specified later.
    Let $\bar\tau$, $\bar{\bar\tau}$, and $a$ be defined as in Lemma~\ref{lem:near} and Lemma~\ref{lem:away}. Then
    \begin{align}
        &\Prob_{(x,y)}\left(\bm\tau\geq t(\e)>\theta^1_{n(\e)},\int_0^{t(\e)}|D(\bx_s^\e,\bxi_s^\e)|^2ds<9\e^{2\alpha-1}\right)\nonumber\\
        &\quad \leq \exp(9\e^{2\alpha-1}/a)\E_{(x,y)}\left(\chi_{\{\bm\tau>\theta^1_{n(\e)}\}}\exp\left(-\frac{1}{a}\int_0^{\theta^1_{n(\e)}}|D(\bx_s^\e,\bxi_s^\e)|^2ds\right)\right)\label{eq:rotations}\\
        &\quad\leq\exp(9\e^{2\alpha-1}/a)\left[\sup_{(x,y)\in\partial D_{r_1}\times\mathbb T^m}\E_{(x,y)}\left(\chi_{\{\bm\tau>\theta^1_1\}}\exp\left(-\frac{1}{a}\int_0^{\theta^1_1}|D(\bx_s^\e,\bxi_s^\e)|^2ds\right)\right)\right]^{{n(\e)}-1}.\nonumber
    \end{align}
    Now let us deal with one excursion from $D_{r_2}$ to $D_{r_1}$. For $(x,y)\in(\partial {D_{r_2}})_{\mathrm{out}}\times\mathbb T^m$,
    \begin{equation}
    \label{eqb:e0.99}
    \begin{aligned}
        &\E_{(x,y)}\left(\chi_{\{\bar{\bar\tau}<\bm\tau\}}\exp\left(-\frac{1}{a}\int_0^{\bar{\bar\tau}}|D(\bx_s^\e,\bxi_s^\e)|^2ds\right)\right)\\
        &\quad\leq \Prob_{(x,y)}(\bar{\bar\tau}<\bm\tau,\int_0^{\bar{\bar\tau}}|D(\bx_s^\e,\bxi_s^\e)|^2ds<a)+\Prob_{(x,y)}(\bar{\bar\tau}<\bm\tau,\int_0^{\bar{\bar\tau}}|D(\bx_s^\e,\bxi_s^\e)|^2ds\geq a)/e\\
        &\quad\leq e^{-0.99},
    \end{aligned}
    \end{equation}
    for all $\e$ sufficiently small, by Lemma~\ref{lem:away}. For $(x,y)\in\partial D_{r_1}\times\mathbb T^m$:
    \begin{align*}
        &\E_{(x,y)}\left(\chi_{\{\bm\tau>\theta^1_1\}}\exp\left(-\frac{1}{a}\int_0^{\theta^1_1}|D(\bx_s^\e,\bxi_s^\e)|^2ds\right)\right)\\
        &\leq \E_{(x,y)}\left(\chi_{\{\bm\tau>\theta^1_1,\bx^\e_{\bar\tau}\in(\partial {D_{r_2}})_{\mathrm{out}}\}}\exp\left(-\frac{1}{a}\int_0^{\theta^1_1}|D(\bx_s^\e,\bxi_s^\e)|^2ds\right)\right)+\Prob_{(x,y)}(\bx^\e_{\bar\tau}\in(\partial {D_{r_2}})_{\mathrm{in}})\\
        &\leq \sup_{(x',y')\in(\partial {D_{r_2}})_{\mathrm{out}}\times\mathbb T^m}\E_{(x',y')}\left(\chi_{\{\bar{\bar\tau}<\bm\tau\}}\exp\left(-\frac{1}{a}\int_0^{\bar{\bar\tau}}|D(\bx_s^\e,\bxi_s^\e)|^2ds\right)\right)+\Prob_{{(x,y)}}(\bx^\e_{\bar\tau}\in(\partial {D_{r_2}})_{\mathrm{in}})\\
        &\leq e^{-0.98},
    \end{align*}
    by Lemma~\ref{lem:exit_location} and \eqref{eqb:e0.99}.
    Now we can come back to \eqref{eq:rotations} and have
    \begin{equation}
    \label{eq:first_prob}
        \Prob_{(x,y)}\left(\bm\tau\geq t(\e)>\theta^1_n,\int_0^{t(\e)}|D(\bx_s^\e,\bxi_s^\e)|^2ds<9\e^{2\alpha-1}\right)\leq\exp(9\e^{2\alpha-1}/a-0.98({n(\e)}-1)).
    \end{equation}
    The second probability on the {right-hand side of} \eqref{eq:top_inequality} can be bounded by Lemmas~\ref{lem:near} and \ref{lem:away} with certain $K>0$:
    \begin{equation}
    \label{eq:second_prob}
        \begin{aligned}
            \Prob_{(x,y)}(\theta^1_n\geq t(\e))&\leq{\E_{(x,y)}\theta^1_n}/{t(\e)}\\
            &\leq{\left(\sup_{(x',y')\in\partial D_{r_1}\times\mathbb T^m}\E_{(x',y')}\bar\tau+\sup_{{(x',y')\in\partial D_{r_2}\times\mathbb T^m}}\E_{(x',y')}\bar{\bar\tau}\right)}\cdot\frac{n(\e)}{t(\e)}\\
            &\leq\frac{n(\e)K|\log\e|}{t(\e)}.
        \end{aligned}
    \end{equation}
    Choose $n(\e)=[10\e^{2\alpha-1}/a+2]$. Then the quantity in \eqref{eq:first_prob} converges to $0$. Choose $t(\e)=100K\e^{2\alpha-1}|\log\e|/a$, then the quantity on the {right-hand side} of \eqref{eq:second_prob} converges to $0.1$. Therefore, the quantity on the {right-hand side} of \eqref{eq:top_inequality} converges to $0.1$. Moreover, since $t(\e)=o(\e^{\alpha-1})$, it follows from \eqref{eq:set} that, for all $x\in V^\e$, $y\in\mathbb T^m$, and $\e$ sufficiently small,
    \begin{align*}
        &\Prob_{(x,y)}(\bm\tau\geq t(\e))\\
        &=\Prob_{(x,y)}\left(\bm\tau\geq t(\e),\sup_{[0,t(\e)]}\left|\sqrt{\e}\int_0^t D(\bx_s^\e,\bxi_s^\e)dW_s\right|\leq 3\e^\alpha\right)\\
        &\leq \Prob_{(x,y)}\left(\bm\tau\geq t(\e),\int_0^{t(\e)}|D(\bx_s^\e,\bxi_s^\e)|^2ds>9\e^{2\alpha-1},\sup_{[0,t(\e)]}\left|\sqrt{\e}\int_0^t D(\bx_s^\e,\bxi_s^\e)dW_s\right|\leq 3\e^\alpha\right)\\
        &\quad+ \Prob_{(x,y)}\left(\bm\tau\geq t(\e),\int_0^{t(\e)}|D(\bx_s^\e,\bxi_s^\e)|^2ds<9\e^{2\alpha-1}\right)\\
        &\leq 0.69+0.11=0.8,
    \end{align*}
    since the stochastic integral in the last inequality can be represented as time-changed Brownian motion. Finally, we have by the Markov property
    \[
    \E_{(x,y)}\bm\tau\leq 5t(\e)=O(\e^{2\alpha-1}|\log\e|).\qedhere
    \]
\end{proof}

\section{Tightness}
\label{sec:tightness}
In this section, we verify the tightness for the projection to the Reeb graph of the original process \eqref{eq:rescaled_process1} on $\mathbb R^2\times\mathbb T^m$.
By the Arzelà–Ascoli theorem {(cf. Theorem 7.3 in \cite{MR1700749})}, it suffices to check the following two conditions 
{\begin{align}
    &(\romannumeral1)\lim_{R\to+\infty}\limsup_{\e\downarrow0}\Prob_{(x,y)}\left(\sup_{0\leq t\leq T}|H(X_t^\e)|\geq R\right)=0,\label{eqc:bounded}\\
    &(\romannumeral2)\lim_{\delta\downarrow0}\limsup_{\e\downarrow0}\Prob_{(x,y)}\left(\sup_{\substack{0\leq s<t\leq T\\ |s-t|<\delta}}r(h(X_t^\e),h(X_s^\e))>\kappa\right)=0,\label{eqc:equi-cont}
\end{align}
hold for all $\kappa>0$ and $(x,y)\in\mathbb R^2\times\mathbb T^m$.}
As in \eqref{eq:H}, we can also write the equation for $H(X_t^\e)$, where we consider $(X_t^\e,\xi_t^\e)$ to be a process on $\mathbb R^2\times\mathbb T^m$:
\begin{equation}
\label{eqc:h}
    \begin{aligned}
    H(X_t^\e)&=H(x)+\int_0^t \nabla_y u_h(X_s^\e,\xi_s^\e)^{\mathsf T}\sigma(\xi_s^\e)dW_s+\e(u_h(x_0,y_0)-u_h(X_t^\e,\xi_t^\e))\\
    &\quad+\int_0^t \nabla_x u_h(X_s^\e,\xi_s^\e)\cdot b(X_s^\e,\xi_s^\e)ds.
\end{aligned}
\end{equation}
By assumption \hyperlink{H4}{\textit{(H4)}} and Lemma~\ref{lem:solution}, $u(x,y)$ is bounded together with its first derivatives.
Besides, by assumption \hyperlink{H2}{\textit{(H2)}}, $H(x)/|x|\to+\infty$ as $|x|\to+\infty$, hence $\sup\{|x|:H(x)\leq R\}/R\to0$ as $R\to+\infty$. 
Also, there exists an $K>0$ such that  and $H(x)>-K$ for all $x\in\mathbb R^2$.
For $R>K$, {let $\tau_R=\inf\{t:|H(X_t^\e)|=R\}\wedge T=\inf\{t:H(X_t^\e)=R\}\wedge T$.} Then, by Markov's inequality and boundedness of second derivatives of $H$,
\begin{align*}
    \Prob_{(x,y)}(\sup_{0\leq t\leq T}|H(X_t^\e)|\geq R)
    &=\Prob_{(x,y)}(H(X^\e_{\tau_R})=R)\\
    &\leq \E_{(x,y)}(H(X^\e_{\tau_R})+K)/R\\
    &\lesssim ({H(x)+(\e+T)\sup\{|\nabla H(x)|:H(x)\leq R\}+T+K})/{R}\\
    &\lesssim ({H(x)+(\e+T)\sup\{|x|:H(x)\leq R\}+T+K})/{R}\to0,
\end{align*}
as $R\to+\infty$, uniformly in $0<\e<1$.
Thus, we have \eqref{eqc:bounded}, and it also follows that
{\begin{equation}
\label{eqc:bounded_gradient}
    \lim_{R\to+\infty}\limsup_{\e\downarrow0}\Prob_{(x,y)}(\sup_{0\leq t\leq T}|\nabla H(X_t^\e)|\geq R)=0,
\end{equation}}since $H(x)/|x|\to+\infty$ and $H$ has bounded second derivatives.
To verify \eqref{eqc:equi-cont}, we see that, for an arbitrary $\kappa>0$ small,
\begin{align*}
    \Prob_{(x,y)}\left(\sup_{\substack{0\leq s<t\leq T\\ |s-t|<\delta}}r(h(X_t^\e),h(X_s^\e))>\kappa\right)&\leq \sum_{k=0}^{[T/\delta]}\Prob_{(x,y)}\left(\sup_{t\leq \delta}r(h(X_{k\delta+t}^\e),h(X_{k\delta}^\e))>\kappa/4\right)\\
    &\leq \sum_{k=0}^{[T/\delta]}\Prob_{(x,y)}\left(\sup_{t\leq \delta}|H(X_{k\delta+t}^\e)-H(X_{k\delta}^\e)|>\kappa/12\right).
\end{align*}
Since \eqref{eqc:bounded_gradient} holds, it is sufficient to prove, for each $R>0$,
{\begin{equation}
\label{eqc:b}
    \lim_{\delta\downarrow0}\limsup_{\e\downarrow0}\sum_{k=0}^{[T/\delta]}\Prob_{(x,y)}\left(\sup_{t\leq \delta}|H(X_{k\delta+t}^\e)-H(X_{k\delta}^\e)|>\kappa/12,\sup_{0\leq t\leq T}|\nabla H(X_t^\e)|\leq R\right)=0.
\end{equation}}Let $K'$ be the upper bound for $|b(x,y)-\nabla^\perp H(x)|$, $|u(x,y)|$, $|\nabla_x u(x,y)|$, and eigenvalues of $\nabla^2 H(x)$ and $(\nabla_y u\sigma\sigma^\mathsf{T}\nabla_y u^\mathsf{T})(x,y)$ over $\mathbb R^2\times\mathbb T^m$. 
On the event $\{\sup_{0\leq t\leq T}|\nabla H(X_t^\e)|\leq R\}$ the following holds. For $\e<\frac{\kappa}{48K'R}$, {due to \eqref{eqc:h} and the fact that the stochastic integral can be represented as a time-changed Brownian motion}, we have
\begin{align*}
    \sup_{t\leq \delta}|H(X_{k\delta+t}^\e)-H(X_{k\delta}^\e)|&\leq\sup_{t\leq \delta}|\int_{k\delta}^{k\delta+t}\nabla_y u_h(X_s^\e,\xi_s^\e)^{\mathsf T}\sigma(\xi_s^\e)dW_s|+2\e K'R+{K'}^2(K'+R)\delta\\
    &\leq\sup_{0\leq t\leq\delta K'R^2}|\tilde W_t|+{K'}^2(K'+R)\delta+\kappa/24,
\end{align*}
where $\tilde W_t$ is another Brownian motion. Hence, for $\delta$ sufficiently small, independently of $\e$,
\begin{align*}
    &\sum_{k=0}^{[T/\delta]}\Prob_{(x,y)}\left(\sup_{t\leq \delta}|H(X_{k\delta+t}^\e)-H(X_{k\delta}^\e)|>\kappa/12,\sup_{0\leq t\leq T}|\nabla H(X_t^\e)|\leq R\right)\\
    &\leq\left[\frac{T}{\delta}+1\right]\cdot\Prob_{(x,y)}\left(\sup_{0\leq t\leq\delta K'R^2}|\tilde W_t|>\kappa/48\right)\\
    &\to0,
\end{align*}
as $\delta\downarrow0$, since each term is exponentially small as $\delta\downarrow0$.
This proves \eqref{eqc:b}, and thus \eqref{eqc:equi-cont}.

\section*{Acknowledgements}
I am grateful to Mark Freidlin for introducing me to this problem and to my advisor Leonid Koralov for invaluable guidance. I am also grateful to Dmitry Dolgopyat and Yeor Hafouta for insightful discussions.
\printbibliography
\end{document}